\documentclass[12pt]{amsart}

\usepackage{amsmath,amssymb,amsthm}
\usepackage{amscd,times,fullpage}
\usepackage{color}

\usepackage{amscd,times}
\usepackage{amsthm,amsfonts,amsmath,amscd,amssymb,times,epsfig,verbatim}
\usepackage[all]{xy}
\usepackage{enumerate,array}

\newtheorem{thm}{Theorem}[section]
\newtheorem{prop}[thm]{Proposition}
\newtheorem{lem}[thm]{Lemma}
\newtheorem{cor}[thm]{Corollary}

\theoremstyle{remark}
\newtheorem{rem}[thm]{Remark}
\newtheorem{defn}[thm]{Definition}
\newtheorem{ex}[thm]{Example}

\newcommand{\C}{\mathbb{C}}

\newcommand{\OO}{\mathcal{ O}}

\newcommand{\R}{\mathbb{ R}}

\newcommand{\rk}{\operatorname{rk}}

\newcommand{\Z}{\mathbb{ Z}}

\DeclareMathOperator{\dummyO}{O}
\renewcommand{\O}{\dummyO}

\DeclareMathOperator{\rank}{rank}
\renewcommand{\Im}{\mathop{\mathrm{Im}}}

\title{Toric topology of the complex Grassmann manifolds}
\author{Victor M.~Buchstaber}
\address{Steklov Mathematical Institute of the Russian Academy
of Sciences, Moscow State University V.~M.~Lomonosov, Skolkovo Institute of Science and Technology, Moscow, Russia}

\email{buchstab@mi.ras.ru}
\author{Svjetlana Terzi\'c}
\address{Faculty of Science and Mathematics, University of Montenegro,
Podgorica, Montenegro}
\email{sterzic@ucg.ac.me}

\begin{document}


\maketitle
\begin{abstract}
 The family of the complex Grassmann  manifolds $G_{n,k}$ with the  canonical action of the  torus $T^n=\mathbb{T}^{n}$ and the analogue  of the moment map $\mu : G_{n,k}\to \Delta _{n,k}$ for the  hypersimplex $\Delta _{n,k}$, is well known.  In this paper we study the structure of the orbit space $G_{n,k}/T^n$ by  developing the methods of  toric geometry and toric  topology. We use a subdivision of $G_{n,k}$  into the strata $W_{\sigma}$. Relying  on this subdivision we determine all regular and singular points of the moment map $\mu$,  introduce the notion  of
the admissible polytopes $P_\sigma$ such that $\mu (W_{\sigma}) = \stackrel{\circ}{P_{\sigma}}$ and the notion of  the spaces  of parameters $F_{\sigma}$,  which together  describe   $W_{\sigma}/T^{n}$ as the product  $\stackrel{\circ}{P_{\sigma}} \times F_{\sigma}$. To find the  appropriate topology  for the set $\cup _{\sigma} \stackrel{\circ}{P_{\sigma}} \times F_{\sigma}$ we    introduce also the notions of the universal space  of parameters $\tilde{\mathcal{F}}$ and  the virtual spaces of parameters $\tilde{F}_{\sigma}\subset \tilde{\mathcal{F}}$ such that there exist the  projections $\tilde{F}_{\sigma}\to F_{\sigma}$. Having  this  in mind, we propose a method for the description of the orbit space $G_{n,k}/T^n$. The existence  of the action of the  symmetric group $S_{n}$ on $G_{n,k}$ simplifies the application of this method.  In our previous paper we proved that the orbit space 
$G_{4,2}/T^4$, which is defined by the canonical $T^4$-action  of complexity $1$,  is homeomorphic to $\partial \Delta _{4,2}\ast \C P^1$.    We prove in this paper that the orbit space $G_{5,2}/T^5$, which is defined by the canonical $T^5$-action of complexity $2$,  is homotopy equivalent to the space which is obtained by attaching the disc $D^8$ to the space $\Sigma ^{4}\R P^2$ by the generator of the group $\pi _{7}(\Sigma ^{4}\R P^2)=\Z _{4}$. In particular, $(G_{5,2}/G_{4,2})/T^5$ is homotopy equivalent to $\partial \Delta _{5,2}\ast \C P^2$.

The methods and the results of this  paper are very important for the  construction of the  theory of $(2l,q)$-manifolds we  have been recently developing,   and which is concerned with manifolds  $M^{2l}$  with an effective action of the torus $T^{q}$, $q\leq l$ and an analogue of the moment map $\mu : M^{2l}\to P^{q}$, where $P^{q}$ is a $q$-dimensional convex polytope.
\end{abstract}

\tableofcontents

\section{Introduction}
There is a canonical action of the   algebraic torus $(\C ^{*})^{n}$ on the algebraic manifold $G_{n,k}$ of all $k$-dimensional complex vector subspaces $L$ in the  $n$-dimensional complex vector space $\C ^{n}$ and, consequently,  an action of the compact torus $T^n = \mathbb{T}^{n}\subset (\C ^{*})^{n}$. The problems  of the description of these actions are  well  known and the interest in these problems  is motivated by the classic and modern problems of algebraic geometry, algebraic and equivariant topology, symplectic geometry and enumerative combinatorics, see~\cite{GFMC},~\cite{GS},~\cite{GEL4},~\cite{K},~\cite{KeelTev}. In the focus of our study is an action of the torus $T^n$ on $G_{n,k}$. The manifolds $G_{n,k}$ and $G_{n, n-k}$ are equivariantly diffeomorphic, so we can assume that $k\leq [\frac{n}{2}]$. There is   an analogue of the moment map~\cite{GS} for $G_{n,k}$ whose image is  the hypersimplex $\Delta _{n,k}$. In the case when $k=1$ we obtain  the complex projective space $\C P^{n-1}$, which is a fundamental example of a  toric manifold and $\Delta _{n,1}$ is a simplex. For $k\geq 2$ the combinatorics of the polytope $\Delta _{n,k}$ does not determine the structure of the orbit space $G_{n,k}/T^n$. Using the classical Pl\"ucker coordinates for $k$-dimensional subspace $L$  in $\C ^{n}$, we introduce $(\C ^{*})^{n}$-equivariant subdivision  of the manifold $G_{n,k}$ by the strata $W_{\sigma}$, where $\sigma$ runs through the set of the so-called admissible subsets of the set $\{1,\ldots ,m\}$, where $m={n\choose k}$.  The disjoint union of all $W_{\sigma}/T^n$  gives a subdivision of the orbit space $G_{n,k}/T^n$. We want to emphasize that the strata we introduce coincide with the known strata introduced,  although by different methods, in the paper of Gel'fand-Serganova~\cite{GS}. The image of the moment map, when restricted to the  stratum $W_{\sigma}$ is the interior of the admissible polytope $P_{\sigma}$, which is the convex hull of some subset of vertices for  $\Delta _{n,k}$.      We prove that all points from $W_{\sigma}$ have the same stabilizer $\C ^{*}_{\sigma}$ and describe  the singular and regular points for the moment map $\mu : G_{n,k}\to \Delta_{n,k}$. The orbit space $F_{\sigma} = W_{\sigma}/(\C ^{*})^{n}$ is an algebraic manifold which we call the space of parameters for $W_{\sigma}$. We consider the set  $E$ which is defined  as a disjoint  union of the  spaces $\stackrel{\circ}{P_{\sigma}}\times F_{\sigma}$, where $\sigma$ runs  through all admissible sets. The canonical projections $W_{\sigma}\to \stackrel{\circ}{P_{\sigma}}$ and $W_{\sigma}\to F_{\sigma}$ define the homeomorphism $W_{\sigma}/T^n\to \stackrel{\circ}{P_{\sigma}}\times F_{\sigma}$ and we obtain the canonical bijection of the sets  $E\to G_{n,k}/T^n$. 

In our approach the set $E$ is the key object   for  the description of the structure of the orbit space $G_{n,k}/T^n$ and the key problem here  is to introduce the topology on $E$ in terms of $P_{\sigma}$ and $F_{\sigma}$  such that the bijection  $ E \to G_{n,k}/T^n$ becomes a homeomorphism.  In the realization of this approach  we essentially had to  develop  the methods of  toric topology~\cite{BP1},~\cite{BP}.  In the case of a  quasitoric manifold  $M^{2n}$, the orbit space $M^{2n}/T^n$ is homeomorphic to the polytope $P^{n}$,  which is the image of the moment map.  Thus,  the orbit space $M^{2n}/T^n$ is determined by the combinatorial structure of the polytope $P^{n}$.  In the case of Grassmann manifolds, the points from  the orbit space $G_{n,k}/T^n$ are determined by the  two coordinates, one from  $P_{\sigma}$ and the other from $F_{\sigma}$.

In order to equip the set $E$ with an appropriate topology  we introduce the new notions:  the universal space of parameters $\tilde{\mathcal{F}}$, which is an algebraic variety and,  for any admissible set $\sigma$,  the virtual space of parameters $\tilde{F}_{\sigma}$,  such that there is the embedding $\tilde{F}_{\sigma}\subset \tilde{\mathcal{F}}$ and the projection $\tilde{F}_{\sigma}\to F_{\sigma}$. We also introduce the set $\mathcal{P}$ which is a formal union of the admissible polytopes $P_{\sigma}$ and define the topology on $\mathcal{P}$ using the fact  that the moment map $G_{n,k}/T^n\to \Delta _{n,k}$ decomposes as $G_{n,k}/T^n\to \mathcal{P}\to \Delta _{n,k}$.  The set $\mathcal{E}$ which is a union of all spaces $\stackrel{\circ}{P_{\sigma}}\times \tilde{F}_{\sigma}$ inherits the  topology from the  canonical embedding  into  the space  $\mathcal{P}\times \tilde{\mathcal{F}}$.  This yields a surjective   mapping  $H : \mathcal{E}\to E $,  which defines a topology on $E$. 

We want to note that  Kapranov in~\cite{Kap} considered  the Chow quotient $G_{n,k}/\!\!/ (\C ^{*})^n$ which is an algebraic variety  and proved that it can be identified with  certain compactification of the space of parameters $F$ of the main stratum. Moreover, Kapranov  proved  that the Chow quotient  $G_{n,k}/\!\!/ (\C ^{*})^n$  is isomorphic to the Chow quotient
$(\C P^{k-1})^{n}/\!\!/PGL(k)$.
In the case when $k=2$ and $n=5$, combining our results with the results of Kapranov~\cite{Kap} and Keel~\cite{Keel}, we obtain here  that the Chow quotient $(\C P^{1})^{5}/\!\!/PGL(2,\C)$ coincides with our     universal space of parameters $\tilde{\mathcal{F}}$. Inspired by  our results on description of $\tilde{\mathcal{F}}$ for $G_{5,2}$  presented in this paper, Klemyatin in~\cite{KM} proved  that  the universal space of parameters for $G_{n,2}$ coincides with the Chow quotient  
$(\C P^{1})^{n}/\!\!/PGL(2,\C)$.  This result is worthy of  attention, since  it implies that the universal space of parameters $\tilde{\mathcal{F}}$  of the  Grassmann manifolds $G_{n,2}$ coincides with the well known  moduli space   $\overline{M_{0,n}}$ of stable $n$-points curves of genus zero which is a smooth compact algebraic variety, see~\cite{Kap}. 

Our approach gives a  method to prove the following  result:

\begin{thm}
The orbit space $G_{n,2}/T^n$ is a quotient space of the space $\mathcal{E}$ by an equivalence relation defined by the map $H$.
\end{thm}

The proof is technically complicated and to overcome difficulties we essentially use the fact that the symmetric group $S_n$  acts on the space $G_{n,2}/T^n$   and that all our constructions are compatible with this action. In our previous paper~\cite{MMJ} we proved this result for $n=4$ and reprove it here  in Example~\ref{42}  . In this paper, we demonstrate  the proof with all the details for $n=5$. 

 In~\cite{MMJ}  it was also   proved that the space $G_{4,2}/T^4$ is homeomorphic  to $\partial \Delta _{4,2}\ast \C P^1$.

The main consequences of the results of this paper are  the following:

\begin{thm}
There exists a continuous map $(G_{5,2}/T^5)/(G_{4,2}/T^4)\to \partial \Delta _{5,2}\ast \C P^2$ which induces an isomorphism in the  integral homology groups and, thus, define a homotopy equivalence between these spaces.
\end{thm}

\begin{thm}\label{prvamain}
The orbit space $G_{5,2}/T^5$ is homotopy equivalent  to the  space  obtained by attaching the disc $D^8$  to the  space   $\Sigma ^{4}\R P^2$ by the  generator of the group $\pi _{7}(\Sigma ^{4} \R P^2)=\Z _{4}$.
\end{thm}

Note that the following  fact  has been used for  the proof of   Theorem~\ref{prvamain}.  Any of the  coordinate  embeddings $\C ^{n-1}\to \C ^{n}$ gives  the $T^n$-equivariant embedding $G_{n-1,k}\to G_{n,k}$ such that the complement $G_{n,k}\setminus G_{n-1,k}$ can be identified with the space of the canonical vector bundle $\eta _{n-1,n-k}$ over the Grassmann manifold $G_{n-1, k-1}$.  In this way we obtain the $T^n$-pair $ (G_{n,k}, G_{n-1, k})$  which defines  $T^n$-action on the  Thom space  $T(\eta _{n-1,n-k})$.   Furthermore,  the corresponding orbit space $T(\eta _{n-1, n-k})/T^n$ is homeomorphic to the  space obtained from   $G_{n,k}/T^n$  by collapsing $G_{n-1,k}/T^{n-1}$ into the point. In particular,
$\partial \Delta _{5,2}\ast \C P^{2} \cong T(\eta _{4,3})/T^5$.

We show that there exists a $T^n$-equivariant smooth closed submanifold $S_{n,k}$ in $G_{n,k}$ of codimension $1$ and $T^n$-equivariant projections $\pi _{0} : S_{n,k}\to G_{n-1,k}$ and $\pi _{1} : S_{n,k}\to G_{n-1, k-1}$ such that the following result holds.

 \begin{thm}
There exist  projections $\hat{\pi} _{0} : S_{n,k}/T^n \to G_{n-1, k}/T^{n-1}$ and $\hat{\pi} _{1} : S_{n,k}/T^n \to G_{n-1, k-1}/T^{n-1}$ such that 
\[
G_{n,k}/T^n = (S_{n,k}/T^n\times [0,1])\cup _{\hat{\pi} _{0}}G_{n-1,k}/T^{n-1}\cup _{\hat{\pi} _{1}}G_{n-1,k-1}/T^{n-1},
\]
where $[(s, 0)]\backsimeq \hat{\pi} _{0}[s]$ and $[(s,1)]\backsimeq \hat{\pi} _{1}[s]$.
\end{thm}

We want to emphasize that this paper is important for  our work on developing  the theory of $(2l, q)$-manifolds, ~\cite{OB},~\cite{MS}. These are $2l$-dimensional manifolds with an effective action of the compact $q$-dimensional torus and an analogue of the moment map $M^{2l}\to P^{q}$, where $P^{q}$ is a $q$-dimensional polytope. Our aim is to develop the methods of  toric topology which enabled us  to describe the action of the torus $T^{q}$ on $M^{2l}$ in terms of the combinatorics of  $P^{q}$ and the structure of the universal space of parameters $\tilde{\mathcal{F}}$ of the dimension $2(l-q)$. 

An important invariant of $(2l,q)$-manifolds is the complexity  $d(M^{2l}, q) = l-q$, which is equal to zero for quasitoric manifolds. 
This notion is used  in symplectic  geometry, where  the  Hamiltonian torus action of complexity $1$ on a symplectic manifold is studied~\cite{KT},~\cite{KT-1},  in algebraic geometry, where the complexity is defined as a codimension of the principal orbit for the action of algebraic torus~\cite{T1},~\cite{T2} and in equivariant topology~\cite{AZ},~\cite{AZ1},~\cite{SS}.

The manifolds $G_{n,k}$ represent a key family  of $(2l, q)$-manifolds, where $l=k(n-k)$ and $q=n-1$. The complexity  for $G_{n,k}$ is $(k-1)(n-k-1)$ and, in this context, the spaces $G_{4,2}$ and $G_{5,2}$ are of special  importance, since they provide  nontrivial examples of the  manifolds whose complexity is $1$ and $2$, respectively. 

In our theory of $(2l, q)$-manifolds there are also other two  important families. The first one is the family of the complex flag manifolds $Fl(n) = U(n)/T^n$ with the canonical action of the torus $T^n$  and an analogue of the moment map $\mu : Fl(n)\to {Pe}^{n-1}$, where ${Pe}^{n-1}$ is  a permutahedron. For these manifolds  $l={n\choose 2}$, $q=n-1$ and $d={n-1\choose 2}$. The second one consists of the manifolds $\C P^{N-1}$, $N={n\choose 2}$ with the action of the torus $T^n$  which is  given as a composition of the second exterior power representation $T^n\to T^{N}$ and the standard action of $T^{N}$ on $\C P^{N-1}$, and  an analogue  of the moment map $\mu : \C P^{N-1}\to \Delta _{n,2}$. For these manifolds $q=N-1$, $l=n-1$ and $d=N-n$. We  proved that for $(6,2)$-manifold $Fl(3)$  with $d=1$ there is a homeomorphism $Fl(3)/T^3\cong S^4$ and for $(10, 3)$-manifold $\C P^5$ with $d=2$  there is a homeomorphism $\C P^5/T^4\cong \partial \Delta _{4,2}\ast \C P^2$, ~\cite{MMJ}.

The canonical action of  the algebraic torus $(\C ^{*})^{n}$ on the manifold $G_{n,k}$ enables us  to assign  to each point from $G_{n,k}$ a toric variety, that is the closure of $(\C ^{*})^{n}$-orbit of this point. It naturally arises the  problem: classify, up to an algebraic equivalence, all such toric varieties. In~\cite{MMJ} this problem was solved for $G_{4,2}$, while in the current paper it is  solved for the manifold $G_{5,2}$. In~\cite{NO} in the frame of this problem, all  such smooth  toric subvarieties in the Grassmann manifod $G_{n,k}$ are described for any $n$ and $k$. The next natural step in this direction would be the question: classify all such  toric subvarieties in $G_{n,k}$ whose   singularities are only  the   fixed points. An answer to this question   for the case $G_{4,2 }$ and $G_{5,2}$ is given in~\cite{MMJ} and in current paper, respectively. 

\vspace*{0.5cm}

{\it Acknowledgment.} We are grateful  to  Anton Ayzenberg, Alexander Gaifullin, Nikita Klemyatin, Yael Karshon, Mikiya Masuda,   Taras Panov and  Hendrik S\" uss     for useful and fruitful discussions on the results of this paper.

\section{The canonical action of $T^n$ on $G_{n,k}$}

The complex  Grassmann manifolds $G_{n,k} =G_{n,k}(\mathbb{C})$ consist of all $k$-dimensional complex subspaces in $\mathbb{C}^{n}$.  The manifolds $G_{n,k}$ and $G_{n, n-k}$ are diffeomorphic.   For the standard scalar product  in $\C ^{n}$  there is the canonical   diffeomorphism  $c_{nk} : G_{n,k}\to G_{n, n-k}$  which sends any $k$-dimensional subspace of $\C ^{n}$ sends to its orthogonal complement.
The coordinate-wise action of the compact torus $T^n$ on $\mathbb{C}^n$ induces the canonical action
 of $T^{n}$  on $G_{n,k}$. The canonical diffeomorphism $c_{nk}$ is equivariant for this action.
One of  the well known and important problems is the  description of  the combinatorial structure and  algebraic topology of the orbit space $G_{n,k}/T^n\cong G_{n,n-k}/T^n$.

\begin{ex}
The manifold $G_{n,1} = \mathbb{C}P^{n-1}$ is  a toric manifold and  its orbit space $G_{n,1}/T^n$ is homeomorphic to the standard simplex $\Delta ^{n-1}$    in $\mathbb{R}^{n}$.
\end{ex}

\subsection{Pl\"ucker coordinates}
Let us fix a basis in  $\mathbb{C}^{n}$ and  fix  a   basis  in   $L\in G_{n,k}$. In these base, $L$  can be be represented by the  $(n\times k)$ - matrix $A_{L}$. Denote by  $P^{I}(A_{L})$  a  $(k\times k)$ -  minor of $A_{L}$, which is  determined by the rows indexed by the elements of $I$,  where $I\subset \{1,\ldots , n\}$ such that    $|I| = k$. The  complex numbers $(P^{I}(A_{L}))$, where $I$ runs through all subsets  of $\{1,\ldots n\}$ such that $|I|=k$,  are known as  the Pl\"ucker coordinates of the subspace  $L\subset \mathbb{C}^{n}$. The Pl\"ucker coordinates $(P^{I}(L))$ are defined uniquely, up to common scalar,  and they give  
the   Pl\"ucker embedding $G_{n,k}\to \mathbb{C}P^{N-1}$, $N= {n\choose k}$ which is  defined by
\begin{equation}
L\rightarrow P(L) = (P^{I}(A_{L})),\;\;  I\subset  \{1,\ldots n\}, \; |I|=k
\end{equation}
The Pl\"ucker embedding provides $G_{n,k}$ with the  structure of a complex $k(n-k)$-dimensional algebraic projective  manifold.

\begin{ex}
The image of the Pl\"ucker embedding of  $G_{4,2}$ in  $\mathbb{C}P^5$ is the hypersurface
\[
z_{12}z_{34} + z_{14}z_{23} = z_{13}z_{24}.
\]
\end{ex}
\begin{ex}\label{Pl2}
The image of the Pl\"ucker embedding of  $G_{5,2}$ in  $\mathbb{C}P^9$  is the intersection  of the five    hypersurfaces:
\[
z_{12}z_{34} + z_{14}z_{23} = z_{13}z_{24}, \;\;
z_{12}z_{35} + z_{15}z_{23} = z_{13}z_{25},
\]
\[
z_{12}z_{45} + z_{15}z_{24} = z_{14}z_{25},\;\;
z_{13}z_{45} + z_{15}z_{34} = z_{14}z_{35},\;\;
z_{23}z_{45} + z_{25}z_{34} = z_{24}z_{35}.
\]
\end{ex}

\begin{ex}
The image of the Pl\"ucker embedding of  $G_{n,2}$, $n\geq 4$  in  $\mathbb{C}P^{N-1}$, $N = {n\choose 2}$   is the intersection of the ${n\choose 4}$ quadratic    hypersurfaces:
\[
z_{ij}z_{kl} + z_{jk}z_{il} = z_{ik}z_{jl}, \;\; 1\leq i< j < k <l\leq n.
\]
\end{ex}
\begin{rem}
Note that the normal bundle for $G_{n,2}$ in $\mathbb{C}P^{N-1}$ is a  complex vector bundle  of  the  dimension ${n-2\choose 2}$. Thus, for $n>4$  we have the  algebraic manifold $G_{n,2}\subset \mathbb{C}P^{N-1}$ without singularities, which is not a complete intersection.
\end{rem}

Let  $\rho _{n, k} : T^{n}\to \mathbb{T}^{N}$ , $N = {n\choose k}$ be a representation  given by the $k$-th exterior power
\[
(t_1,\ldots, t_n) \to (t_1\cdots t_k,\ldots , t_{n-k+1}\cdots t_n).
\]
Consider the action  of $T^{n}$ on $\C P^{N-1}$, $N= {n\choose k}$, which is given as a composition of the standard action of $T^{N}$ on $\C P^{N-1}$ and the representation $\rho _{n,k}$.  It directly follows:

\begin{lem}\label{Plembed}
The Pl\"ucker embedding
\[
T^{n}\;  \curvearrowright G_{n,k}\;  \stackrel{\rho _{n,k}}{\to}\;  \mathbb{C} P^{N-1}\;  \curvearrowleft \; T^{n}
\]
is equivariant for the canonical action of $T ^n$ on $G_{n,k}$ and the $k$-th exterior  power action of $T^{n}$ on $\C P^{N-1}$.
\end{lem}
The  standard moment map 
$\mu: \mathbb{C} P^{N-1} \to \Delta^{N-1}\subset \mathbb{R}^N$  for $T^{N}$-action on $\mathbb{C} P^{N}$  is given by
\[
\mu({\bf z}) = \frac{1}{|{\bf z}|^2}(|z_1|^2,\ldots , |z_{N}|^2),
\]
where ${\bf z} = (z_1:\ldots :z_{N})$, $|{\bf z}|^2 = \sum\limits_{i=1}^{N}|z_i|^2$ and $\Delta ^{N-1}$ is the standard simplex.
In this way it is defined the map 
\[
\mu \circ \rho _{n,k} : G_{n,k}\to \Delta ^{N-1}.
\]
\subsection{Moment map}
Let $\mathbb{R}^{n}$  be  $n$-dimensional real  vector space with a fixed basis. The weight vectors  of the representation $\rho _{n,k}$ are:
\[
 \Lambda _{I}\in \Z ^{n}\subset \mathbb{R}^n,\quad (\Lambda _{I})_{j} = 1\; \text{for}\;  j\in I,\quad (\Lambda _{I})_{j} =0 \; \text{for}\; j\notin I,
\]
where $I \subset \{1,\ldots ,n\}$, $|I|=k$. In other words,  $\Lambda_{I}$ has $1$ at $k$ places\, and\, it has $0$ at the other $(n-k)$ places.

The moment map $\mu : G_{n,k} \to \mathbb{R}^n$, see~\cite{K},~\cite{GS}  is defined  by
\begin{equation}\label{momentmap}
\mu (L) = \frac{1}{|P(L)|^2}\sum |P^{I}(A_{L})|^2\Lambda _{I},\;\;\;
 |P(L)|^2 =
 \sum |P^{I}(A_{L})|^2,
\end{equation}
where 
$\Lambda _{I}\in R^n$,\quad $(\Lambda _{I})_{j} = 1$ for $j\in I$,\quad $(\Lambda _{I})_{j} =0$ for $j\notin I$ and 
the sum goes  over all subsets $I \subset \{1,\ldots ,n\}$, $|I|=k$.

\subsection{Hypersimplex}
 The image of $\mu$ is the polytope which is obtained as the  convex hull of the vectors $\Lambda _{I}$. This  polytope  is known as the hypersimplex and it is denoted by  $\Delta _{n,k}$. More precisely,
\[
\Delta _{n,k} = I^n\cap \{(x_1,\ldots ,x_n)\in \R^{n},\; \sum\limits _{i=1}^{n}x_i =k\}.
\]
  It follows that  $\Delta _{n,k}$ belongs to the hyperplane $x_1+\cdots +x_n=k$ in $\mathbb{R}^{n}$ which  implies  that 
$\Delta _{n,k}$ is a  $(n-1)$-dimensional polytope.

\begin{ex}
$\Delta _{n,1}$ is the standard simplex  $\Delta ^{n-1}$, while $\Delta _{4,2}$ is the octahedron.
\end{ex}

 The hypersimplex $\Delta _{n,k}$ has ${n\choose k}$ vertices,  $2n$ facets and $k(n-k)$ edges meet at every vertex. The set of facets for $\Delta _{n,k}$  can be expressed as a  union of  $n$ copies of  $\Delta _{n-1, k}$ and $n$ copies of  $\Delta _{n-1, k-1}$, ~\cite{Ziegler}.  We shortly describe how these facets arise from the Grassmannian $G_{n,k}$.   

The   embeddings
$
i_{q} : \mathbb{C}^{n-1}\to \mathbb{C}^{n}$,  $ i_{q}(z_1,\ldots, z_{n-1}) = (z_1, \ldots z_{q-1}, 0, z_{q},\ldots, z_{n-1}),  1\leq q\leq n-1$ are equivariant for the coordinate-wise torus action and they   induce the embeddings
$\tilde{i_{g}} : G_{n-1,k}\to G_{n,k}$ and $\hat{i_{q}} : G_{n-1, k-1}\to G_{n,k}$ for any $n\in \mathbb{N}$, whose images we denote by $G_{n-1,k}(q)$ and $G_{n-1,k-1}(q)$, respectively. These embeddings are equivariant for the corresponding embedding of the standard tori which  implies that they are compatible with the moment map.  Therefore they induce the embeddings
\[
\Delta _{n,k-1}(q)\to \Delta _{n+1, k} \;\; \text{and}\;\; \Delta _{n,k}(q)\to \Delta _{n+1,k}, \;\; 1\leqslant q\leqslant n+1.
\]
 It is straightforward to see that    $\Delta _{n-1,k}(q)\subset I^{n}_{x_q=0}$,    $\Delta _{n-1,k-1}(q)\subset I^{n}_{x_q=1}$,  and
 $\Delta _{n-1,k}(q)$,  $\Delta _{n-1,k-1}(q)$  belong to the boundary of $\Delta _{n, k}$ for all  $1\leq q\leq n$.

\begin{rem}
We note that the  set of vertices for $\Delta _{n,k}$ decomposes into two subsets according to the formula ${n\choose k} = {n-1\choose k} + {n-1\choose k-1}$. There are $n$ such decompositions  given by the pairs $(\Delta _{n-1, k}(q), \Delta _{n-1,k-1}(q)),$  $1\leqslant q\leqslant n$. In this way we obtain the combinatorial structure on the sphere  $S^{n-2}$. The symmetric  group $S_{n}$ is a symmetry group for this structure. More precisely this combinatorial structure is given by  the orbit  of  the pair $\big(\Delta _{n-1, k}(1), \Delta _{n-1, k-1}(1)\big)$ by an  action of the group $S_{n}$ .
\end{rem}

\subsection{Admissible polytopes and strata}\label{aps}
Using Pl\"ucker coordinates it can be defined a smooth atlas on $G_{n,k}$. The charts are given by $M_{I} = \{L\in G_{n,k} : P^{I}(L)\neq 0\}$, $I\subset \{1,\ldots ,n\}$, $|I|=k$ and the homeomorphisms $u_{I} : M_{I} \to \C ^{k(n-k))}$ are constructed as follows. Any $L\in M_{I}$ can be uniquely represented by the $(n\times k)$-matrix $A_L$ whose submatrix, determined by the rows indexed by $I$ is an identity matrix.  In this way,  the matrix $A_{L}$ has $k(n-k)$-variables $a_{ij}(L)$ and the  homeomorphism $u_{I} : M_{I}\to \C ^{k(n-k)}$ is given by $u_{I}(L) = (a_{ij}(L))$, where $i\notin I$.  It can be easily seen  that any chart $M_{I}$ contains exactly one fixed point given by the element $L$ such that $u _{i}(L) = 0 \in \C ^{k(n-k)}$.
The number of charts is  $m ={n\choose k}$ and it coincides with the number  of  fixed points for the canonical $T^{n}$-action on $G_{n,k}$ .  The charts $M_{I}$ are open,  $T^n$ invariant and  dense sets  in $G_{n,k}$.  It follows  that the sets   $Y_{I} = G_{n,k}\setminus  M_{I}$ are closed and $T^n$-invariant.

Let us enumerate the charts as  $(M_{I_1},u_{I_1}),\ldots ,(M_{I_{m}},u _{I_m})$. Define the spaces $W_{\sigma}$, where $\sigma = \{I_1,\ldots ,I_l\}$ and   $I_i\subset \{1,\ldots ,n\}$ such that  $|I_i| = k$,  $1\leq i\leq l$ and $1\leq l\leq m$,   by 
\begin{equation}\label{st}
W_{\sigma} = M_{I_1}\cap \cdots M_{I_l}\cap Y_{I_{l+1}}\cap \cdots Y_{I_{m}},
\end{equation}
where $\{I_{l+1},\ldots ,I_{m}\} = \{I_1, \ldots , I_{m}\} \setminus \{I_1,\ldots ,I_l\}$.  More precisely,
\begin{equation}\label{str}
W_{\sigma} = \{ L \in G_{n,k} : P^{I}(L)\neq 0 \; \text{for}\; I\in \sigma\;  \text{and}\; P^{I}(L)=0 \; \text{for}\; I\notin \sigma\},
\end{equation}

\begin{defn}
A non-empty space $W_{\sigma}$  is said to be the  stratum . The  index set $\sigma$  of a stratum  $W_{\sigma}$ is   said to be an  admissible set.   
\end{defn}
\begin{defn}
The stratum $W_{\sigma}=M_{I_1}\cap\cdots \cap M_{I_{m}}$, where $\sigma =\{I_1, \ldots , I_m\}$  we call the main stratum and denote by $W$. 
\end{defn}
We denote by $\mathcal{A}$  the set of all admissible sets.
\begin{lem}
For $\sigma =\{I\}$, $W_{\sigma}$ is a fixed point.
\end{lem}
\begin{proof}
In this case,  any element $L\in W_{\sigma}$ can be represented by the  matrix $A$ such that $A_{I}=I_{d}$ and any  $k\times k$-minor of $A$ different from that one given by $A_{I}$ is zero. This  implies that all elements of $A$ apart from $A_{I}$ are zero,  which  implies  that $W_{\sigma}$ consists of one element that is  a fixed point.  
\end{proof}

\begin{lem}
The strata $W_{\sigma}$ for $G_{n,k}$ are disjoint subspaces which  are  $T^n$-invariant  and give the equivariant subdivision $G_{n,k} = \cup _{\sigma} W_{\sigma}$.
\end{lem}

From the definition of the moment map it follows  that $\mu (W_{\sigma}) = \stackrel{\circ}{P_{\sigma}}$, where   $P_{\sigma}$ is the convex hull of the points  $\Lambda _{I_1},\ldots ,\Lambda _{I_l}$ and  $\sigma =\{I_1, \ldots, I_l\}$. In particular for $l=1$ it holds  $\mu (W_{\sigma}) =\Lambda _{I}$. 

\begin{defn}
A polytope $P_{\sigma}$  in $\Delta _{n,k}$, whose interior is  the image  of the stratum $W_{\sigma}$ by the moment map that is $\mu (W_{\sigma}) = \stackrel{\circ}{P}_{\sigma}$   is called  the  admissible polytope.
\end{defn}

\subsection{$(\C ^{*})^{n}$-action on $G_{n,k}$}
The canonical action of $T^{n}$ on $G_{n,k}$ extends to a canonical action of the algebraic torus $(\C ^{*})^{n}$ on $G_{n,k}$. The strata we $W_{\sigma}$ are invariant for  the action of the algebraic torus $(\C ^{*})^{n}$ as well. 

\begin{defn}
The space $F_{\sigma} = W_{\sigma}/(\C ^{*})^{n}$ is called  the space of parameters of the stratum $W_{\sigma}$.
\end{defn}

The space of parameters of the main stratum $W$ is  denoted  by $F=F_{n,k}$.  We will omit the indices  $n,k$ when it is  clear from the context that the word is about the main stratum.

The moment map $\mu : G_{n,k}\to \Delta _{n,k}$ relates the $(\C ^{*})^{n}$-orbits on $G_{n,k}$ and some polytopes in $\Delta _{n,k}$. More precisely, the classical convexity theorem of~\cite{AT},~\cite{GUST} states:
\begin{thm}\label{moment-map-polytope}
Let  $\OO _{\C}(L)$ be an orbit of an element $L\in G_{n,k}$ under the canonical action of $(\C ^{*})^{n}$. Then
$\mu (\overline{\OO _{\C}(L)})$ is a convex polytope in $\R ^{n}$ whose vertex set is given by $\{ \Lambda _{I} | P^{I}(L)\neq 0 \}.$ The mapping $\mu$ gives a bijection between $p$-dimensional orbits of the group $(\C ^{*})^{n}$ in 
$\overline{\OO _{\C}(L)}$ and $p$-dimensional open faces of the polytope $\mu (\overline{\OO _{\C}(L)})$.
\end{thm}

The closure $\overline{\OO _{\C}(L)}$ is a toric variety for any $L\in G_{n,k}$. This toric variety  is a smooth manifold depending if the corresponding admissible polytope is a simple polytope or not. A natural question is to describe the singularities of these toric varieties.  It is done in~\cite{MMJ} for the case of $G_{4,2}$,  while in this paper we do it for $G_{5,2}$ by Theorem~\ref{toric}.

\begin{rem}
A notion of the stratum  in  $G_{n,k}$ is defined  in~\cite{GS},~\cite{GEL4} by  different methods in three equivalent ways. 
By the one of  these  definitions  two  points $L_1,L_2\in G_{n,k}$ are said to belong to the same  stratum in $G_{n,k}$ if
\[
\mu (\overline{\OO _{\C}(L_1)}) = \mu (\overline{\OO _{\C}(L_2)}).
\]
In other words, by~\cite{GS} the stratum  in  $G_{n,k}$  consists of all $(\C ^{*})^{n}$-orbits which are mapped by the moment map  to the same  polytope in $\Delta _{n,k}$.   It is straightforward to verify  that our definition of strata is equivalent to this definition.
\end{rem}

\begin{rem}
We want to note that our approach to the  notion of  strata given by~\eqref{st} is more general since it does not use the existence of $(\C ^{*})^n$-action on $G_{n,k}$,  which extends $T^n$-action.  This approach is  fundamental  for developing the theory of $(2l,q)$-manifolds. 
\end{rem}

\subsection{The  space of parameters of the main stratum}
Let us consider the main stratum $W\subset G_{n,k}$ and a point $L\in W$.  The subspace $L$ can be described  by a matrix and this matrix can be considered as $n$- tuple of $k$-dimensional  vectors. Any  $k$ vectors from  this $n$-tuple define the Pl\"ucker coordinate, which is non-zero, and,  therefore, these vectors  are linearly independent.  In this way,  we obtain  the map $W\to (\C ^{k}\setminus {\bf 0})^{n}/GL(k, \C )$.  This map is invariant for $(\C ^{*})^{n}$-action and it  induces an  embedding 
\begin{equation}\label{embF}
F = W/(\C ^{*})^{n} \to (\C P^{k-1})^{n}/PGL(k, \C).
\end{equation}
Note that this map plays an  important role in~\cite{Kap}.

Let $k=2$ and $\mathcal{W}_{n,2} = \{((z_1:z_1^{'}),\ldots ,(z_n:z_{n}^{'}))\in  (\C P^{1})^{n}\; | \; (z_i:z_i^{'})\neq (z_j:z_{j}^{'})\}$.  
\begin{lem}\label{Wn}
The projective group $PGL(2, \C)$ acts freely on $\mathcal{W}_{n,2}$ and $F\cong \mathcal{W}_{n,2}/PGL(2,\C)$.
\end{lem}
\begin{proof}
The   matrices   which represent the elements from  the main stratum in $G_{n,2}$ are characterized  by the condition that all their $(2\times 2)$-minors are non-zero. This condition is  equivalent to the condition that the image of the map  $W\to (\C ^{2}\setminus {\bf 0})^{n}/GL(2,\C)$ consists of $n$-tuples of pairs, which are pairwise non-collinear.   It implies that $F\cong \mathcal{W}_{n,2}/PGL(2,\C)$.
\end{proof}
It is a  classical result that   $\mathcal{W}_{3,2}/PGL(2, \C)$ is a point. For  arbitrary $n$ we prove:

\begin{prop}\label{em}
The space of parameters $F$ is homeomorphic to   
\[
(\C P^{1})^{n-3}\setminus G_{n},\;\;  G_{n}= (\cup_{j=1}^{n} (A\times _{j} (\C P^1)^{n-1}))\cup (\cup _{1\leq i<j\leq n}\Delta _{i,j}(\C P^1)),
\]
where $A= \{(1:0), (0:1), (1:1)\}$, $A\times _{j} (\C P^1)^{n-1} = \C P^1\times \cdots \times \C P^1\times  \stackrel{j}{A}\times \C P^1\times \cdots \times\C P^{1}$   and $\Delta _{i,j}(\C P^1) \subset (\C P^{1})^{n-3}$ is given by the diagonal on $i$-th and $j$-th factor.
\end{prop}
\begin{proof}
It follows from Lemma~\ref{Wn} that $F\cong \mathcal{W}_{n,2}/PGL(2,\C)$.  Let $(z_1,\ldots, z_n)\in \mathcal{W}_{n,2}$. Since any three different points $z_1, z_2, z_3$  in  $\C P^1$ produce a     projective basis  for $\C P^1$, there is a unique  homography   defined by  $z_1\to (1:0)$,   $z_2\to (0:1)$ and $z_3\to (1:1)$. Then  the mapping $\mathcal{W}_{n,2}\to (\C P^{1})^{n-3}$ is given by the cross-ratio   $z_i\to [z_1,z_2,z_3,z_i]$ of the points $z_1,z_2,z_3, z_i$. It implies that $\mathcal{W}_{n,2}/PGL(2,\C) \cong \{(z_4,\ldots , z_{n})\in (\C P^{1})^{n-3}, \; z_i\notin A, \; z_i\neq z_j\}$, which  proves the statement.
\end{proof}
\begin{cor}
$F_{4,2}$ is homeomorphic to $\C P^{1}\setminus A$ and $F_{5,2}$ is homeomorphic to $(\C P^1\times \C P^1)\setminus G_5$, 
where $G_5 = (A\times \C P^1)\cup (\C P^1\times A)\cup \Delta_{1,2}(\C P^1)$.
\end{cor}

\begin{rem}
The space $F_{5,2}$ is a two-dimensional open algebraic manifold. In order to describe  the orbit space $G_{5,2}/T^5$ it is necessary to find a  compactification of the space $F_{5,2}$ that  corresponds  to  the compactification of the orbit space of the main stratum $W/T^5$ in $G_{5,2}/T^5$.
\end{rem} 
\begin{prop}\label{enm}
The non-point space of parameters $F_{\sigma}$  of a  stratum $W_{\sigma }$ in $G_{n,2}$,   which is different from the main stratum,  can be embedded in $(\C P^{1})^{l}$ for $1\leq l \leq n-4$.
\end{prop}
\begin{proof}
Since $W_{\sigma}$ is not the main stratum its points have at least one common  zero Pl\"ucker coordinate. It implies that at least two rows in the matrix representing a point of this stratum are collinear. Note that the condition that $F_{\sigma}$ is not a point  implies  that the matrices which represent  $W_{\sigma}$  must have at least four rows such that  any two of them  are non-collinear.   Therefore,  $W_{\sigma}/(\C ^{*})^{n}$  can be embedded  in $(\C P^{1})^{s}/PGL(2,\C )$, $4\leq s\leq n-1$,  which  implies that $F_{\sigma}$ can be embedded  in $(\C P^{1})^{s-3}$.
\end{proof}

\subsection{Action of the symmetric group $S_n$}
The symmetric group $S_n$ acts on  $\C ^{n}$ by permuting  coordinates. This action  induces the action of $S_n$ on $G_{n,k}$.

The  $S_n$-action on $G_{n,k}$ can be  interpreted as a  $S_n$-action on the corresponding $(n\times k)$- matrices $A_L$, which is given by the permutation of  the rows. It implies  that  $L_1,L_2\in G_{n,k}$ has the same non-zero Pl\"ucker coordinates if and only if  $s(L_1), s( L_2)$  has the same non-zero Pl\"ucker coordinates for any $s\in S_n$. In other words $L_1, L_2$ belong to the same stratum if and only if $s(L_1), s(L_2)$ belong to the same stratum, for any $s\in S_5$. Therefore the group $S_n$ acts on the set of all strata for $G_{n,k}$.   Since this action is equivariant relative to the action of the torus $(\C ^{*})^{n}$, we deduce:
\begin{lem}\label{Sn-action} 
The canonical action of  $S_n$ on $G_{n,k}$ induces an $S_n$- action  on the set of  admissible polytopes and the set of spaces of parameters of the strata by $s(P_{\sigma}) = \mu (s(W_{\sigma}))$ and $s( F_{\sigma}) = s( W_{\sigma})/(\C ^{*})^{n}$.
\end{lem}
\begin{lem}\label{Sn-action-charts}
The strata $W_{\sigma}$ and $s(W_{\sigma})$ belongs to  the same number of charts for any stratum $W_{\sigma}$ and any $s\in S_n$. Consequently, the admissible polytopes $P_{\sigma}$ and $s(P_{\sigma})$ have the same number of vertices.
\end{lem}
\begin{proof}
 The statement  directly follows from the fact that the strata $W_{\sigma}$ and $s(W_{\sigma})$ have the same number of non-zero Pl\"ucker coordinates and that the vertices of an admissible polytope are determined by the non-zero Pl\"ucker coordinates of the corresponding stratum.
\end{proof}
We elaborate this action in more detail.
Let $S[n]$ be the power set of $\{1,\ldots, n\}$. Then $S_n$ acts on  $S[n]$  and $|\sigma| = |s(\sigma)|$ for $s\in S_{n}$, $\sigma  \in S[n]$. It implies that $S_n$ acts on the set $W(p) = \{ W_{\sigma}, \; |\sigma| = p\}$ for any fixed  $p$, $1\leq p\leq N$. Put $m_{p} = |W(p)|$
and let $S_{\sigma}(p)$ be a subgroup of $S_n$, which is a stabilizer for  $W_{\sigma}$, where   $W_{\sigma}\in W(p)$.

If the orbit $S_{n}(W_{\sigma})$ coincides with  $W(p)$, then  $m_{p} = \frac{n!}{|S_{\sigma}(p)|}$ and there is one $S_{n}$-generator $W_{\sigma_{1_{p}}}$ for $W(p)$.
Otherwise,  $S_n$ acts on $W(p)\backslash S_{n}(W_{\sigma_{1_{p}}})$,  we repeat the argument and obtain the next generator $W_{\sigma_{2_{p}}}$. In this way we obtain  the set  of generators $W_{\sigma _{1_{p}}}, \ldots ,W_{\sigma _{q_{p}}}$.

\begin{defn}
The strata  $W_{\sigma _{i_{p}}}$, $1\leq i\leq q$   are said to be the fundamental strata for a  given $p$, where $ q=q_p$.
\end{defn}

From the description of the fundamental strata we obtain:
\[
m_{p}= \sum _{i=1}^{q}\frac{n!}{|S_{\sigma_{i_{p}}}(p)|},
\]

\begin{cor}
There exist  integer functions   $m_{p}, q_{p} : [1, N]\to \mathbb{Z}$, $N={n\choose k}$ which are invariants  for $T^{n}$-action on $G_{n,k}$.
\end{cor}

\begin{defn}
Admissible polytopes which correspond to the fundamental strata are called the  fundamental polytopes.
\end{defn}

Thus,    any admissible polytope  can be obtained by the action of the group $S_n$ on some of the fundamental polytopes.

\subsection{The $T^n$-action on  the Thom spaces}
We show that the canonical torus action on the complex Grassmann manifolds induces a torus action on the corresponding Thom spaces. Denote by $\xi_{n,k}$ and $\eta_{n,n-k}$ the canonical complex vector bundles  over $G_{n,k}$, where  $\dim _{\mathbb{C}}\,\xi_{n,k}=k$,\, $\dim _{\mathbb{C}}\,\eta_{n,n-k}=n-k$. Moreover,  it holds  $\xi_{n,k}\oplus \eta_{n,n-k}=n(1)$.

Put $E_{1,q}=G_{n,k}\backslash G_{n-1,k-1}(q)$  and $E_{2,q}=G_{n,k}\backslash G_{n-1,k}(q).$ An $T^{n}$-equivariant projection $\pi _{1,q} : E_{1,q}\to G_{n-1, k}$ is defined using the following observation. If  $L\in E_{1,q}$ then $L$ does not contain the coordinate vector $e_{q}$ which  implies that the retraction $r_{q} : \mathbb{C}^{n}\to \mathbb{C}^{n-1}_{z_q=0}$ maps  $L$ to a $k$-dimensional subspace in $\mathbb{C} ^{n-1}$. Thus,  $\pi _{1,q}(L)\in G_{n,k}$ is defined by $r_{q}(L)$.

We define  a $T^{n}$-equivariant projection $\pi _{2,q} : E_{2,q}\to G_{n-1, k-1}$ as the following composition:
\[
E_{2,q} \stackrel{c_{n,k}}{\longrightarrow} G_{n, n-k}\backslash G_{n-1, n-k-1} \stackrel{\pi _{1,q}}{\longrightarrow}G_{n-1, n-k}\stackrel{c_{n-1, n-k}}{\longrightarrow}G_{n-1, k-1}.
\]
We obtain the  following,( cf~\cite{Hog}):
\begin{prop}
The equivariant projection $\pi_{1,q}: E_{1,q} \rightarrow G_{n-1,k}$ can be identified with the  $T^{n}$-vector bundle $\xi_{n-1,k}\rightarrow G_{n-1,k}$, while the equivariant projection $\pi_{2,q}: E_{2,q} \rightarrow G_{n-1,k-1}$  can be identified with  the $T^{n}$-vector bundle $\eta_{n-1,n-k}\rightarrow G_{n-1,k-1}$.
\end{prop}

\begin{proof} 
 The homeomorphism $\xi _{n-1,k}\to E_{1,q}$ is defined by
\[
(x\in L, L) \to A_{q}(x,L)= {A_{x}\choose A_{L}}
\]
 where $A_{L}$ is an $((n-1)\times k)$-matrix of the   subspace $L$  in a fixed basis  and $A_{x}$ is  an $k$-vector which represents the  vector $x$ in this basis. The notation $A_{q}(x,L)$ means that the vector $A_{x}$ is inserted as an $q$-th row of the matrix consisting of the matrices  $A_{L}$ and $A_{x}$. Then the subspace $L\subset \C ^{n}$ determined by the matrix $A_{q}(x,L)$ does not contain the coordinate vector $e_{q}$.
 
\end{proof}

\begin{cor}
For any $q$,  $1\leqslant q \leqslant n$ there are two  $T^{n}$-pairs $(G_{n,k}, G_{n-1, k-1}(q))$\\  and $(G_{n, k}, G_{n-1, k}(q))$ which  define  $T^{n}$-spaces $X_{n,k}(q) = T(\xi _{n-1, k})$ and $Y_{n, k}(q) = T(\eta _{n-1,n-k})$  that is  the Thom spaces for $\xi _{n-1,k}$ and  $\eta _{n-1, n-k}$.
\end{cor}

\begin{proof}
Since $E_{1,q} = G_{n,k}\backslash G_{n-1,k-1}(q)$ is homeomorphic to $\xi _{n-1,k}$ it follows   that
$G_{n,k}/ G_{n-1,k-1}(q)$ is homeomorphic to the Thom space  $T(\xi _{n-1,k})$. 
An action of $T^{n}$ on $T(\xi _{n-1,k})$ is given by
\[
T^{n} \cdot (x, L) = T^{n} \cdot A_{q}(x,L)= A_{q}(T^{1}x, T^{n-1}L) = {T^{1}\cdot A_{x}\choose T^{n-1}\cdot A_{L}}
\]
where $T^1= T^1(q)$ is a subgroup of $T^{n}$  given by $z\to (z,\ldots,z,1_q,z,\ldots ,z)$ and $T^{n-1}=T^n/T^1$.
\end{proof}

Note that a continuous  action of a group $G$ on a topological space $X$ induces the $G$-action on the quotient space  $X/Y$ for any $G$-invariant subspace $Y\subset X$.  Moreover, $(X/Y)/G\cong (X/G)/(Y/G)$. Therefore, the previous Corollary implies:

\begin{cor}\label{orbit-Thom}
The orbit spaces of the Thom spaces $X_{n,k}(q) = T(\xi _{n-1,k})$ and $Y_{n, k}(q) = T(\eta _{n-1,n-k})$ by the $T^n$-action are homeomorhic to the quotient spaces
$(G_{n,k}/T^n)/(G_{n-1, k-1}/T^{n-1})$ and $(G_{n,k}/T^n)/(G_{n-1,k}/T^{n-1})$ respectively.
\end{cor} 

Since $G_{n-1, 1} = \C P^{n-2}$ and $\C P^{n-2}/T^{n-2}\cong \Delta _{n-2}$,  one has:

\begin{cor}\label{orbX}
The orbit space $X_{n,2}(q)/T^n$ is homotopy equivalent to $G_{n,2}/T^n$.
\end{cor}
  
In particular, since~\cite{MMJ} gives that $G_{4,2}/T^4\cong S^5$ , one has
\begin{cor}
The orbit spaces $X_{4,2}(q)/T^4$ and $Y_{4,2}(q)/T^4$ are homotopy equivalent to $S^5$.
\end{cor}  

\section{Orbit spaces of the strata}
It is known  that   $G_{n,k}/(\C ^{*})^{n}$ is  a non-Hausdorff topological space~\cite{GFMC}  for the topology induced from $G_{n,k}$.  We prove that the spaces of parameters $F_{\sigma}$ as well as the orbit spaces $W_{\sigma}/T^n$ are much better behaved.

The charts $M_{I}$ are invariant under the action of $(\C ^{*})^n$  and this action via homeomorphism $u_{I}: M_{I} \to \C ^{k(n-k)}$  induces an  action of $(\C ^{*})^n$ on $\C ^{k(n-k)}$.  For $I=\{1,\ldots ,k\}$ this induced action is given as follows:
\[
\left(\begin{array}{cccc}
t_1 & \ldots & &  0\\
0 & t_2 & \ldots & 0 \\
\vdots &\ddots & & \vdots \\
0 & \ldots & & t_k\\
t_{k+1}a_{k+11} &\ldots & & t_{k+1}a_{k+1k}\\
\vdots & & &\vdots \\
t_{n}a_{n1} & \ldots & & t_{n}a_{nk}
\end{array}\right) =  
\left(\begin{array}{cccc}
1 & \ldots & &  0\\
0 & 1 & \ldots & 0 \\
\vdots &\ddots & & \vdots \\
0 & \ldots & & 1\\
\frac{t_{k+1}}{t_1}a_{k+11} &\ldots & & \frac{t_{k+1}}{t_k}a_{k+1k}\\
\vdots & & &\vdots \\
\frac{t_{n}}{t_1}a_{n1} & \ldots & & \frac{t_{n}}{t_k}a_{nk}
\end{array}\right).
\]

Let us consider a representation $(\C^{*})^{n}\to (\C ^{*})^{k(n-k)}$ defined by 
\[
(t_1,\ldots ,t_{n})\to (\frac{t_{k+1}}{t_1},\ldots ,\frac{t_{k+1}}{t_k},\ldots ,\frac{t_{n}}{t_1},\ldots ,\frac{t_{n}}{t_k}).
\]
Put $\tau _{i} = \frac{t_{k+1}}{t_i}$ for $1\leq i\leq k$ and $\tau _{k+i} = \frac{t_{k+i+1}}{t_1}$ for $1\leq i \leq n-k-1$. It holds   $\frac{t_p}{t_s} = \frac{\tau _{p-1}\tau _{s}}{\tau _{1}}$ for any $k+2\leq p\leq n$ and $1\leq s\leq k$. In this way, the  representation of the torus $(\C ^{*})^{n}$ in $(\C ^{*})^{k(n-k)}$ is given by:
\begin{equation}\label{repr}
(\tau _{1},\ldots ,\tau _{n}) \to (\tau _{1},\ldots ,\tau _{k}, \tau _{k+1},\frac{\tau _{k+1}\tau _{2}}{\tau _{1}},\ldots , \frac{\tau _{k+1} \tau _{k}}{\tau _{1}},\tau _{k+2},\frac{\tau_{k+2}\tau_{2}}{\tau_{1}}
, \ldots \frac{ \tau_{k+2}\tau _{k}}{\tau _{1}}, \ldots, \tau _{n}, \frac{\tau _{2}\tau _{n}
}{\tau _1}, \ldots \frac{\tau _{n}\tau _{k}}{\tau _{1}}).
\end{equation}   
It follows that the action of the torus  $(\C ^{*})^{n}$ on $\C ^{k(n-k)}$ in a local chart for $G_{n,k}$ is given as a composition of the representation~\eqref{repr} and the standard  action of the torus $(\C ^{*})^{k(n-k)}$ on
$\C ^{k(n-k)}$.

\subsection{Orbits of $(\C ^{*})^{k}$-action  on   $\C ^{n}$}
In order to describe the action of the algebraic torus $(\C ^{*})^{n}$ locally in the charts for $G_{n,k}$ we consider more general situation.

Assume it is given an action of the torus $(\C ^{*})^{k}$ on $\C ^{n}$ as a composition of a representation
$\rho ^{*}: (\C ^{*})^{k}\to (\C ^{*})^{n}$ and the standard action of $(\C ^{*})^{n}$ on $\C ^{n}$. The representation $\rho ^{*}$ induces the representation $\rho : T^k\to T^{n}$  which can be written as  $\rho = (\rho _1,\ldots , \rho _{n})$, where $\rho _{j} : T^{k}\to S^1$, $1\leq j\leq n$ are homomorphisms. The characters $\rho _{j}$ can be further written as $\rho _{j}(t_1,\ldots ,t_k) = \rho _{j}(e^{2\pi \sqrt{-1}x_1},\ldots ,e^{2\pi \sqrt{-1}x_k}) = e^{2\pi \sqrt{-1}\sum _{i=1}^{k}\alpha _{j}^{i}x_i}$ for some  $(\alpha_{j}^{1},\ldots,\alpha _{j}^{k})\in \Z ^{k}$. The vectors $\Lambda ^{j} = (\alpha _{j}^{1},\ldots ,\alpha _{j}^{k})$, $1\leq j\leq n$  are  known to be the weight vectors of the  representations $\rho$ and $\rho ^{*}$.

Let  us fix some subset $J =\{j_1,\ldots ,j_l\}\subseteq \{1,\ldots ,n\}$ and put 
\[
\C ^{J} = \{ (z_1,\ldots ,z_n)\in \C ^n | z_j\neq 0,\; j\in J,\;\; z_j =0,\;j\notin J \}.
\]
The coordinate subspaces $\C ^{J}$ are $(\C ^{*})^k$-invariant and $\C ^{n} = \cup _{J} \C ^{J}$. 

Denote by  $V^{J}$ a matrix given by the weight vectors $\Lambda ^{j}$, $j\in J$. This matrix defines   a linear map $f_{J} : \R ^k\to \R ^{|J|}$ such that  $f_{J}(\Z ^k)$ is a direct summand in $\Z ^{|J|}$.  Let $q = \rank V^{J}$. 
\begin{lem}\label{free}
The points from $\C ^{J}$ have the same stabilizer $(\C ^{*})^{k}_{J}\subseteq (\C ^{*})^k$, where $(\C ^{*})^{k}_{J} = (\C ^{*})^{k-q}$,  that is $(\C ^{*})^q$ acts freely on $\C ^{J}$. 
\end{lem}
\begin{proof}
The stabilizer of a point $z\in \C ^{J}$ is given by the intersection of the kernels of the characters $\rho _{j}$, $j\in J$. It is a subgroup $(\C ^{*})^{k}_{J}\subseteq (\C ^{*})^k$, and since we assume that $q =\rank V^{J}$,  we obtain that   
$(\C ^{*})^{k}_{J} = (\C ^{*})^{k-q}$.
\end{proof}
\begin{cor}
If $l=|J| = q$, then $\C ^{J}$ is an $\C ^{k}$-orbit.
\end{cor}

Let us consider now the case when  $q < l=|J|$. We apply  standard procedure for the  description of the   $(\C ^{*})^{k}$-orbit of a point from $\C ^{J}$. 

\begin{itemize}
\item Let  $L$ be a lattice in $\Z ^{l}$ spanned by $f_{J}(\Lambda ^{j})$, $j\in J$.   
\item There exists a  unique integral lattice $\hat{L}$ in $\Z ^{l}$  which is  orthogonal to $L$. 
\item By fixing a basis in $\hat{L}$ we obtain the matrix $\hat{V}^{J}$  of  the dimension $l\times (l-q)$. Denote its elements by $\omega _{i}^{j}\in \Z$, $1\leq i\leq l$, $1\leq j\leq l-q$. 
\end{itemize}

Let us consider  an  algebraic   map $F_{J} : (\C ^{*})^{J}\to (\C ^{*})^{l-q}$ given by
\begin{equation}\label{c-orbit}
(z_{j_1},\ldots ,z_{j_l}) \to (z_{j_1}^{\omega _{1}^{1}}\cdots z_{j_l}^{\omega _{l}^{1}},\ldots ,z_{j_1}^{\omega _{1}^{l-q}}\cdots z_{j_l}^{\omega _{l}^{l-q}}).
\end{equation}

Since the lattice  $\hat{L}$ is orthogonal to $L$ we obtain:

\begin{lem}
An algebraic map $F_{J} : (\C ^{*})^{J}\to (\C ^{*})^{l-q}$ is $(\C ^{*})^k$-invariant, where  $\C ^{l-q}$ is considered  with  trivial  $(\C ^{*})^k$-action.
\end{lem}

\begin{cor} 
The set of  preimages  $F_{J}^{-1}(c)$, where $c=(c_1,\ldots ,c_{l-q})\in (\C ^{*})^{l-q}$ gives the set of all $(\C ^{*})^{k}$-orbits in $\C ^{J}$. 
\end{cor}
 
The characters of  the  representation $\rho: T^{n}\to T^{k(n-k)}$ which is given by~\eqref{repr}, are:
\begin{equation}\label{reprgr}
\rho _{k+1,i}(\tau_1,\ldots ,\tau_n) = \tau_i, \;\; 1\leq i\leq k, 
\end{equation}
\[ \rho _{i,1}(\tau_1,\ldots ,\tau _n) =  \tau_i,\;\;  k+1\leq i\leq n,
\]
while   for $\{p,s\}\neq \{k+1,i\}, \{i,1\}$, we obtain
\[
\rho _{p,s} = \frac{\tau_{p-1}\tau_s}{\tau_1}, \; \; 2\leq s\leq k,\; k+2\leq p\leq n. 
\]
It follows that the weight vectors for the representation $\rho  : T^{n}\to T^{k(n-k)}$ are given by
\[
\Lambda _{k+1,i} = (0,\ldots 0, \underbrace{1}_{i} ,0,\ldots ,0),\;\;  1\leq i\leq k,
\]
\[
 \Lambda _{i,1} = (0,\ldots ,0, \underbrace{1}_{i}, 0,\ldots ,0),\;\;   k+1\leq i\leq n,
\]
\[
\Lambda _{p,s} = (-1,0\ldots ,0,\underbrace{1}_{p-1},0,\ldots ,0,\underbrace{1}_{s},0,\ldots,0),\;\; \{p,s\}\neq \{k+1,l\}, \{l,1\}.
\]

Let us  consider a stratum  $W_{\sigma}$ and  assume that $W_{\sigma}\subset  M_{I}$, where without loss of generality we can take  that $I=\{1,\ldots ,k\}$.  We obtain  the coordinate map $u_{I}: W_{\sigma}\to \C ^{k(n-k)}$ given by $u_{I}(X) = (z_1,\ldots ,z_{k(n-k)})$. Note first the following: if $z_i=0$ for some $X_{0}\in W_{\sigma}$ then $z_i=0$  for all $X\in W_{\sigma}$. Namely, in the notation of Subsection~\ref{aps},   let  $z_i = a_{ps}$ for some $k+1\leq p\leq n$ and $1\leq s\leq k$. Put  $\hat{I} = (I - \{s\})\cup \{p\}$. Then   $P^{\hat{I}}(X_{0})=0$, so  the definition of $W_{\sigma}$ implies that $P^{\hat{I}}(X)=0$ for all $X\in W_{\sigma}$. This further implies that $z_i=0$ for all $X\in W_{\sigma}$.  

Therefore, we may assume that  
\[
u _{I}(W_{\sigma})\subseteq \C ^{J} = \{ (z_1,\ldots ,z_{k(n-k)}) \in \C ^{k(n-k)},\;   z_{i}\neq 0,\; i\in J,\; z_i=0,\; i\notin J\}.
\]
Lemma~\ref{free}  implies:
\begin{prop}\label{station}
All  points of  a  stratum $W_{\sigma}$ in $G_{n,k}$ have the same stabilizer $\C ^{*}_{\sigma}\subset (\C ^{*})^{n}$.
\end{prop}
 A  stratum $W_{\sigma}$ is by  definition invariant for  the action of the algebraic torus $(\C ^{*})^{k}$. Denote by  $I_{j_1},\ldots ,I_{j_d}\subset \{1,\ldots ,k+1\}$, $|I_{j_i}|=k$ such subsets that $P^{I_{j_i}}(X) = 0$ for all $X\in W_{\sigma}$. Note that the conditions  $P^{I_{j_i}}(X) = 0$ impose   relations on $(z_1,\ldots ,z_{k(n-k)}) \in u_{I}(W_{\sigma})$,  which we denote by $P_{J}^{I_i}(z_1,\ldots ,z_{k(n-k)}) =0$.

It follows  that  $u _{I}(W_{\sigma})$ is:
\begin{itemize}
\item  either the whole $\C ^{J}$,
\item  or  the intersection of  the set of all $(\C ^{*})^{k}$ -orbits in $\C ^{J}$  with the surfaces which are defined by the equations imposed   by those    Pl\" ucker coordinates which are zero for  the points of $W_{\sigma}$:
\[
u_{I}(W_{\sigma}) = \left\{
\begin{array}{c}
F_{J}^{-1}(c),\;\; c=(c_1,\ldots ,c_{l-q})\in (\C ^{*})^{l-q}\\
P^{I_i}_{J}(z_1,\ldots ,z_{k(n-k)}) = 0,\; i=1,\ldots ,d.  
\end{array} \right.
\]
Here $l$ and $q$ are as in the  previous paragraph. 
\end{itemize}

As a result we obtain a   family of algebraic surfaces which are parametrized by   $F_{\sigma, I}\subseteq(\C ^{*})^{l-q}$.  Here 
$F_{\sigma, I}= \{ c\in (\C ^{*})^{l-1} | F_{J}^{-1}(c) \cap\{ (z_1,\ldots ,z_{k(n-k)}) | P^{I_i}_{J}(z_1,\ldots, z_{k(n-k)})=0, i=1,\ldots , d\}\neq \emptyset\}$ is an open algebraic manifold.  Any  of these surfaces is,  by Lemma~\ref{free},  an $(\C ^{*})^{q}$-orbit.  According to~\cite{AT},  this orbit maps  by the moment map to $\stackrel{\circ}{P_{\sigma}}$. In this way,  we obtain  that $W_{\sigma}/(\C ^{*})^{n}=F_{\sigma}$ is homeomorphic to $F_{\sigma, I}$ and that the map $\widehat{\mu} : W_{\sigma}/T^n \to \stackrel{\circ}{P_{\sigma}}$ is a fiber bundle with the  fiber either a point either $F_{\sigma, I}\subseteq(\C ^{*})^{l-q}$.

Thus,  we proved:
\begin{thm}~\label{triv}
The map $\hat{\mu} : W_{\sigma}/T^n\to \stackrel{\circ}{P}_{\sigma}$ is a locally trivial fiber bundle whose fiber is an  open algebraic manifold  $F_{\sigma}$. Thus, $W_{\sigma}/T^n \cong \stackrel{\circ}{P}_{\sigma}\times F_{\sigma}$.
\end{thm}
    
We want to emphasize  that the trivialization stated in Theorem~\ref{triv} is  canonical:
\begin{prop}\label{canon-trivial}
There is the canonical  trivialization $h_{\sigma} : W_{\sigma}/T^{n} \to \stackrel{\circ}{P_{\sigma}}\times F_{\sigma}$.
\end{prop}
\begin{proof}
The space of parameters $F_{\sigma}$ is defined by $F_{\sigma} = W_{\sigma}/(\C ^{*})^{n}$, so denote by $p_{\sigma} : W_{\sigma}/T^{n}\to F_{\sigma}$ the canonical projection. It follows that  the map $h_{\sigma} : W_{\sigma}/T^{n}\to \stackrel{\circ}{P_{\sigma}}\times F_{\sigma}$ defined by $h_{\sigma} = (\hat{\mu}, p_{\sigma})$ is a canonical homeomorphism.
\end{proof}

We use the notation $h : W/T^n \to \stackrel{\circ}{\Delta _{n,k}}\times F$ for the canonical trivialization of the main stratum.

\subsection{On singular and regular values of the moment map}. Let us  consider  an  effective action of the algebraic torus $(\C ^{*})^{n-1}$ on $G_{n,k}$.  We want to  prove that a  regular  point of the  moment map $\mu : G_{n,k}\to \Delta _{n,k}$ is determined by the  non-triviality of its stationary subgroup. In order to do that we consider more general situation.

 Let us consider the complex  projective space $\C P^{N}$, where $N={n\choose k}-1$ endowed with an action of the torus $T^{n}$ which is given by a composition of the $k$-th exterior power representation $\rho : T^{n}\to T^{N+1}$ and the canonical action of $T^{N+1}$ on $\C P^{N}$. This action extends to an action of the algebraic torus $(\C ^{*})^{n}$. Let $\Lambda _{I}\in \Z ^{n}$  denote the weight vectors of the representation $\rho$, where $I\in {n \choose k}$. More precisely $\Lambda _{I}^{i}=1$ for $i\in I$,  while $\Lambda _{I}^{i} =0$ for $i\notin I$. Then the moment map $\mu : \C P^{N} \to \R ^{n}$ for the considered $T^n$ action on $\C P^{N}$ is   given by
\begin{equation}\label{moment}
\mu ([z_{I}, \; I\in {n\choose k}]) = \frac{1}{\sum_{I\in {n\choose k}}|z_{I}|^{2}}\sum _{I\in {n\choose k}}|z_{I}|^{2}\Lambda _{I}.
\end{equation}
We want to determine the rank of the differential $d\mu$ at an arbitrary point ${\bf z} = [z_{I}]\in \C P^{N}$. Let $P_{\bf{z}}$ denote a   polytope  which is the convex span of those vectors $\Lambda _{I}$ such that $z_{I}\neq 0$. 
\begin{thm}\label{rankcp}
The rank of the differential  of the moment map $\mu :  \C P^{N} \to \R^{n}$ given  by~\eqref{moment}  is
\begin{equation}
\rk d\mu ({\bf z}) = \dim P_{\bf{z}}
\end{equation}
for  an arbitrary point ${\bf z}\in \C P^{N}$.
\end{thm}

\begin{proof}

We first note that the map $\mu$ is $T^{n}$-invariant and  that the image of $\mu$ is the hypersimplex $\Delta _{n,k}$.  This implies that $\rk d\mu \leq n-1$  for any $z\in \C P^{N}$.  We also note  that  the map $\mu  : \O ({\bf z}) \to \stackrel{\circ}{P}_{\bf{z}}$ is a smooth fiber bundle for  the $(\C ^{*})^{n}$-orbit $\O (\bf{z})$ of a point $\bf{z}$, where $\stackrel{\circ}{P_{\bf{z}}}$ is an interior of the polytope $P_{\bf z}$. It implies that the rank of the differential $d\mu$ at a point ${\bf z}$ is $\geq \dim   P_{\bf{z}}$.   Therefore,
we differentiate the following cases for the point ${\bf z}$:
\begin{itemize}
\item If $z_{I}\neq 0$ for all $I\in {n\choose k}$ then   $\stackrel{\circ}{P_{\bf{z}}} = \stackrel{\circ}{\Delta} _{n,k}$ and $\rk d\mu ({\bf z}) =  n-1$.
\item Let $ \{I_1, \ldots , I_{l}\}$ denote the set of those  $I$ for which $z_{I_{j}}=0$, while $z_{I}\neq 0$ for $I\notin \{I_1, \ldots I_{l}\}$. Fix  a chart in $\C P^{N}$ determined by the condition that the coordinate indexed by a fixed $I_{p}\notin\{I_1, \ldots I_{l}\}$ is not zero. Then the point  ${\bf z}$ belongs to this chart and the  moment map $\mu $, in  local coordinates of this chart, is given by
\[
 \mu ([z_{I}, \; I\in {n\choose k}\setminus I_{p}]) = \frac{1}{1+\sum_{I\neq I_{p}}|z_{I}|^{2}}(\Lambda _{I_{p}}+\sum _{I\neq  I_{p}}|z_{I}|^{2}\Lambda _{I}).
\]
Let $z_{I}= x_{I} +iy_{I}$.
It follows that 
\[
\frac{\partial \mu}{\partial x_{I_{0}}} = \frac{1}{(1+\sum _{I \neq I_p} |z_{I}|^{2})^{2}}\big ( 2x_{I_{0}}(1+\sum _{I_\neq I_{p}}|z_{I}|^{2}) \Lambda _{I_{0}}-2x_{I_0}(\Lambda _{I_{p}}+\sum _{I\neq I_{p}}|z_{I}|^{2}\Lambda _{I})\big )
\]
and
\[
\frac{\partial \mu}{\partial y_{I_{0}}} = \frac{1}{(1+\sum _{I \neq I_p} |z_{I}|^{2})^{2}}\big ( 2y_{I_{0}}(1+\sum _{I_\neq I_{p}}|z_{I}|^{2}) \Lambda _{I_{0}}-2y_{I_0}(\Lambda _{I_{p}}+\sum _{I\neq I_{p}}|z_{I}|^{2}\Lambda _{I})\big )
\]
Therefore,
$\frac{\partial \mu}{\partial x_{I_{0}}} = \frac{\partial \mu}{\partial _{I_{0}}}=0$ at a point $z_{I_{0}}=0$. It implies that
$\rk d\mu$ at a point $z_{I_{0}}=0$ is equal to the $\rk d\mu |_{z_{I_{0}}=0}$, where $\mu |_{z_{I_{0}}=0}$ is the  restriction of $\mu$  to the coordinate subspace $z_{I_{0}}=0$, that is
\[
(rkd\mu)(z_{I_{0}}=0) = (\rk d(\mu |{z_{I_{0}}=0}))(z_{I, I\neq I_{0}}),
\]
Altogether we obtain that 
\begin{equation}\label{ukupan}
(\text{rank}d\mu)({\bf{z}}) = \text{rank}d(\mu |_{z_{I_{1}}=\ldots =z_{I_{l}}=0})(z_{I, I\neq I_1, \ldots I_l}).
\end{equation}
Denote by $\C P^{N-l}_{I_1, \ldots I_{l}}$ the complex projective space defined   by the coordinate subspace $\C ^{N-l+1}_{I_1, \ldots .I_{l}}$ in $\C ^{N+1}$ whose axis are  indexed by $I\notin \{I_{1}, \ldots I_{l}\}$ and denote by $ P_{I_{1}, \ldots ,I_{l}}$  the  convex polytope which is obtained as the convex hull of the vectors $\Lambda _{I}$, $I\notin \{I_1, \ldots , I_{l}\}$. Let us consider the map
\[\mu |_{z_{I_{1}}=\ldots =z_{I_{l}}=0} : \C P^{N-l}_{I_1,\ldots , I_l}\to P_{I_{1}, \ldots I_{l}}.
\]
The point $\bf{z}$ belongs to $\C P^{N-l}_{I_1, \ldots I_{l}}$ as well as its $(\C ^{*})^{N-l+1}_{I_1, \ldots ,I_l}$-orbit. Since all coordinates for  ${\bf z}\in \C P^{N-l}_{I_1, \ldots I_{l}}$ are non zero, it follows that
\[
\mu |_{\C P^{N-l}_{I_1, \ldots I_{l}}} (  (\C ^{*})^{N-l+1}_{I_1, \ldots ,I_l}\cdot {\bf z} )) = \stackrel{\circ}{P}_{I_{1}, \ldots I_{l}}.
\]
Thus, $P_{\bf{z}} = P_{I_{1}, \ldots I_{l}}$ and 
\[
\rk d(\mu |_{\C P^{N-l}_{I_1, \ldots I_{l}}})({\bf z}) = \dim  P_{I_{1}, \ldots ,I_{l}}.
\]
Together with~\eqref{ukupan} this proves that $(\rk d\mu)({\bf z}) = \dim P_{\bf z}$. 
\end{itemize}
\end{proof}

By Lemma~\ref{Plembed}  the Pl\"ucker embedding $G_{n,k}\to \C P^{N}$ is  $T^n$-equivariant. Therefore, by Theorem~\ref{rankcp} for the  moment map $\mu : G_{n,k}\to \R ^{n}$  defined by~\eqref{momentmap}  it holds:
\begin{cor}\label{rkmu}
The rank of the differential of the moment map $\mu : G_{n,k} \to \R^{n}$  is given by
\[
\rk d\mu (L) = \dim P_{L}
\]
at an arbitrary point $L$, where $P_{L}$ is an admissible polytope for the point $L\in G_{n,k}$, that is $P_{L} = \mu (\overline{(\C ^{*})^{n}\cdot L})$.
\end{cor}
Together with Theorem~\ref{moment-map-polytope} this further gives:
\begin{cor}\label{reg-str}
A point $L\in G_{n,k}$ is a regular point for the moment map $\mu : G_{n,k}\to \R ^{n}$  if and only if   its  stationary subgroup $\C ^{*}_{L}\subset (\C ^{*})^{n+1} $ is trivial.
\end{cor}

It follows that the singular points in $G_{n,k}$  are given by those strata $W_{\sigma}$  whose stationary subgroups $\C ^{*}_{\sigma}$ are non-trivial.

As for the regular values of the moment map we deduce:
\begin{thm}\label{reg}
A point $x\in \stackrel{\circ}{\Delta} _{n,k}$  is a regular value for the moment map $\mu : G_{n,k}\to \Delta _{n,k}$ if and only if 
the preimage $\mu ^{-1}(x)$ does not intersect any stratum which has non-trivial stationary subgroup.
\end{thm}

\begin{rem} 
According to  classical results of  mathematical analysis, $\mu ^{-1}(x)$ is a closed smooth submanifold in $G_{n,k}$  of dimension $q+n-1$ for a regular value $x\in \stackrel{\circ}{\Delta}_{n,k}$. Here  $q=2(k-1)(n-k-1)$ is the dimension of the space of parameters of the main stratum.
\end{rem} 

We define  the {\it defect} of a point $L\in G_{n,k}$ to be an integer $q$ such that  the $T^{n}$- orbit of $L$ has the dimension $n-1-q$.  Then Theorem~\ref{reg} can be reformulated as:
\begin{cor}\label{reform}
A point $x\in \stackrel{\circ}{\Delta} _{n,k}$  is a regular value for the moment map $\mu : G_{n,k}\to \Delta _{n,k}$ if and only if 
any point from  the preimage $\mu ^{-1}(x)$ has the defect equal to zero.
\end{cor}
 
We provide  geometrical description of   some submanifolds in $G_{n,k}$, which consist of singular points for the moment map $\mu$. Let us consider the decomposition $\C ^n = \C ^{n_1}\times \C^{n_2}$, $n=n_1+n_2$,  and consider the submanifold  $G_{n_1,k_1, n_2,  k_2}$, $k=k_1+k_2$, in $G_{n,k}$,   which consists of all $k$-dimensional subspaces $L\subset \C ^{n}$ such that $L=\mathcal{L}(L_1, L_2)$, where $L_1$ is a $k_1$-dimensional subspace in $\C ^{n_1}$ and $L_2$ is a $k_2$-dimensional subspace in $\C ^{n_2}$. Then $G_{n_1, k_1, n_2, k_1} \cong G_{n_1, k_1}\times G_{n_2, k_2}$ and the action of $(\C ^{*})^{n}$ on $G_{n_1,k_1, n_2, k_2}$ is induced  by the canonical action of $(\C ^{*})^{n_1}$ on $G_{n_1, k_1}$ and the canonical action of $(\C ^{*})^{n_2}$ on $G_{n_2, k_2}$. Therefore, the defect of a point   $L\in G_{n_1,k_1, n_2, k_2}$ is $q \geq 1$. Then Corollary~\ref{reform} implies:
\begin{prop}\label{G-sing}
The points which belong to the submanifolds $G_{n_1,k_1, n_2,k_2}\cong G_{n_1, k_1}\times G_{n_2, k_2}$  in $ G_{n,k}$  are  singular points for the moment map $\mu : G_{n,k}\to \Delta _{n,k}$.
\end{prop}

Consider a function $f _{j}: \Delta _{n,k}\to \R$ defined by the projection on the $j$-th coordinate, that is $f(x_1,\ldots ,x_n)=x_j$, where $1\leq j\leq n$ and the  function $\mu_j : G_{n,k}\to \R$, $1\leq j\leq n$ defined by the composition $\mu_{j}=f_{j}\circ \mu$.

\begin{lem}
The set of critical values of the function  $\mu_{j}$ consists of two points  $0$ and $1$, and the set of its critical points is a  disjoint union of the smooth manifolds $\mu _{j}^{-1}(0) = G_{n-1, k}$ and $\mu _{j}^{-1}(1) = G_{n-1, k-1}$.
\end{lem}

\begin{proof}
  Note that $\mu _{j}(G_{n,k})=[0, 1]$.    By Corollary~\ref{rkmu} we have that   any fixed point $L$ for  the considered  $T^n$-action is a critical point for the map $\mu _{j}$.  Since $\mu _{j}(L)$ is equal to $0$ or $1$ for any fixed point $L$, it follows that the points $0$ and $1$ are critical values for $\mu _j$. It is easy to deduce from the proof of Corollary~\ref{rkmu} that no $t\in (0,1)$ can be a critical value for $\mu _j$. Moreover,  $\mu _{j}^{-1}(0) \cong G_{n-1, k}$ while $\mu _{j}^{-1}(1)\cong G_{n-1, k-1}$. 
\end{proof}

Note that since  the differential of the map $\mu _{j} : G_{n, k}\setminus (G_{n-1, k}\cup  G_{n-1, k-1})\to (0,1)$ is an epimorphism it follows that there exists a smooth, closed $T^n$ - invariant  submanifold $S_{n,k}$ in $G_{n,k}$ such that $G_{n, k}\setminus (G_{n-1, k}\cup  G_{n-1, k-1}) \cong S_{n,k}\times (0,1)$.
 Altogether the following result holds.

\begin{prop}
There exist $T^n$-equivariant  projections  $\pi _{0} : S_{n,k} \to G_{n-1, k}$ and $\pi _{1} : S_{n,k} \to G_{n-1, k-1}$ such that 
\[
G_{n,k} =( S_{n,k}\times [0,1])\cup _{\pi _{0}}G_{n-1,k}\cup _{\pi _{1}}G_{n-1,k-1},
\]
where $(s, 0)\backsimeq \pi _{0}(s)$ and $(s,1)\backsimeq \pi _{1}(s)$.
\end{prop}
It immediately follows: 
\begin{thm}\label{Snk}
There exist  projections $\hat{\pi} _{0} : S_{n,k}/T^n \to G_{n-1, k}/T^{n-1}$ and $\hat{\pi} _{1} : S_{n,k}/T^n \to G_{n-1, k-1}/T^{n-1}$ such that 
\[
G_{n,k}/T^n = (S_{n,k}/T^n\times [0,1])\cup _{\hat{\pi} _{0}}G_{n-1,k}/T^{n-1}\cup _{\hat{\pi} _{1}}G_{n-1,k-1}/T^{n-1},
\]
where $[(s, 0)]\backsimeq \hat{\pi} _{0}[s]$ and $[(s,1)]\backsimeq \hat{\pi} _{1}[s]$.
\end{thm}

For $k=2$ this yields to the following result.
\begin{cor}
The orbit space $G_{n,2}/T^n$ is homotopy equivalent to the space $C(S_{n,2}/T^n)\cup _{\hat{\pi} _{0}} G_{n-1,2}/T^{n-1}$ that is obtained by attaching the conus  $C(S_{n,2}/T^n)$ to the space $G_{n-1,2}/T^{n-1}$ by the attaching map $\hat{\pi} _{0} : S_{n,2}/T^n\cong  \partial C(S_{n,2}/T^n) \to G_{n-1,2}/T^{n-1}$.
\end{cor}
In this way we obtain an  inductive description of the orbit space $G_{n,2}/T^n$ in terms of the orbit spaces $S_{l, 2}/T^l$, where $l\leq n$.
\begin{ex}\label{S52}
 For $n=4$ and $k=2$ we obtain that $\Sigma (S_{4,2}/T^4)$ is homotopy equivalent  to  $G_{4,2}/T^4\cong S^5$. For $n=5$ and $k=2$ we obtain that $G_{5,2}/T^5$ is homotopy equivalent  to $C(S_{5,2}/T^5)\cup _{\hat{\pi} _{0}} S^5$. 
\end{ex}
\begin{rem}
It  arises a problem to describe  the orbit space $S_{n,k}/T^n$. This will be a topic of our forthcoming paper.
\end{rem}

\begin{rem}
Note that the manifolds $G_{n_1, k_1, n_2, k_2}$ are the closures of the strata in $G_{n,k}$ which are mapped, by the moment map, to the polytopes  $\Delta _{n_1,k_1}\times \Delta _{n_2, k_2}\subset \Delta _{n,k}$. The polytopes $\Delta _{n_1,k_1}\times \Delta _{n_2, k_2}$ can be perceived  as  the intersection of $\Delta _{n,k}$ with the planes $x_1+\cdots +x_{n_1}=k_1$ in $\R ^n$ according  to the decomposition $\R ^{n} = \R ^{n_1}\times \R^{n_2}$, where $1\leq k_1\leq n_1$, $1\leq n_1\leq n$.
\end{rem}
\section{The canonical action of $(\C ^{*})^{5}$ on $G_{5,2}$}
 We consider the canonical action of $T^5$ on the complex Grassmann manifold $G_{5,2}$. Note that the diagonal $\Delta$ in $T^5$ acts trivially  on $G_{5,2}$, so  the torus $T^4/\Delta$ acts  effectively on $G_{5,2}$. In the sequel we 	 equivalently consider both of these actions. The image of  $G_{5,2}$ by  the moment map~\eqref{momentmap} is the convex hull of the points $\Lambda_{ij} = (\delta _{ij}^{l})\in \R ^5$, $1\leq l\leq 5$ and $1\leq i<j\leq 5$ , such that $\delta_{ij}^{i}=\delta_{ij}^{j}=1$ and $\delta_{ij}^{l}=0$ for $l\neq i,j$. This polytope  is known as the hypersimplex $\Delta _{5,2}$. The hypersimplex $\Delta_{5,2}$ is a four-dimensional polytope which belongs to the hyperplane $x_1+\ldots +x_5=2$.  The atlas for $G_{5,2}$, defined by the 
Pl\"ucker coordinates,  consists of ten charts  $M_{I} = \{ L\in G_{5,2} | P^{I}(L)\neq 0 \}$, where $I= \{ i_1, i_2 \}\subset \{ 1,\ldots ,5\}$.   We     describe  the strata on $G_{5,2}$ following the pattern given in~\cite{MMJ}.

\subsection{ $(\C ^{*})^{5}$ -orbits and their admissible polytopes}
 We first explicitly  describe  the orbits of the  $(\C ^{*})^{5}$-action  on $\C ^{6}$  in a fixed chart.
\begin{prop}\label{Orbits}
The orbit of a point  ${\bf a} = (a_1,\ldots ,a_6)\in \C ^{6}$ for the induced  $(\C ^{*})^5$-action on $\C ^6$ is one of the following: 
\begin{enumerate}
\item it is   the $8$-dimensional manifold given by the equations
\[
c_1^{'}z_1z_5 = c_1z_2z_4,\;\; c_2^{'}z_1z_6=c_{2}z_3z_4,\;\; c_3^{'}z_2z_6=c_{3}z_3z_5,\;\; \text{for}\;\; c_i,c_{i}^{'}\neq 0\;\; \text{and}\;\; c_{1}c_2^{'}c_3=c_{1}^{'}c_{2}c_{3}^{'},
\]
if $a_i\neq 0$, $1\leq i\leq 6$; 
\item it is  the $8$-dimensional manifold given by  one of the following  six equations:
\begin{itemize}
 \item [(a)]  $c_{1}^{'}z_1z_5 = c_1z_2z_4,\; z_3=0\;\; \text{or}\;\; z_6=0\;text{for}\; c_1, c_1^{'}\neq 0$,\;\;  $\text{if}\;\; a_3=0\;\; \text{or}\;\; a_6=0$,
\item[(b)] 
$c_{2}^{'}z_1z_6 = c_2z_3z_4,\; z_2=0\;\; \text{or}\;\; z_5=0\;\text{for}\; c_2, c_2^{'}\neq 0$,\;\; $\text{if}\;\; a_2=0\;\; \text{or}\;\; a_5=0$,
\item[(c)]
 $c_{3}^{'}z_2z_6 = c_{3}z_3z_5,\; z_1=0\;\; \text{or}\;\; z_4=0\;\text{for}\; c_3, c_3^{'}\neq 0$, \;\; $\text{if}\;\; a_1=0\;\; \text{or}\;\; a_4=0$;
\end{itemize}
\item it is the $6$-dimensional manifold given by  one of the following three  equations:
\begin{itemize}
 \item[(a)]
$c_{1}^{'}z_1z_5 = c_{1}z_2z_4,\;  z_3=z_6=0\;\; \text{for}\;\; c, c^{'}\neq 0,$\;\; $\text{if}\;\; a_3=a_6=0$,
\item[(b)]
$c_{2}^{'}z_1z_6 = c_{2}z_3z_4,\;  z_2=z_5=0\;\; \text{for}\;\; c, c^{'}\neq 0,$ \;\; $ \text{if}\;\; a_2=a_5=0$,
\item[(c)]
 $c_{3}z_2z_6 = c_{3}^{'}z_3z_5 ,\;  z_1=z_4=0\;\; \text{for}\;\; c, c^{'}\neq 0$, \;\; $\text{if}\;\; a_1=a_4=0$;
\end{itemize}
\item 
 it is $\C _{I}^{*}\subset (\C ^{*})^6$, where $I\subset \{1,\ldots ,6\}$, $|I| =4$ and $I\neq \{2,3,5,6\}, \{1,3,4,6 \}, \{1,2,4,5 \}$,
if $a_i=a_j=0,\;\; \{i,j\}\neq \{1,4\}, \{2,5\}, \{3,6\}$;
\item
 it is $\C _{I}^{*}\subset (\C ^{*})^6$, where $I\subset \{1,\ldots ,6\}$, $|I|\leq 3$, if three of more coordinates of ${\bf a}$ are zero.
\end{enumerate}
\end{prop}   
\begin{proof}
Without loss of generality we consider the chart $M_{12}$. It follows from~\eqref{repr} that the    $(\C ^{*})^5$ - orbit of a  point $L\in M_{12}$  is given by 
$(t_1a_1,t_2a_2,t_3a_3,t_4a_4,\frac{t_2t_4}{t_1}a_5,\frac{t_3t_4}{t_1}a_6)$. Thus,  if $a_i\neq 0$ for all $1\leq i\leq 6$ then an element $(z_1,\ldots ,z_6)$ belongs to this orbit if and only if it satisfies the system of equations
\[
c_1^{'}z_1z_5 = c_1z_2z_4,\;\; c_2^{'}z_1z_6 = c_{2}z_3z_4,
\]
where $c_1=a_1a_5$, $c_{1}^{'} = a_2a_4$, $c_2=a_1a_6$, $c_{2}^{'} = a_3a_4$, $c_{3} = a_2a_6$, $c_{3}^{'} = a_3a_5$.
Since  $c_i,c_{i}^{'}\neq 0$ we obtain  in this way  the orbits of the type $(1)$.

If exactly one of the $a_i$'s is  equal to zero we obtain the orbits of the type $(2)$.  For example if  $a_1$ or $a_4$  is equal to zero we obtain  the orbits  $(c)$  in $(2)$.  The orbits of the type $(3)$ are obtained by elements $(a_1,\ldots ,a_6)$ such that   two of  the $a_i$'s  which belong to the same row are equal to zero.  For all other choices of  an element $(a_1,\ldots ,a_6)\in \C ^6$, its  orbit will be a coordinate  subspace as stated in the Proposition.  

If two of the  $a_i$'s  which do not belong to the same row are equal to zero, say $a_1=a_5=0$, the orbit of a such element will be $(0,t_2a_2,t_3a_3,t_4a_4,0, \frac{t_3t_4}{t_1}a_6)$. This orbits writes as  $(0,t_2a_2,t_3a_3,t_4a_4,0,t_1a_6)$. It implies that the orbit of this element is coordinate subspace $\C _{2346}^{*}$.  Similarly,  the complex four-dimensional coordinate subspaces $\C ^{*}_{ijkl}$, $\{i,j,k,l\}\neq \{1,3,4,6\},  \{2,3,5,6\}, \{1,2,4,5\}$ are $8$-dimensional orbits of the algebraic torus $(\C ^{*})^{5}$.

If three or more $a_i$'s are equal to zero then it can be  directly checked that the orbit of a such  element will be 
$\C _{I}^{*}$, where $I =\{i | a_i\neq 0\}$. In this way, we obtain all possible $\C _{I}^{*}$, where $I\subset \{ 1,\ldots ,6\}$ and $| I | \leq 3$. 
\end{proof}

\subsection{Admissible polytopes in the chart $M_{12}$.}
The image  by the moment map of an arbitrary $(\C ^{*})^{5}$-orbit is the  interior of a convex polytope spanned by some subset of vertices for  $\Delta _{5,2}$. These  polytopes, as we already said,  are called the   admissible polytopes.  We  describe here  admissible polytopes for  $(\C ^{*})^{5}$ -orbits,  which belong  to the chart $M_{12}$,  following their description  given in Proposition~\ref{Orbits}.
\begin{prop}\label{Polytopes}
Admissible polytopes for  $(\C ^{*})^{5}$ - orbits in $G_{5,2}$ in the chart   $M_{12}$  are, according  to Proposition~\ref{Orbits}, as  follows:
\begin{enumerate}
\item  For the orbits (1)  --  one of the following  four-dimensional polytopes:
\begin{itemize}  
\item the hypersimplex $\Delta _{5,2}$ if  $c_i,c_i^{'}\neq 0$, $c_i\neq c_i^{'}$; 
\item   spanned by   $9$ vertices different from   $\Lambda_{45}$ if $c_1\neq c_{1}^{'}$,  $c_3=c_3^{'}$; 
\item  spanned  by $9$ vertices different from  $\Lambda_{35}$ if $c_1\neq c_1^{'}$, $c_2=c_{2}^{'}$;
\item  spanned  by  $9$ vertices different  $\Lambda_{34}$ if $c_1=c_{1}^{'}$, $c_2\neq c_{2}^{'}$;  
\item spanned by  $7$ vertices different from   $\Lambda_{34}$, $\Lambda_{35}$, $\Lambda_{45}$ if $c_1=c_1^{'}$; $c_2=c_2^{'}$, $c_3=c_{3}^{'}$ . 
\end{itemize}
\item For the orbits  (2)  -- one of the following four-dimensional polytopes:
\begin{enumerate}
\item spanned by  $9$ vertices different from:   
\begin{itemize}
\item $\Lambda_{25}$ or $\Lambda_{15}$ for the orbits (a), 
\item $\Lambda_{24}$ or $\Lambda_{14}$ for the orbits (b),
\item $\Lambda_{23}$ or $\Lambda_{13}$ for the  orbits (c)
\end{itemize}
if  $c_i\neq c_{i}^{'}$;
\item  spanned  by $8$ vertices different from:
\begin{itemize}
\item  $\Lambda_{34}, \Lambda_{25}$ or $\Lambda_{34}, \Lambda_{15}$ for the orbits (a), 
\item $\Lambda_{35}, \Lambda_{24}$ or $\Lambda_{35}, \Lambda_{14}$ for the  orbits (b),
\item $\Lambda_{45}, \Lambda_{23}$ or $\Lambda_{45}, \Lambda_{13}$ for the  orbits (c)
\end{itemize}
if  $c_{i}=c_{i}^{'}$.
\end{enumerate}
\item For the orbits  (3)  -- one of the following   three-dimensional polytopes:
\begin{enumerate}
\item  an   octahedron spanned by   the vertices:
\[
\Lambda_{12}, \Lambda_{13},  \Lambda_{14}, \Lambda_{23}, \Lambda_{24}, \Lambda_{34};
\]
\[
\Lambda_{12}, \Lambda_{13}, \Lambda_{15}, \Lambda_{23}, \Lambda_{25}, \Lambda_{35};
\]
\[
\Lambda_{12}, \Lambda_{14}, \Lambda_{15}, \Lambda_{24}, \Lambda_{25}, \Lambda_{45}
\]
if  $c_i\neq c_{i}^{'}$;
\item spanned by the vertices:
\begin{itemize}
\item $\Lambda_{12}, \Lambda_{13}, \Lambda_{14}, \Lambda_{23}, \Lambda_{24}$ for the orbits (a),
\item $\Lambda_{12}, \Lambda_{13}, \Lambda_{15}, \Lambda_{23}, \Lambda_{25}$ for the orbits (b),
\item $\Lambda_{12}, \Lambda_{14}, \Lambda_{15}, \Lambda_{24}, \Lambda_{25}$ for the orbits (c)
\end{itemize}
if  $c_i=c_{i}^{'}$.
\end{enumerate}
\item For the orbits  $(4)$  -- one of the following   four-dimensional polytopes: 
\begin{itemize}
\item 
 spanned by  $8$ vertices, which do not contain the   pairs  $\Lambda_{1i}, \Lambda_{2j}$, where $i\neq j$ and $i,j\geq 3$;
\item 
spanned by  $7$ vertices,  which do not contain   the triples  $\Lambda_{1i}, \Lambda_{1j}$, $\Lambda_{ij}$ or the pairs $\Lambda_{2i}, \Lambda_{2j}$ where  $i\neq j$ and $i,j\geq 3$; 
\end{itemize}
\item For the orbits $(5)$  --  one of the following one, two or three-dimensional  polytopes:  
\begin{itemize}
\item  spanned by $\Lambda_{12}$, $\Lambda_{1i}$ or $\Lambda_{12}$, $\Lambda_{2i}$, where $3\leq i\leq 5$ for the   one-dimensional polytopes;
\item spanned by $\Lambda_{12}$,  $\Lambda_{1i}$, $\Lambda_{1j}$ or $\Lambda_{12}$, $\Lambda_{2i}$, $\Lambda_{2j}$ or $\Lambda_{12}$,  $\Lambda_{1i}$, $\Lambda_{2i}$
or $\Lambda_{12}$, $\Lambda_{1i}, \Lambda_{2j}, \Lambda_{ij}$,  where $3\leq i,j\leq 5$ and $i\neq j$ for the   two-dimensional polytopes;
\item spanned by $\Lambda_{12}$, $\Lambda_{13}$, $\Lambda_{14}$, $\Lambda_{15}$ or $\Lambda_{12}$,  $\Lambda_{23}$, $\Lambda_{24}$ $\Lambda_{25}$ or $\Lambda_{12}$,  $\Lambda_{1i}$, $\Lambda_{1j}$, $\Lambda_{2i}$, $\Lambda_{ij}$ or $\Lambda_{12}$,  $\Lambda_{2i}$, $\Lambda_{2j}$, $\Lambda_{1i}$, $\Lambda_{ij}$ or  $\Lambda_{12}$,   $\Lambda_{1i}$, $\Lambda_{1j}$, $\Lambda_{2k}, \Lambda_{ik}, \Lambda _{jk}$ or   $\Lambda _{12}$, $\Lambda_{2i}$, $\Lambda_{2j}$, $\Lambda _{1k}$, $\Lambda _{ik}, \Lambda _{jk}$ where $i,j,k\geq 3$, $i\neq j, k$ and  $k\neq j$,
for the  three-dimensional polytopes. 
\end{itemize}
\end{enumerate}   
\end{prop}
\begin{proof}
The proof is similar to  the proof of  Proposition~\ref{Orbits}. For the orbits of type (1) such that  $c_1\neq c_2$ and $c_1,c_2\neq 1$, we see that all minors in the corresponding  matrix $A_{L}$ of  an element $L$  of a such  orbit, are non-zero. Therefore,  the image of this orbit by the moment map will be the  convex hull over  all vertices $\Lambda_{ij}$, $1\leq i<j\leq 6$,  that  is the hypersimplex $\Delta _{5,2}$. If $c_1=c_2\neq 1$, then  $P^{45}(A_{L})=0$, while all  other minors are non-zero, which  implies that the image of the orbit of a such element $L$   will be a four-dimensional polytope,  described as the  convex hull of the vertices $\Lambda_{ij}$ apart from  $\Lambda_{45}$. The proof for the other  cases is similar. We take into account the description of  orbits given in the proof of Proposition~\ref{Orbits}  and from the set of all vertices remove  those corresponding  to the minors that are equal to zero.
\end{proof} 

More geometric description of the polytopes from  Proposition~\ref{Polytopes}  would be as follows.

\begin{cor}\label{polytopes_chart}
The admissible polytopes for  $(\C ^{*})^{5}$ - orbits in  the chart $M_{12}$ are:
\begin{enumerate}
\item four-dimensional:
\begin{itemize}
\item the  hypersimplex $\Delta _{5,2}$;
\item any of $9$ polytopes  with $9$ vertices that contains the vertex $\Lambda_{12}$;
\item any of $12$ polytopes with   $8$ vertices that   contains  the vertex $\Lambda _{12}$ and does not contain    
any  of the pairs $\Lambda_{1i}, \Lambda_{2j}$ or $\Lambda_{1i}, \Lambda_{jk}$ or $\Lambda_{2i}, \Lambda_{jk}$ where $i,j,k\geq 3$ and $i\neq j,k$, $j\neq k$;  
\item any of $7$ pyramids with $7$ vertices whose  top vertex is  one of the following:  $\Lambda_{12}$, $\Lambda_{13}$, $\Lambda _{14}$, $\Lambda_{15}$, $\Lambda_{23}$, $\Lambda_{24}$ or $\Lambda_{25}$;
\end{itemize}
\item three-dimensional:
\begin{itemize}
\item any of $3$ octahedra that contains  the pair of  vertices  $\Lambda_{12}, \Lambda_{34}$ or $\Lambda_{12}, \Lambda_{35}$ or $\Lambda_{12}, \Lambda_{45}$;
\item any of $3$ pyramids with $5$ vertices whose apex is $\Lambda_{12}$;
\item any of $6$ prisms with vertices $\Lambda_{12}, \Lambda_{1i}, \Lambda_{1j}, \Lambda_{2k}, \Lambda_{ik}, \Lambda_{jk}$ or $\Lambda_{12}, \Lambda_{2i}, \Lambda_{2j}, \Lambda_{1k}, \Lambda_{ik}, \Lambda_{jk}$, where $i,j,k\geq 3$, $i,j\neq k$, $i\neq j$;
\item any of $12$ pyramids with $5$ vertices that contains the vertex $\Lambda _{12}$ and whose apex is one of the following:  $\Lambda_{13}, \Lambda_{14}, \Lambda_{15}, \Lambda_{23}, \Lambda_{24}$ or $\Lambda_{25}$.
\item any of $2$ tetrahedra that contain the vertex $\Lambda_{12}$;
\end{itemize}
\item two-dimensional:
\begin{itemize}
\item any of $9$ triangles that contain the vertex $\Lambda_{12}$;
\item any of $6$ rectangles that  contain the vertex $\Lambda_{12}$;
\end{itemize}
\item one-dimensional: any of $6$ edges that  contains the vertex $\Lambda_{12}$;
\item zero-dimensional: the vertex $\Lambda_{12}$.       
\end{enumerate}
\end{cor}

\subsection{The admissible polytopes and the strata for $G_{5,2}$.}
It follows from Lemma~\ref{Sn-action-charts} that  any chart contains the same number of admissible polytopes and, in addition, it contains the same number of admissible polytopes of the same type, as  given by Proposition~\ref{polytopes_chart}.  Moreover,  the number of vertices of an admissible polytope $P_{\sigma}$ coincides with the number of charts which  contain $W_{\sigma}$. In addition, the number of all charts is equal to the Euler characteristic $\chi (G_{5,2})$. Using this we determine the admissible polytopes for $G_{5,2}$ and their numbers.

\begin{prop}
The admissible polytopes for  $G_{5,2}/T^5$  are:
\begin{enumerate}
\item four-dimensional:
\begin{itemize}
\item hypersimplex $\Delta _{5,2}$;
\item $10$ polytopes spanned by the $9$ vertices;
\item $15$ polytopes spanned by $8$ vertices;
\item $10$ pyramids spanned by $7$ vertices;  
\end{itemize}
\item three-dimensional:
\begin{itemize}
\item $5$ octahedra;
\item $30$ pyramids with $5$ vertices;
\item $10$ prisms with $6$ vertices;
\item $5$ tetrahedra;
\end{itemize}
\item two-dimensional:
\begin{itemize}
\item $30$ triangles;
\item $15$ squares;
\end{itemize}
\item $30$ edges;
\item $10$ vertices.
\end{enumerate} 
\end{prop}

{\bf Notations. } We set the following notations for some admissible polytopes. The admissible  polytope with nine vertices which does not contain the vertex $\Lambda _{ij}$ we denote by $K_{ij}(9)$, $1\leq i<j\leq 10$. The admissible polytope with eight vertices which does not contain the vertices $\Lambda_{ij}$ and $\Lambda_{kl}$ we denote by $K_{ij, kl}$, $1\leq i<j\leq 10$, $1\leq k<l\leq10$, $\{i,j\}\neq \{k,l\}$. The admissible  pyramid with seven vertices and  with the apex  $\Lambda_{ij}$  we denote by $K_{ij}(7)$, $1\leq i<j\leq 10$. All polytopes $K_{ij}(9), K_{ij, kl}, K_{ij}(7)$ are four-dimensional.  The prisms with six  vertices  we denote by $P_{ij}$, $1\leq i<j\leq 10$ and it is   easy to see they represent  base for the pyramids $K_{ij}(7)$.  By $O_{i}$, $1\leq i\leq 5$,  we denote the admissible octahedron which does not contain the vertices $A_{ik}$ where $1\leq k\leq 5$, $k\neq i$. 
By $T_{i}$, $1\leq i\leq 5$  we denote the  admissible tetrahedron whose vertices are $\Lambda_{ik}$, $1\leq k\leq 5$, $k\neq i$. The $30$ pyramids with five vertices belong to the $5$ octahedra and give their polytopal decomposition. The triangles are the facets of the octahedra and tetrahedra,  the $15$  squares are the diagonal squares for the  octahedra, three for each octahedron. 

\subsection{The combinatorial approach to admissible polytopes.}
We can approach the description of  admissible polytopes in purely combinatorial way.  An  admissible polytope is the convex hull  
$\sum _{1\leq i< j\leq 5}\alpha _{ij}\Lambda_{ij}$, where $0\leq \alpha _{ij}\leq 1$, $\sum _{1\leq i<j\leq 5} \alpha _{ij} =1$ and $\Lambda_{ij}$ are the vertices for $\Delta _{5,2}$.  The numbers $\alpha _{ij}$ can be calculated  using the  Pl\"ucker coordinates as follows:
\[
\alpha _{ij} = \frac{|P^{ij}|^{2}}{\sum \limits_{1\leq i< j\leq 5} |P^{ij}|^{2}}, \;\; 1\leq i<j\leq 5.
\]
The Pl\"ucker coordinates satisfies the relations
\begin{equation}\label{Plrel}
P^{i_1j_1} P^{i_2j_2} = P^{i_2j_1} P^{i_1j_2} - P^{i_1i_2} P^{j_1j_2},
\end{equation}
which impose the relations on $\alpha _{ij}$. Therefore the admissible polytopes can be described  as all possible  convex combinations of $\Lambda_{ij}$ where the coefficients of the combinations satisfy the relations imposed from~\eqref{Plrel}.
For example,
\begin{itemize}
\item if all $\alpha _{ij}\neq 0$ we obtain the admissible polytope $\Delta _{5,2}$;
\item if  $\alpha _{i_{1}j_{1}} =0$ and $\alpha _{ij}\neq 0$ for $\{i,j\}\neq \{i_{1},j_{1}\}$ we obtain the admissible polytope   $K_{i_{1}j_{1}}(9)$;
\item if $\alpha _{i_1j_1} =\alpha _{i_2j_2} = 0$ and $\{i_1,j_1\}\cap \{i_2, j_2\} = \emptyset$ and $\alpha _{ij} \neq 0$ for $\{i,j\}\neq \{i_1,j_1\}, \{i_2, j_2\}$ we obtain the admissible polytope 
$K_{i_1j_1,i_2j_2}$.
\item if   $\alpha _{i_1j_{1}} =\alpha _{i_1j_2}=0$  and $\alpha _{i_1j_3}\neq 0$  for $j_3\neq j_1,j_2$ it follows from~\eqref{Plrel} that $\alpha _{j_1j_2}=0$. In this way, we obtain the admissible polytope $K_{i_3j_3}(7)$, where $i_3 = \{1,\ldots ,5\} \setminus \{i_1,j_1,j_2,j_3\}$.  
\end{itemize}

One can also in a purely combinatorial way verify if some point from $\Delta _{5,2}$ belongs to the interior  of a fixed admissible polytope.
\begin{lem}
A point $x\in \Delta _{5,2}$ belongs to the interior of an admissible polytope $K$ with the vertices $\Lambda_{i_1j_1},\ldots ,\Lambda_{i_sj_s}$ if and only of $x = \sum _{l=1}^{s}\alpha _{i_lj_l}\Lambda_{i_lj_l}$ where $\sum _{l=1}^{s}\alpha _{i_1j_l} = 1$ and $\alpha _{i_lj_l} >0$ for all $1\leq l\leq s$.
\end{lem} 
The symmetric group $S_5$ acts on the set  of all admissible polytopes for $G_{5,2}$. We determine here the corresponding stationary subgroups and the fundamental strata.

\begin{lem}
The stationary subgroups for  $S_5$-action on the set of all admissible polytopes are isomorphic to:
\[
\Delta _{5,2}: \; S_5,\;\; K_{ij}(9):\; S_{2}\times S_{3}, \;\; K_{ij,kl}:\; S_2\times S_2\times S_2;\;\; 
K_{ij}(7):\; S_2\times S_3,\;\; P_{ij}:\;  S_2\times S_3,\]
\[ O_i :\; S_4,\;\; T_{i}:\; S_{4},\;\;
\text{five-vertex pyramids}:\; S_2\times S_2,\]
\[ \text{squares}:\; S_2\times S_2\times S_2 ,\;\; 
\text{triangles}:\; S_3\;\text{or}\; S_2\times S_3\]
\[ \text{edges}:\; S_2\times S_2,\; \text{vertices}:\; S_2\times S_3.
\]
\end{lem}
\begin{proof}
We demonstrate the proof for the admissible triangles. The vertices of an admissible triangle are $\{\Lambda_{ij},\Lambda_{ik},\Lambda_{il}\}$ or $\{\Lambda_{ij},\Lambda_{ik},\Lambda_{jk}\}$, where $i,j,k,l$ are pairwise distinct.  The stationary subgroup for  the triangles of the first form is $S_3$ and for the triangles of the second form it is $S_2\times S_3$.
\end{proof}

\begin{lem}\label{fund-5}
The fundamental polytopes for $G_{5,2}$ for admissible polytopes which:
\begin{itemize}
\item  belong to the interior of $\Delta _{5,2}$ are: $\Delta _{5,2}$, one  $K_{ij}(9)$, one $K_{ij, kl}$, one  $K_{ij}(7)$ and one   $P_{ij}$;
\item belong to the boundary of $\Delta _{5,2}$ are: one  $O_{i}$, one $T_i$, one square, two triangles, one edge  and  one vertex.
\end{itemize}
 \end{lem}
\begin{proof}
  The stationary subgroup for $K_{ij}(9)$  is isomorphic to $S_2\times S_3$ which has the  order $12$. Since the order of $S_5$ is $120$ we obtain that the $S_5$-orbit of $K_{ij}(9)$ has $10$ elements.  This coincides with the number of the admissible polytopes with $9$ vertices.  The proof for the other admissible polytopes, apart from triangles and edges, is analogous.  The orbit of a triangle, whose stationary subgroup is $S_3$ has $20$ elements, while the orbit of a triangle, whose stationary subgroup is $S_2\times S_3$ has $10$ elements. Altogether this gives $30$ triangles.
\end{proof}

\begin{cor}
The number of the strata for $T^{5}$-action on $G_{5,2}$ is $125+46=171$. The number of the  fundamental strata is $13$, where $q_3=q_4=q_6=2$ and $q_p=1$ for $p\neq 3,4,6$.
\end{cor}
\subsection{The closure of $(\C ^{*})^{5}$-orbits.}
The closure of a $(\C ^{*})^{5}$-orbit in $G_{5,2}$ is a toric variety. Following the arguments from~\cite{MMJ} we provide the description of these toric varieties.

\begin{thm}\label{toric}
The closure of an orbit for $(\C ^{*})^{5}$-action on $G_{5,2}$ is a toric variety:
\begin{enumerate}
\item of dimension eight (complex dimension $4$) and with:
\begin{itemize}
\item $10$ singular points if its admissible polytope is $\Delta _{5,2}$;
\item $9$ singular points if its admissible polytope  is $K_{ij}(9)$, $1\leq i<j\leq 5$;
\item $4$ singular points if its admissible polytope is $K_{ij,kl}$, $1\leq i,j,k,l\leq 5$, $i\neq j\neq k\neq l$;
\item $1$ singular point if its admissible polytope is $K_{ij}(7)$, $1\leq i<j\leq 5$;
\end{itemize}
\item of dimension six (complex dimension $3$): 
\begin{itemize}
\item with $6$ singular points if its admissible polytope is an octahedra $O_{i}$, $1\leq i\leq 5$;
\item with $1$ singular point if its admissible polytope is  a pyramid over square;
\item $\C P^2\times \C P^1$ if its admissible polytope is a prism $P_{ij}$, $1\leq i<l\leq 5$;; 
\item $\C P^3$  if its admissible polytope is a tetrahedra $T_{i}$, $1\leq i\leq 5$;
\end{itemize}
\item of dimension four:
\begin{itemize}
\item $\C P ^2$ if its admissible polytope is a triangle;
\item $\C P^1\times \C P^1$ if its admissible polytope is a rectangle;
\end{itemize}
\item of dimension two, i.~e.~ $\C P^1$.
\item of dimension zero vertices, i.~e.~ fixed points. 
\end{enumerate}
All singular points of these toric varieties map, by the moment map,  to the vertices of $\Delta _{5,2}$.
\end{thm}
As for the regular point of the moment map, Corollary~\ref{reg-str} implies:
\begin{cor}
All points of  the  eight-dimensional orbits  for $(\C ^{*})^{5}$-action on $G_{5,2}$ are the regular points for the moment map $\mu : G_{5,2}\to \Delta _{5,2}$.
\end{cor}
The only  $l$- dimensional admissible polytopes, $l\leq 3$,  which intersect  $\stackrel{\circ}{\Delta} _{5,2}$ are the prisms $P_{ij}$. More precisley,  $\stackrel{\circ}{P}_{ij}\subset \stackrel{\circ}{\Delta} _{5,2}$. Thus,  Theorem~\ref{reg} implies:
\begin{cor}
The set of singular values in  $\stackrel{\circ}{\Delta}_{5,2}$  for the moment map $\mu $ is  given by  $\cup _{1\leq i<j\leq 5}\stackrel{\circ}{P}_{ij}$. The set of singular points in $G_{5,2}$,  which are  mapped  by the moment map to $\stackrel{\circ}{\Delta}_{5,2}$,   is given by  the principal orbits of the toric varieties $W_{ij}\cong \C P^2\times \C P^1$.   The toric varieties $W_{ij}$ are  mapped, by the moment map,  to $P_{ij}$, where $1\leq i<j\leq 5$.  
\end{cor}
By Proposition~\ref{G-sing},  the toric varieties $W_{ij}\cong \C P^2\times \C P^1\subset G_{5,2}$ have a nice geometric description: 
\begin{cor}
The ten toric varieties $W_{ij}$, $1\leq i<j\leq 5$ correspond to the ten decomposition $\C ^5 = \C ^3\times \C ^2$. More precisely,   $W_{ij}$ consist of two-dimensional subspaces $L\subset \C ^5$ such that $L=\mathcal{L}(L_1, L_2)$, where $L_1, L_2$ are the lines in $\C ^3, \C ^2$ respectively,  for the corresponding decomposition of $\C ^5=\C ^3\times \C ^2$.
\end{cor}

\subsection{The strata in the chart $M_{12}$, their  spaces  of parameters and orbit spaces}
We analyze now the spaces of parameters of the strata as well as the $T^{5}$- orbit spaces of the strata. Since the action of $S_5$ on $G_{5,2}$ permutes the charts, the strata and, consequently, the admissible polytopes and the spaces of parameters, it is enough to consider these objects just in an one chart.

From Proposition~\ref{Orbits} and Proposition~\ref{Polytopes} it directly follows:  

\begin{cor}
Any  stratum $W_{\sigma}$   for which  $P_{\sigma}\neq  \Delta _{5,2}, K_{ij}(9), O_{l}$   consists of one orbit. Consequently,  $F_{\sigma}$ is a point.
\end{cor}  

The  main stratum $W$   consists  of the points whose all Pl\"ucker coordinates are non-zero. The main stratum belongs to all chart for $G_{5,2}$, so Proposition ~\ref{Orbits} implies:

\begin{cor}
The main stratum is, in a fixed chart,   given by the following system of equations:
\begin{equation}\label{main}
c_1^{'}z_1z_5 = c_1z_2z_4,\;\; c_2^{'}z_1z_6=c_{2}z_3z_4,\;\; c_3^{'}z_2z_6=c_{3}z_3z_5,
\end{equation}
where 
\begin{equation}\label{paramain}
c_1c_{2}^{'}c_3=c_{1}^{'}c_2c_3^{'} \;\; \text{and}\;\; c_i, c_i^{'}\neq 0, \;  c_i\neq c_{i}^{'}, \; c_i, c_{i}^{'}\in \C.
\end{equation}
\end{cor}
\begin{cor}
The space  of parameters $F$ of the main stratum $W$ is homeomorphic to
\[
\{((c_1:c_1^{'}), (c_2: c_2^{'}), (c_3: c_3^{'}))\in \C P^1 \times \C P^1\times \C P^1,\;\; c_{1}c_2^{'}c_3=c_{1}^{'}c_{2}c_{3}^{'},\;\; c_i,c_i^{'}\neq 0,\;\; c_i\neq c_i^{'}\}.
\] 
\end{cor}

The other non-orbit strata  are, in  the chart $M_{12}$,   given by the family of surfaces (1), (2) and (3) from Proposition~\ref{Orbits}.
Therefore we have: 
\begin{cor}
The space  of parameters for a  non-orbit stratum, which is  different from the main stratum, that is for a   stratum  whose admissible polytope is  $K_{ij}(9)$ or $O_{i}$,  is homeomorphic to  $\C P^{1}_{A} = \C P^{1}\setminus A$, where $A=\{(0:1), (1:0), (1:1)\}$.
\end{cor}

From Theorem~\ref{triv} it follows:

\begin{cor}
The orbit space  $W_{\sigma}/T^5$  of    a stratum $W_{\sigma}$ whose admissible polytope is:
\begin{enumerate}
\item $\Delta _{5,2}$ is homeomorphic to $\stackrel{\circ}{\Delta}_{5,2}\times F$;
\item $K_{ij}(9)$  is homeomorphic to $\stackrel{\circ}{K}_{ij}(9)\times \C P^1_{A} $;
\item  $O_{i}$ is homeomorphic to $\stackrel{\circ}{O_{i}}\times \C P^1_{A}$.
\end{enumerate} 
\end{cor}
\begin{cor}
The orbit space for $W_{\sigma}/T^5$ of a stratum $W_{\sigma}$ whose admissible polytope $P_{\sigma}$ is different from $\Delta _{5,2}$, $K_{ij}(9)$ and $O_{i}$ is homeomorphic to $\stackrel{\circ}{P_{\sigma}}$.
\end{cor}

\begin{rem}\label{param-chart}
Let us fix a chart $M_{ij}$ and let $W_{\sigma}\subset M_{ij}$ be a stratum such that $F_{\sigma}$ is not a point. Then, as Proposition~\ref{Orbits} shows, $W_{\sigma}$ is defined by the  equations whose  coefficients belong to the space of   parameters $F_{\sigma}$.  We denote by $F_{\sigma, ij}$ a coordinate record of the space of parameters $F_{\sigma}$ in the chart $M_{ij}$.  By $f_{\sigma, ij} : F_{\sigma, ij}\to F_{\sigma}$ we denote the induced  canonical homeomorphism.  We want to emphasize that if $F_{\sigma}$ is a point then  one can not assign to the point $F_{\sigma}$ the parameter coordinates in the chart $M_{ij}$,  since in this case the stratum $W_{\sigma}$ consists of one $(\C ^{*})^{5}$-orbit and by Proposition~\ref{Orbits}  no equations are imposed on the points of $W_{\sigma}$.
\end{rem}

Summing up all we obtain:

\begin{prop}
The strata for  $(\C ^{*})^5$-action on $G_{5,2}/T^5$, their number in each dimension and corresponding polytopes are given as follows:
\begin{equation}
\left[\begin{array}{cccccccccccc}
{\bf 12} & {\bf 10} & 8 & 8 & {\bf 8} & 6 & 6 & 6 & 4 & 4 & 2 & 0\\
{\bf 1} & {\bf 10} & 15 & 10 & {\bf 5} & 10 & 30 & 5& 15 & 30 & 30 & 10\\
{\bf \Delta _{5,2}} & {\bf K_{ij}(9)} & K_{ij,kl}(8) & K_{ij}(7) & {\bf O_{i}}(6) & P_{ij}(6) & P(6) & T_{i}(4) & P (4) & P(3) & P(2) & Ver
\end{array}\right].
\end{equation}
Apart from the bolded strata, all other strata consist of one orbit.
\end{prop}

\subsection{The facets of the admissible polytopes}
We describe here the facets of the four-dimensional admissible polytopes.  Using the description of the main stratum in the chart $M_{12}$ we obtain:

\begin{prop}
The facets of the  hypersimplex $\Delta _{5,2}$ are the octahedra $O_i$ and the  tetrahedra $T_i$, where $1\leq i\leq 5$.
\end{prop}

We describe first the facets of the admissible polytopes  $K_{ij}(9)$.  For the clearness of the exposition we demonstrate it for the polytope $ K_{13}(9)$. We denote by $P_{i}^{kl}$, $k,l\neq i$, the four sided prism spanned by the vertices of the octahedron $O_i$ excluding the 
vertex $\Lambda _{kl}$.  
\begin{prop}
The boundary of $K_{13}(9)$ consists of the octahedra $O_{1}$ and $O_{3}$, the  tetrahedra $T_2$, $T_4$ and $T_5$, the pyramids
$P_{2}^{13}$, $P_{4}^{13}$ and $P_{5}^{13}$ and the prism $P_{13}$ with the based triangles spanned by $\{\Lambda_{12},\Lambda_{14},\Lambda_{15}\}$ and $\{\Lambda_{23},\Lambda_{34},\Lambda_{35}\}$.
\end{prop}
\begin{proof}
The stratum $W_{\sigma}$, $\sigma =\{I\subset\{1,\ldots , 5\}, |I|=2\}\setminus \{1,3\}$,  which maps to $K_{13}(9)$ is, in the chart $M_{12}$, given by the following  system of equations $c^{'}z_2z_6=cz_3z_5$ and $z_4=0$, where $(c:c^{'})\in \C P^{1}_{A}$. The boundary of any $(\C ^{*})^{5}$-orbit from $W_{\sigma}$ consists of the $(\C ^{*})^{5}$-orbits of the  smaller dimensions and there is a bijection between these orbits and  the faces of $K_{13}(9)$. The orbits of  complex dimension $3$,   belonging to  the boundary of an  orbit from  $W_{\sigma}$ are given by the following conditions: $z_1=0$ ( its polytope is   $O_{3}$), $z_2=z_3=0$ (its polytope is  $P_{13}$),  $z_2=z_5=0$ ( its polytope is  $P_{4}^{13}$),  $z_3=z_6=0$ (its polytope is the pyramid $P_{5}^{13}$(, $z_5=z_6=0$ ( its polytope is $T_2$). The  stationary subgroup for $K_{13}(9)$ regarded to $S_5$-action  is $S_{13} = S_{2}\{1,3\}\times S_{3}\{2,4,5\}$. By the action of $S_{13}$ on those facets of $K_{13}(9)$, which belong  to  the chart $M_{12}$ we obtain  all  facets for $K_{13}(9)$. In this way in addition  we obtain  $O_1$, $P_{2}^{13}$, $T_4$ and $T_5$ to be the facets for $K_{13}(9)$. 
\end{proof}  
More generally:
\begin{prop}\label{bound-9}
The boundary of $K_{ij}(9)$ consists of the octahedra $O_{i}$ and $O_{j}$, the  tetrahedra $T_{k}$, $T_{l}$, $T_{m}$, the  pyramids $P_{k}^{ij}$, $P_{l}^{ij}$ and $P_{m}^{ij}$ and  the prism $P_{ij}$ with the based triangles spanned by $\{\Lambda_{ik},\Lambda_{il},\Lambda_{im}\}$ and $\{\Lambda_{jk},\Lambda_{jl},\Lambda_{jm}\}$, where $k,l,m\neq i,j$.
\end{prop}

In an analogous way we describe   the facets for the polytopes $K_{ij,kl}$ and $K_{ij}(7)$.

\begin{prop}\label{bound-8}
The boundary of the polytope $K_{ij,kl}$ consists of two prisms $P_{ij}$ and $P_{kl}$ with the based triangles spanned by $\{\Lambda_{is},\Lambda_{ip},\Lambda_{iq}\}$, $\{\Lambda_{js},\Lambda_{jp},\Lambda_{jq}\}$ and $\{\Lambda_{kr},\Lambda_{ku},\Lambda_{kv}\}$, $\{\Lambda_{lr},\Lambda_{lu},\Lambda_{lv}\}$, the  tetrahedron $T_{m}$ and the pyramids $P_{i}^{kl}$, $P_{j}^{kl}$, $P_{k}^{ij}$, $P_{l}^{ij}$, where $\{s,p,q\} =\{1,2,3,4,5\}\setminus \{i,j\}$, $\{r,u,v\}=\{1,2,3,4,5\}\setminus \{k,l\}$ and $m\neq i,j,k,l$ while $i,j\neq k,l$.
\end{prop}

\begin{rem}
An admissible polytope $K_{ij, kl}$ has eight vertices, among them  four vertices  are  simple, while the other four vertices have  five edges.
\end{rem}

\begin{prop}\label{bound-7}
The boundary of the polytope $K_{ij}(7)$ consists of the  prism $P_{ij}$ with the based triangles spanned by  $\{\Lambda_{ip},\Lambda_{iq},\Lambda_{ir}\}$, $\{\Lambda_{jp},\Lambda_{jq},\Lambda_{jr}\}$, the  tetrahedra $T_{i}$, $T_{j}$ and the pyramids $P_{p}^{qr}$, $P_{q}^{pr}$, $P_{r}^{pq}$,
where $\{p,q,r\}=\{1,2,3,4,5\}\setminus \{i,j\}$.
\end{prop}

\begin{cor}\label{bound-p6}
The prism $P_{ij}$ with the based triangles spanned by $\{\Lambda_{ip},\Lambda_{iq},\Lambda_{ir}\}$ and  $\{\Lambda_{jp},\Lambda_{jq},\Lambda_{jr}\}$ is the common interior facet for $K_{ij}(9)$, polytopes $K_{ij,kl}$ and pyramid $K_{ij}(7)$. It is the only facet for $K_{ij}(9)$ and $K_{ij}(7)$
that  belongs to the interior of $\stackrel{\circ}{\Delta}_{5,2}$. 
\end{cor}

Note that the polytopes $K_{il,kl}$, $K_{ij,pq}$ ,  $K_{ij,rs}$ , $K_{kl}(7)$, $K_{pq}(7)$,  $K_{rs}(7)$ are sub-polytopes of $K_{ij}(9)$ and that  any of the  pairs $(K_{ij,kl}, K_{kl}(7))$,  $(K_{ij,pq}, K_{pq}(7))$,  $(K_{ij,rs}, K_{rs}(7))$
 gives the polytopal decomposition of $K_{ij}(9)$. On the other hand,  the  pairs $(K_{ij}(9), K_{ij}(7))$ give the polytopal decomposition of $\Delta_{5,2}$, where $K_{ij}(9)$ and $K_{ij}(7)$ are glued together along the prism $P_{ij}$, which is their common interior  facet.

\section{The  idea of the universal space of parameters}
The  main stratum is an everywhere dense set in $G_{5,2}$ meaning that any other stratum is in the boundary of the main stratum. Using that fact our goal is to prove the following theorem.

\begin{thm}\label{allcharts}
There exists a topological space   $\tilde{\mathcal{F}}$ such that:
\begin{enumerate}
\item for  any chart $M_{ij}$ there is  a map $H_{ij} : \mathcal{A} \to \tilde{\mathcal{F}}$, $\sigma \to \tilde{F}_{\sigma, ij}$, where $\tilde{F}_{\sigma, ij}\subset \tilde{\mathcal{F}}$ is defined by: $c\in \tilde{F}_{\sigma, ij}$ if and only if  there exists a sequence $ (x_n, c_n)\subset \stackrel{\circ}{\Delta} _{5,2}\times F_{ij}$ such that $c_n$ converges to $c$ and  $h^{-1}(x_n, f_{ij}(c_n))$ converges to a  point from $W_{\sigma}/T^5$.
\item $\tilde{F}_{\sigma, ij}$ is homeomorphic to $\tilde{F}_{\sigma, kl}$  for any charts $M_{ij}, M_{kl}$.
\end{enumerate} 
\end{thm}
Here $\mathcal{A}$ is the set of all admissible sets, $h: W/T^5\to \stackrel{\circ}{\Delta _{5,2}}\times F$ is the canonical trivialization of the main stratum and $f_{ij} : F_{ij} \to F$ is the canonical  homeomorphism for the chart $M_{ij}$.

Let us explain the meaning of this theorem a little more.  The orbit space $G_{5,2}/T^5$ is a compactification of $W/T^5$  and the boundary of $W/T^5$ in $G_{5,2}/T^5$ is given by $ \cup _{W_{\sigma}\neq W} W_{\sigma}/T^{5}$.  Theorem~\ref{triv} states that for any $\sigma \in \mathcal{A}$,  
$W_{\sigma}/T^n $ is homeomorphic to  $\stackrel{\circ}{P}_{\sigma}\times F_{\sigma, ij}$.   Theorem~\ref{universal} claims that there exists a  corresponding compactification for $\stackrel{\circ}{\Delta}_{5,2}\times F_{ij}$  which is given by the subspaces $\tilde{F}_{\sigma, ij}\subset \tilde{\mathcal{F}}$,  that is a  compactification for $\stackrel{\circ}{\Delta}_{5,2}\times F_{ij} $ of the form $\cup _{\sigma}\stackrel{\circ}{P}_{\sigma}\times  \tilde{F}_{\sigma, ij}$.

\begin{defn}
The space $\tilde{\mathcal{F}}$ is called  the universal space of parameters for $T^{5}$-action on $G_{5,2}$.
\end{defn}
\begin{defn}
The spaces $\tilde{F}_{\sigma, ij}$ are said to be the virtual spaces of parameters for the strata $W_{\sigma}$ in the chart $M_{ij}$.
\end{defn}

Note that Theorem~\ref{universal} directly implies that $F_{ij}\subset \tilde{\mathcal{F}}$ for any chart $M_{ij}$. We know that $F_{ij}\subset \C P^1\times \C P^1\times \C P^1$ and we will prove  that the closure $\bar{F}_{ij}\subset \C P^1\times \C P^1 \times \C P^1$ has to belong to $\tilde{\mathcal{F}}$ as well. But,  as the consideration that follows will  show this closure is not enough, the  universal space of parameters turns out to be wider. 

We proceed with the proof of Theorem~\ref{universal} though  few steps. First, we describe the  subsets $\bar{F}_{\sigma, 12}\subset \bar{F}_{12}$  for the strata $W_{\sigma}$,   which belong to the chart $M_{12}$. The subsets $\bar{F}_{\sigma, 12}$ are   obtained  using  the fact that the main stratum is a  dense set in this chart . This description will show that in finding $\tilde{\mathcal{F}}$ we should start with  $\bar{F}\subset \C P^{1}\times \C P^1\times \C P^1$. Then we describe  the automorphisms of $F$ given by the  transition homeomorphisms  between $F_{ij}$ and $F_{kl}$ in  different charts  $M_{ij},  M_{kl}$. Finally we find the compactification $\tilde{\mathcal{F}}$ for $F$ such that any of these automorphism extends to an automorphism of $\tilde{\mathcal{F}}$. This compactification will satisfy the conditions of Theorem~\ref{universal}.

\section{The spaces  $\bar{F}_{\sigma, 12}$ for  the strata $W_{\sigma}$  in the chart $M_{12}$} 
We describe the spaces  $\bar{F}_{\sigma, 12}\subset \bar{F}_{12}$  for the strata which belong to the chart $M_{12}$.
\begin{prop}\label{parametrization-9}
The spaces $\bar{F}_{ij, 12}$ for the strata whose admissible polytopes are $K_{ij}(9)$ and which belong to the chart $M_{12}$,   are   $\bar{F}_{ij, 12}\subset \bar{F}_{12}\subset \C P^1\times \C P^1\times \C P^1$ given as  follows:
\begin{enumerate}
\item  $\bar{F}_{23, 12}\to  ((0:1), (0:1),  (c:c^{'})) $,
\item  $\bar{F}_{24, 12}\to  ((1:0), (c:c^{'}) , (0:1))$,
\item  $\bar{F}_{25, 12}\to ((c:c^{'}) , (1:0), (1:0))$,
\item $\bar{F}_{13, 12}\to ((1:0), (1:0), (c:c^{'}))$,
\item $\bar{F}_{14, 12}\to  ((0:1),  (c:c^{'}), (1:0))$,
\item $\bar{F}_{15, 12}\to  ((c:c^{'}),  (0:1), (0:1))$,
\item $\bar{F}_{34, 12}\to  ((1:1),(c:c^{'}), (c:c^{'}))$,
\item $\bar{F}_{35, 12}\to  ((c:c^{'}), (1:1), (c^{'}:c))$,
\item $\bar{F}_{45, 12}\to    ((c:c^{'}), (c:c^{'}), (1:1))$,
\end{enumerate}
where $(c:c^{'})\in \C P^{1}_{A}$.
\end{prop}
\begin{proof}
The  strata $(1)-(9)$ belong to the chart $M_{12}$ and the given list of their spaces of   parameters  follows from the description of these strata given by Proposition~\ref{Orbits}.The stratum $(1)$ is,  in the chart $M_{12}$,  given by the condition  $z_1=0$, $z_i\neq 0$, $i\neq 1$ and $\frac{z_2z_6}{z_3z_5}\neq 1$.  So, if  a sequence of points $(z_1^n,\ldots ,z_6^n)$ from the main stratum  converges to a given point from the stratum $(1)$,  it follows that $z_1^n\rightarrow 0$, while $z_i^n\rightarrow z_i\neq 0$. Thus,  $c_1^n=\frac{z_1^nz_5^n}{z_2^nz_4^n}\rightarrow 0$, $c_2^n=\frac{z_1^nz_6^n}{z_3^nz_4^n}\rightarrow 0$, while $\frac{c_1^n(c_1^n-1)}{c_2^n(c_2^n-1)}\rightarrow c$, where $c=\frac{z_2z_6}{z_3z_5}$.  Therefore the orbits from the stratum $(1)$ are continuously  parametrized by triples $(0,0,c)$, where $c\in \C - \{0,1\}$. In the same way, we deduce the parametrization for the orbits from the strata $(2)-(9)$. 
\end{proof} 

\begin{prop}\label{parametrization-8}
The  spaces  $\bar{F}_{(ij, kl), 12}$  for the strata whose admissible polytopes are $K_{ij, kl}$ and which belong to the chart $M_{12}$, are $\bar{F}_{(ij,kl), 12} \subset \bar{F}_{12}\subset \C P^{1}\times \C P^{1}\times \C P^{1}$  given  as follows:
\begin{itemize}
\item $\bar{F}_{(14, 23), 12} \to   ((0:1), (0:1), (1:0))$,\;\;  $\bar{F}_{(13,24), 12} \to  ((1:0), (1:0), (0:1))$, 
\item $\bar{F}_{(15,24), 12} \to   ((1:0), (0:1), (0:1))$,\;\; $\bar{F}_{(23,45), 12}\to  ((0:1), (0:1), (1:1))$,
\item $\bar{F}_{(24,35), 12}\to  ((1:0), (1:1), (0:1))$,\;\; $\bar{F}_{(25,34), 12}\to ((1:1), (1:0), (1:0))$, 
\item $\bar{F}_{(15, 23), 12}\to ((0:1), (0:1), (0:1))$,\;\;  $\bar{F}_{(13, 25), 12}\to ((1:0), (1:0), (1:0))$,
\item $\bar{F}_{(14,25), 12}\to ((0:1), (1:0), (1:0))$,\;\;  $\bar{F}_{(13,45), 12}\to ((1:0), (1:0), (1:1))$;
\item  $\bar{F}_{(14,35), 12}\to ((0:1), (1:1), (1:0))$,\;\;  $\bar{F}_{(15,34), 12}\to ((1:1), (0:1), (0:1))$.
\end{itemize}
\end{prop}
\begin{prop}\label{parametrization-7}
The spaces $\bar{F}_{ij, 12}(7)$  for the strata  whose admissible polytopes are $K_{ij}(7)$ and which belong to the chart $M_{12}$,  are   $\bar{F}_{ij, 12}(7)\subset \bar{F}_{12}\subset \C P^{1}\times \C P^{1}\times \C P^{1}$   given  as follows:
\begin{itemize}
\item $\bar{F}_{23, 12}(7)\to  (0,1)\times (0:1)\times \C P^1$,\;\; $\bar{F}_{24, 12}(7)\to  (1:0)\times \C P^{1}\times (0:1)$,
\item $\bar{F}_{25, 12}(7)\to \C P^{1}\times (1:0)\times (1:0)$,\;\; $\bar{F}_{13, 12}(7)\to  (1:0)\times (1:0)\times \C P^1$,
\item $\bar{F}_{14, 12}(7)\to  (0:1)\times \C P^1\times (1:0)$,\;\;  $\bar{F}_{15, 12}(7)\to \C P^1\times (0:1)\times (0:1)$,
\item $\bar{F}_{12, 12}(7)\to ((1:1), (1:1), (1:1))$.
\end{itemize}
\end{prop}
\begin{prop}\label{parametrization-6}
The  spaces $\bar{F}_{i,12}$ for the strata  whose admissible polytopes are $O_{i}$ and which belong to the chart $M_{12}$, are the subsets  $\bar{F}_{i, 12}\subset \bar{F}_{12}\subset \C P^{1}\times \C P^{1}\times \C P^{1}$  given   as follows:
\begin{itemize}
\item $\bar{F}_{3, 12}\to \bar{F}_{12}\cap ( \C P^{1}\times \C P^{1}\times \C P^{1}_{A})$,
\item $\bar{F}_{4, 12}\to \bar{F}_{12}\cap (\C P^{1}\times  \C P^{1}_{A}\times \C P^1)$,
\item $\bar{F}_{5, 12}\to \bar{F}_{12}\cap (\C P^{1}_{A}\times \C P^1\times \C P^1)$.
\end{itemize}
\end{prop}

\begin{prop}\label{parametrization-5}
The spaces $\bar{F}_{ij,12}(6)$  for the strata  whose admissible polytopes are $P_{ij}$ and which belong to the chart $M_{12}$,  are   $\bar{F}_{ij, 12}(6)\subset \bar{F}_{12}\subset \C P^{1}\times \C P^{1}\times \C P^{1}$   given as  as follows:
\begin{itemize}
\item $\bar{F}_{23, 12}(6)\to  (0,1)\times (0:1)\times \C P^1$,\;\; $\bar{F}_{24, 12}(6)\to  (1:0)\times \C P^{1}\times (0:1)$,
\item $\bar{F}_{25, 12}(6)\to \C P^{1}\times (1:0)\times (1:0)$,\;\; $\bar{F}_{13, 12}(6)\to  (1:0)\times (1:0)\times \C P^1$,
\item $\bar{F}_{14, 12}(6)\to  (0:1)\times \C P^1\times (1:0)$,\;\;  $\bar{F}_{15, 12}(6)\to \C P^1\times (0:1)\times (0:1)$,
\end{itemize}
\end{prop}

\begin{rem}\label{oct}
It holds $\bar{F}_{3,12} =  \bar{F}_{12}\cap  (\C P^{1}\times \C P^{1}\times \C P^{1}_{A}) = F_{12}\cup \bar{F}_{13,12}\cup \bar{F}_{23,12}\cup \bar{F}_{34,12}\cup \bar{F}_{35,12}$ and this  corresponds to the fact that $O_3$ is a facet for each of $\Delta _{5,2}$, $K_{13}(9)$, $K_{23}(9)$, $K_{34}(9)$ and  $K_{35}(9)$. In other words orbits from the  stratum  whose admissible polytope is $O_3$ are  in the boundary of the orbits of the  strata  with the admissible polytopes  $\Delta _{5,2}$, $K_{13}(9)$, $K_{23}(9)$, $K_{34}(9)$ and $K_{35}(9)$. Then using  the action of the symmetric group $S_5$ we immediately obtain  the description for $\bar{F}_{4,12}$ and $\bar{F}_{5,12}$.
\end{rem} 
In an analogous way we can describe the  spaces of parameters $\bar{F}_{\sigma, 12}\subset \bar{F}_{12}$ for any stratum $W_{\sigma}\subset M_{12}$. We do not find necessary to list all of them here since  the admissible polytopes for all strata,  which are not listed in the previous propositions,  are  faces of $\Delta _{5,2}$ and their spaces of parameters are homeomorphic to a point. 
 
\begin{rem}
The previous propositions suggest  that in  order to find an universal space of parameters we should start with  the closure of the space of parameters $F$ in $\C P^1\times \C P^1\times \C P^1$.
\end{rem} 
\begin{rem}\label{strata137}
 The stratum  whose admissible polytope is $K_{13}(7)$ is, in the chart  $M_{12}$, parametrized by $((1:1), (1:1), (1:1))\in \bar{F}_{12}$ according to  Proposition~\ref{parametrization-7}. This stratum belongs to the chart $M_{13}$ as well and, using the same argument as above,  it can be  seen  that,  in that chart $M_{13}$, this stratum  is parametrized by $(1:0)\times (1:0)\times \C P^1\subset \bar{F}_{13}$. Since  the virtual spaces of parameters $\tilde{F}_{\sigma, ij}$ of a  stratum $W_{\sigma}$  in different charts should  be homeomorphic, this shows  that  the closure  $\bar {F}$ of $F$  in $\C P^1\times \C P^1\times \C P^1$  can not be taken as an  universal space  of parameters.
\end{rem}

\section{The universal space of parameters $\tilde{\mathcal{F}}$}
\subsection{The closure $\mathcal{F}$  of $F$ in $\C P^1\times \C P^1\times \C P^1$}
The closure of the space $F_{kl}\subset  \C P^1\times \C P^1\times \C P^1$,  which represent the space of  parameters of the main stratum in the chart $M_{kl}$, should  not depend on the chart $M_{kl}$, since all our constructions are invariant under the action of the symmetric group $S_5$. The closure of $F_{kl}$ in $\C P^1\times \C P^1\times \C P^1$ is:
\[
\bar{F}_{kl} = \{((c_{1,kl}:c_{1,kl}^{'}), (c_{2,kl}:c_{2,kl}^{'}),( c_{3,kl}:c_{3,kl}^{'})) \; | \; c_{1,kl}c_{2,kl}^{'}c_{3,kl} = c_{1,kl}^{'}c_{2,kl}c_{3,kl}^{'}\}.
\]
Because of our further purposes we describe this closure in more detail. We obtain the  points from the boundary of  $F_{kl}$    if we put in the cubic equation $c_{1,kl}c_{2,kl}^{'}c_{3,kl}=c_{1,kl}^{'}c_{2,kl}c_{3,kl}^{'}$ that some $c_{i,kl}=0$, $c_{i,kl}^{'}=0$ or $c_{i,kl}=c_{i,kl}^{'}$,
but keeping that $c_{i, kl}, c_{i,kl}^{'}$ may not be both zero, where $1\leq i\leq 3$.   The explicit description of all such points  is: 
\begin{lem}
The boundary of $F_{kl}\subset \C P^{1}\times \C P^{1}\times \C P^{1}$ consists of the following sets:
\[
\bar{F}_{kl}^{12}= (0:1)\times (0:1)\times \C P^1,\;\;\bar{F}_{kl}^{13^{'}}=
 (0:1)\times \C P^1\times (1:0),
\]
\[ \bar{F}_{kl}^{1^{'}2^{'}}= (1:0)\times (1:0)\times \C P^1,\;\;
\bar{F}_{kl}^{1^{'}3}= (1:0)\times \C P^1\times (0:1),
\]
\[ \bar{F}_{kl}^{23}=  \C P^{1}\times  (0:1)\times (0:1),\;\; 
 \bar{F}_{kl}^{2^{'}3^{'}}= \C P^{1}\times (1:0)\times (1:0),
\]
\[
 \bar{F}_{kl}^{11^{'}}=\{( (1:1), (c_{2,kl}:c_{2,kl}^{'}), (c_{2,kl}: c_{2,kl}^{'}))\},\;\; 
\bar{F}_{kl}^{22^{'}}=\{((c_{1,kl}:c_{1,kl}^{'}), (1:1), (c_{1,kl}^{'}:c_{1,kl}^{'}))\},
\]
 \[
\bar{F}_{kl}^{33^{'}}= \{((c_{1,kl}:c_{1,kl}^{'}), (c_{1,kl}:c_{1,kl}^{'}), (1:1))\}.
\]
\end{lem}
\begin{proof}
We consider the  cubic equation in~\eqref{mainparamar} and   directly deduce the following: 
 $c_{1,kl}=c_{2,kl}=0$ gives  $\bar{F}_{kl}^{12}$, 
 $c_{1,kl}=c_{3,kl}^{'}=0$ gives $\bar{F}_{kl}^{13^{'}}$,
$c_{1,kl}^{'}=c_{2,kl}^{'}=0$ gives $\bar{F}_{kl}^{1^{'}2^{'}}$,
$c_{1,kl}^{'}=c_{3,kl}=0$ gives $\bar{F}_{kl}^{1^{'}3}$,
$c_{2,kl}=c_{3,kl}=0$ gives $\bar{F}_{kl}^{23}$ and 
$c_{2,kl}^{'}=c_{3,kl}^{'}=0$ gives $\bar{F}_{kl}^{2^{'}3^{'}}$.
 For 
$c_{1,kl}=c_{1,kl}^{'} = 1$, the cubic equation gives that $c_{2,kl}^{'}c_{3,kl} = c_{2,kl}c_{3,kl}^{'}$,   which implies that $(c_{3,kl}:c_{3,kl}^{'}) = (c_{3,kl}: \frac{c_{2,kl}^{'}c_{3,kl}}{c_{2,kl}})= (c_{2,kl}:c_{2,kl}^{'})$, so we obtain  $\bar{F}_{kl}^{11^{'}}$. In the same way 
$c_{2,kl}=c_{2,kl}^{'}=1$ gives $ \bar{F}_{kl}^{22^{'}}$ and 
 $c_{3,kl}=c_{3,kl}^{1}  =1$ gives  $ \bar{F}_{kl}^{33^{'}}$.
\end{proof}

The  non-trivial  intersections of these  boundary sets are as follows: 
\[
\bar{F}_{kl}^{12} \cap  \bar{F}_{kl}^{13^{'}} = ((0:1), (0:1), (1:0)),\;\;  \bar{F}_{kl}^{12}\cap \bar{F}_{kl}^{23} = ((0:1), (0:1),  (0:1)),
\]
\[
 \bar{F}_{kl}^{13^{'}}\cap  \bar{F}_{kl}^{2^{'}3^{'}} = ((0:1), (1:0), (1:0)) , \;\;
 \bar{F}_{kl}^{1^{'}2^{'}}\cap  \bar{F}_{kl}^{1^{'}3} = ((1:0), (1:0), (0:1)) ,
\]
\[
 \bar{F}_{kl}^{1^{'}2^{'}}\cap  \bar{F}_{kl}^{2^{'}3^{'}}=  ((1:0), (1:0), (1:0)) ,\;\; \bar{F}_{kl}^{1^{'}3}\cap \bar{F}_{kl}^{23} = ((1:0), (0:1), (0:1)),
\]
\[
 \bar{F}_{kl}^{1^{'}2^{'}}\cap  \bar{F}_{kl}^{33^{'}}= ( (1:0), (1:0), (1:1)),\;\;   \bar{F}_{kl}^{23}\cap \bar{F}_{kl}^{11^{'}} = ((1:1), (0:1), (0:1)),
\]
\[
 \bar{F}_{kl}^{2^{'}3^{'}}\cap \bar{F}_{kl}^{11^{'}} = ((1:1), (1:0), (1:0)),\;\; \bar{F}_{kl}^{11^{'}}\cap \bar{F}_{kl}^{22^{'}}\cap \bar{F}_{kl}^{33^{'}} = ((1:1),  (1:1), (1:1)) .
\]

\subsection{The space $\tilde{\mathcal{F}}$ as the blowup of $\mathcal{F}$}
In order to resolve the problem indicated in  Remark~\ref{strata137}  we will blowup the cubic surface  $\bar{F}_{ij}$ at the point $((1:1), (1:1), (1:1))$. We  denote by  $\mathcal{F}$  the  hypersurface in $\C P^1 \times \C P^1\times \C P^1$ defined by
\[
\mathcal{F} = \{( (c_1:c_1^{'}), (c_{2}:c_{2}^{'}), (c_3:c_{3}^{'})) | \;  c_1c_{2}^{'}c_{3} = c_{1}^{'}c_{2}c_{3}^{'}\}.
\]
Since the gradient of the cubic equation $c_1c_{2}^{'}c_{3} = c_{1}^{'}c_{2}c_{3}^{'}$, which defines this hypersurface, is non-zero,  it holds:
\begin{lem}
$\mathcal{F}$ is a smooth manifold.
\end{lem}

We prove the following:
\begin{thm}\label{universal}
An universal space of parameters $\tilde{\mathcal{F}}$ for $G_{5,2}$ can be obtained as the blowup of $\mathcal{F}$ at the point $((1:1), (1:1), (1:1))$.
\end{thm}

Before to proceed with the proof of Theorem~\ref{universal} let us discuss  the space $\tilde{\mathcal{F}}$ in more detail.

\begin{cor}
$\tilde{\mathcal{F}}$ is a smooth compact four-manifold.
\end{cor}

We describe this blowup in an neighborhood of the point $((1:1), (1:1), (1:1))$.

\begin{lem}\label{blowlem}
The blowup of the complex surface $\mathcal{F} = \{ ((c_1:c_1^{'}), (c_2: c_{2}^{'}), (c_3:c_3^{'})), \; c_1c_2^{'}c_3=c_1^{'}c_2c_{3}^{'}\}$ in an neighborhood of  the point $((1:1), (1:1), (1:1))$ is the surface $\tilde {\mathcal{F}}\subset \mathcal{F}\times \C P^1$  defined by the equation: 
\begin{equation}\label{blow}
 (1-c_1^{'})x_2 = (1-c_{2}^{'})x_1,
\end{equation}
where $(x_1:x_2)\in \C P^1$.
\end{lem}

\begin{proof}
We proceed in a standard way for doing the blowup at the point
$((1:1), (1:1), (1:1))$. Consider a  neighborhood $U$  of this point of the form $((1:c_1^{'}), (1:c_2^{'}), (1:c_3^{'}))$, where $c_1^{'}\neq 0$.  The cubic equation  implies that $c_2^{'} = c_{1}^{'}c_{3}^{'}$, so, since we assume $c_{1}^{'}\neq 0$  in this neighborhood,  we can take    $(c_1^{'}, c_{2}^{'})$ as coordinates in $U$.  Then the preimage  of $U\setminus\{((1:1), (1:1), (1:1))\}$  in $\tilde{\mathcal{F}}\setminus \C P^{1}$ is an open submanifold of $\mathcal{F}\times \C P^1$ given by the equation  $(1-c_1^{'})x_2 = (1-c_{2}^{'})x_1$. This proves~\eqref{blow}.
\end{proof}

\begin{cor}
The open manifold $\tilde{\mathcal{F}}\setminus \{((0:1), (0:1), (c_3:c_3^{'}))\}$ is given by the equation
\[
(c_1-c_1^{'})c_2c^{'} = c_1(c_2-c_2^{'})c.
\]
\end{cor}

\begin{rem}\label{blowgen}
The  blowup construction implies  that the projection map  on the first coordinate $\pi : \tilde{\mathcal{F}} \to \mathcal{F}$ is an isomorphism between $\tilde{\mathcal{F}}\setminus \{((1:1), (1:1), (1:1))\}\times \C P^1$ and $\mathcal{F}\setminus \{((1:1), (1:1), (1:1))\}$, and $\pi ^{-1}((1:1), (1:1), (1:1)) = \C P^1$.
\end{rem}

\begin{rem}
In the sequel, we write the points from $\tilde{\mathcal{F}}\setminus \C P^1$ in  coordinates of the manifold $\mathcal{F}$ and the points from the divisor  $\C P^1\subset \tilde{\mathcal{F}}$ in the form $(((1:1), (1:1), (1:1)), (c_3:c_3^{'}))$.
\end{rem}

\begin{thm}\label{bl3}
The universal space of parameters $\tilde{\mathcal{F}}$  is homeomorphic to the space $Y$ which is the  blowup of $\C P^1\times \C P^1$ at the points $B_1= ((1:0), (1:0))$, $B_2 = ((0:1), (0:1))$ and $B_3 = ((1:1), (1:1))$.
\end{thm}
\begin{proof}
The homeomorphism $\tilde{\mathcal{F}}\to Y$ is given by
$((c_1:c_1^{'}), (c_2:c_2^{'}), (c_3:c_3^{'})) \to ((c_1:c_1^{'}), (c_2:c_2^{'}))$  at the points  $((c_1:c_1^{'}), (c_2:c_2^{'}), (c_3:c_3^{'}))$ for which $((c_1:c_1^{'}), (c_2:c_2^{'}))\neq B_1, B_2, B_3$. For the other  points,  this homeomorphism  is given by
$((1:0), (1:0), (c_3:c_3^{'}))\to (c_3:c_3^{'})\in \text{blowup}(B_1)\cong \C P^1$,  $((0:1), (0:1), (c_3:c_3^{'}))\to (c_3:c_3^{'})\in \text{blowup}(B_2)\cong \C P^1$ and $((1:1), (1:1), (1:1),(c_3:c_3^{'}))\to (c_3:c_3^{'})\in \text{blowup}(B_3)\cong \C P^1$.
\end{proof}

We provide an  interpretation of the universal space of parameters in one more way. Let 
\[
U=(\mathbb{C}P^1\times \mathbb{C}P^1)\backslash\{ ( (c_1:c_1^{'}),(1:0) ), ( (1:0),(c_2:c_2^{'}) ) \}, \quad V=\mathbb{C}P^2\backslash \{(x:y:0)\}.
\]
The following obviously holds:
\begin{lem}\label{3-4}
The map
\begin{equation}\label{map}
f \colon \mathbb{C}P^1\times \mathbb{C}P^1\longrightarrow \mathbb{C}P^2, \;\; f ((c_1:c_1^{'}),(c_2:c_2^{'}))  = (c_1c_2^{'}:c_1^{'}c_2 : c_1^{'}c_2^{'})
\end{equation}
has the following properties:
\begin{enumerate}
\item  $f$ is defined everywhere except and the point  $((1:0),(1:0))$.
\item The image of  $f$ is $V$.
\item The space of parameters   $F$ of the main stratum  belongs to $U$.
\item  $f$ defines a homeomorphism $f \colon U \to V$.
\end{enumerate}
\end{lem}
 
Lemma~\ref{3-4} implies that the map $f$ induces a homeomorphism between the blowup of $\C P^1\times \C P^1$ at three points and the blow up of $\C P^2$ at four points:

\begin{cor}\label{CP2}
The universal space of parameters $\tilde{\mathcal{F}}$  is homeomorphic to the space $Y$ obtained as the   blowup of $\C P^{2}$ at four points $(1:0:0),(0:1:0), (0:0:1), (1:1:1)$.
\end{cor}
\begin{proof}
Let us consider the blowup $\C P^1\times \C P^1$ at the point $((1:0), (1:0))$. We extend the  map~\eqref{map} to the map $f$ from this blowup  to $\C P^2$ by mapping  the blowup$((1:0), (1:0))\cong \C P^1$    to $(x:y:0)\subset \C P^2$. Then we consider the blowups  of $\C P^2$ at the points $(1:0:0)$ and $(0:1:0)$ and  map to these blowups  the sets $\{( (1:0),(c_2:c_2^{'}))\}$and $ \{((c_1:c_1^{'}),(1:0))\}$,  respectively. We obtain a homeomorphism between the blowup of $\C P^1\times \C P^1$ at one  point and the blowup of $\C P^{2}$ at two points. Now,  in addition, we do the blowup of $\C P^1\times \C P^1$ at the points $((0:1), (0:1))$ and $((1:1), (1:1))$ and obtain the universal space of parameters. We accordingly do the blowup of $\C P^2$ at the points $(0:0:1)$ and $(1:1:1)$ and  map to them these two new blowups of $\C P^1\times \C P^1$. As a result we obtain the homeomorphism between the blowup of $\C P^1\times\C P^1$at three points and the blowup of $\C P^2$ at four points.    
\end{proof}

\begin{rem}
The product  $\C P^1\times \C P^1$ can be embedded into $\C P^3$ by the Serge map. It follows from~\eqref{map} that there exists a mapping from $\C P^1\times \C P^1$ to $\C P^2$ which is not defined  at one point. The proof of Corollary~\ref{CP2} shows that, after blowing up of  $\C P^1\times \C P^1$ at that undefined  point and blowing up $\C P^2$ at two points,  we obtain a homeomorphism of these new spaces.    
We want to point that this is a special case of  the result from  toric topology~\cite{Fulton},  which states that the blowup of $\C P^1\times \C P^1$ at $k$ points is homeomorphic to the blowup of  $\C P^2$ at $k+1$ points for $k>1$.  The proof of Corollary~\ref{CP2} provides an explicit description of this  homeomorphism for $k=2$.   
\end{rem}

\begin{rem}
In his seminal paper Kapranov~\cite{Kap} defined  and studied the properties of the Chow quotient $G_{n,k}/\!\!/(\C ^{*})^{n}$. It is   established an isomorphism between  $G_{n, 2}/\!\!/(\C ^{*})^{n}$ and the Grothendieck-Knudsen  moduli space $\overline{M_{0,n}}$.  Moreover,  it is proved (Theorem 4.3.3,~\cite{Kap})  that the variety $\overline{M_{0,n}}$  can be obtained from $\C P^{n-3}$  by  series of blowups.  In the first non-trivial case when $n=5$,  it implies that   $\overline{M_{0,5}}$  is the  blowup of $\C P^2$ at four points. In addition,  as it is pointed in~\cite{Kap}, in the paper~\cite{Keel} one can find  a representation of $\overline{M_{0,n}}$ as an iterated blowup of $(\C P^{1})^{n-3}$.  From this representation (see~\cite{Keel}, page 555), it directly follows  that $\overline{M_{0,5}}$ is isomorphic to a blow up of $\C P^1\times \C P^1$ at three points. Thus, by Theorem~\ref{bl3} and Corollary~\ref{CP2}  we see in both ways  that the universal space of parameters $\tilde{\mathcal{F}}$ for $G_{5,2}$ is isomorphic to the Chow quotient $G_{5,2}/\!\!/(\C ^{*})^{5}$.
\end{rem}

\section{The transition automorphisms for  the spaces of parameters   of the  strata}

Consider the chart $M_{kl}$ and let $z_{i}^{kl}$, $1\leq i\leq 6$  be the coordinates in this chart.  The main stratum  in this  chart is given by the system of equations
\begin{equation}\label{mainarbitrary}
c_{1,kl}^{'}z_{1}^{kl}z_{5}^{kl} = c_{1,kl}z_{2}^{kl}z_{4}^{kl},\;\; c_{2,kl}^{'}z_{1}^{kl}z_{6}^{kl} = c_{2,kl}z_{3}^{kl}z_{4}^{kl},\;\; c_{3,kl}^{'}z_{2}^{kl}z_{6}^{kl} = c_{3.kl}z_{3}^{kl}z_{5}^{kl}.
\end{equation}
The set of parameters  $F_{kl}\cong F$  of the main stratum in the coordinates of the chart $M_{kl}$ is given by
\begin{equation}\label{mainparamar}
F_{kl} = \{((c_{1, kl}:c^{'}_{1,kl}), (c_{2,kl}:c_{2, kl}^{'}), (c_{3,kl}:c_{3,kl}^{'}))\in \C P^1\times\C P^{1}\times \C P^{1}\},
\end{equation}
\[
c_{i, kl}, c_{i,kl}^{'}\neq 0, \;\; c_{i, kl}\neq c_{i, kl}^{'},\;\; c_{1,kl}c_{2,kl}^{'}c_{3,kl}=c_{1,kl}^{'}c_{2,kl}c_{3,kl}^{'}.
\]
We have that
\[
c_{3,kl}^{'} =  \frac{c_{1,kl}c_{2,kl}^{'}}{c_{1,kl}^{'}c_{2,kl}}c_{3, kl},
\]
which implies 
\[
(c_{3,kl}:c_{3,kl}^{'}) =  (c_{3,kl}:\frac{c_{1,kl}c_{2,kl}^{'}}{c_{1,kl}^{'}c_{2,kl}}c_{3, kl}) = (c_{1,kl}^{'}c_{2, kl} : c_{1,kl}c_{2,kl}^{'}).
\]
Thus,  we can take
\begin{equation}\label{tri}
c_{3,kl} = c_{1,kl}^{'}c_{2,kl},\;\; c_{3,kl}^{'} = c_{1,kl}c_{2,kl}^{'}.
\end{equation}

\subsection{The transtion automorphisms for  the space of  parameters $F$ of the main stratum}\label{mainstratum}
Consider now the charts $M_{12}$ and $M_{13}$. 
\begin{prop}
The homeomorphism $f_{12,13} : F_{12} \to F_{13}$ is given by
\begin{equation}\label{1213}
((c_{1, 12}:c^{'}_{1,12}), (c_{2,12}:c_{2, 12}^{'}), (c_{3,12}:c_{3,12}^{'})) \to
\end{equation}
\[
\to ((c_{1, 12}:c_{1,12}-c^{'}_{1,12}), (c_{2,12}:c_{2,12}-c_{2, 12}^{'}), ( (c_{1,12}-c_{1,12}^{'})c_{2,12}^{'}c_{3,12}:c_{1,12}^{'}(c_{2,12}-c_{2,12}^{'})c_{3,12}^{'})  ).
\] 
\end{prop}

\begin{proof}
 The coordinates $z_{i}^{12}$ and $z_{i}^{13}$, $1\leq i\leq 6$ are, on $M_{12}\cap M_{13}$,  related by
\begin{equation}\label{coordinatechange}
z_{1}^{13} = -\frac{z_{1}^{12}}{z_{4}^{12}},\;\; z_{2}^{13}=z_{2}^{12}-\frac{z_{1}^{12}}{z_{4}^{12}}z_{5}^{12},\;\; z_{3}^{13}=z_{3}^{12}-\frac{z_{1}^{12}}{z_{4}^{12}}z_{6}^{12}. 
\end{equation}
\[
z_{4}^{13}=\frac{1}{z_{4}^{12}},\;\; z_{5}^{13}=\frac{z_{5}^{12}}{z_{4}^{12}},\;\; z_{6}^{13}=\frac{z_{6}^{12}}{z_{4}^{12}}.
\]
It follows   that 
\begin{equation}\label{formulas}
z_{1}^{13}z_{5}^{13} = -\frac{z_1^{12}z_{5}^{12}}{(z_{4}^{12})^{2}},\;\; z_{2}^{13}z_{4}^{13} = \frac{z_{2}^{12}z_{4}^{12} -z_1^{12}z_{5}^{12}}{(z_{4}^{12})^{2}},\;\; z_{1}^{13}z_{6}^{13} = -\frac{z_1^{12}z_{6}^{12}}{(z_{4}^{12})^{2}},
\end{equation}
\[z_{3}^{13}z_{4}^{13} = \frac{z_{3}^{12}z_{4}^{12} -z_1^{12}z_{6}^{12}}{(z_{4}^{12})^{2}},\;\;\; z_{2}^{13}z_{6}^{13} = z_{6}^{12}\frac{z_{2}^{12}z_{4}^{12} -z_1^{12}z_{5}^{12}}{(z_{4}^{12})^{2}},\;\; \; z_{3}^{13}z_{5}^{13} = z_5^{12}\frac{z_{3}^{12}z_{4}^{12} -z_1^{12}z_{6}^{12}}{(z_{4}^{12})^{2}}.
\]
Therefore,  the relation $c_{1,13}^{'}z_{1}^{13}z_{5}^{13} = c_{1,13}z_{2}^{13}z_{4}^{13}$ can be written as
\[
c_{1,13}^{'}z_{1}^{12}z_{5}^{12} = c_{1,13}(z_{1}^{12}z_{5}^{12}-z_{2}^{12}z_{5}^{12}) \Rightarrow (c_{1,13}-c_{1,13}^{'})z_{1}^{12}z_{5}^{12} = c_{1,13}z_{2}^{12}z_{4}^{12}.
\]
It follows from~\eqref{mainarbitrary}  that 
\[  c_{1,13} = c_{1,12},\;\;  c_{1,13}^{'} =    c_{1,12} -  c_{1,12}^{'}.
\]
In the same way we deduce that 
\[
c_{2,13}= c_{2,12},\;\; c_{2,13}^{'} = c_{2,12}-c_{2,12}^{'}.
\]

From~\eqref{mainarbitrary} we also deduce that
\[
\frac{c_{1,kl}}{c_{1,kl}^{'}} = \frac{z_1^{kl}z_{5}^{kl}}{z_{2}^{kl}z_{4}^{kl}},\;\;\frac{c_{2,kl}}{c_{2,kl}^{'}} = \frac{z_1^{kl}z_{6}^{kl}}{z_3^{kl}z_{4}^{kl}},\;\; \frac{c_{3,kl}}{c_{3,kl}^{'}} = \frac{z_2^{kl}z_{6}^{kl}}{z_3^{kl}z_{5}^{kl}}. 
\]
Taking into account~\eqref{formulas}, it follows that 
\[
c_{3,13}^{'} = \frac{z_{5}^{12}}{z_{6}^{12}}\frac{z_3^{12}z_{4}^{12}-z_{1}^{12}z_{6}^{12}}{z_{2}^{12}z_{4}^{12}-z_{1}^{12}z_{5}^{12}}c_{3,13} =
\]
\[
\frac{z_{5}^{12}}{z_{6}^{12}}\frac{z_{3}^{12}z_{4}^{12}}{z_{2}^{12}z_{4}^{12}}\frac{\frac{c_{2,12}}{c_{2,12}^{'}}-1}{\frac{c_{1,12}}{c_{2,12}^{'}}-1} c_{3,13} = \frac{c_{1,12}^{'}(c_{2,12}-c_{2,12}^{'})c_{3,12}^{'}}{(c_{1,12}-c_{1,12}^{'}
)c_{2,12}^{'}c_{3,12}}c_{3,13}.
\]
Thus,  
\[
(c_{3,13} : c_{3,13}^{'}) = ( (c_{1,12}-c_{1,12}^{'})c_{2,12}^{'}c_{3,12}:c_{1,12}^{'}(c_{2,12}-c_{2,12}^{'})c_{3,12}^{'}).
\]
\end{proof}

Due to the fact that the group $S_5$ permutes the charts it can be similarly   explicitly constructed a homeomorphism $f_{ij, kl} : F_{ij} \to F_{kl}$ between the sets of parameters  $F_{ij}$ and $F_{kl}$  of the main stratum for an arbitrary two charts $M_{ij}$ and $M_{kl}$.

Since  the homeomorphism $f_{ij, kl} : F_{ij}\to F_{kl}$  is induced by the coordinate transition    map for  the charts $M_{ij}$ and $M_{kl}$, we deduce:

\begin{lem}\label{tricharts}
For any three charts $M_{ij}, M_{kl}$ and $M_{mn}$  it holds
\[
f_{ij, kl} =  f_{mn, kl}\circ  f_{ij, mn}.
\]
\end{lem}

\begin{cor}
The automorphisms $f_{ij,kl}$ of the space of parameters $F$  of the main stratum  for $G_{5,2}$, which are  induced by the coordinate transition  maps  between all  charts,  
form a group. The  generators for  this group  are  given by the set $\{ f_{i_{0}j_{0}, kl}, i_0j_0\neq kl\}$ for any fixed $i_0j_0$.
\end{cor}

\subsection{The transition automorphisms for  the spaces of  parameters of  other  non one-orbit strata}\label{non-orbit}
Let  $W_{\sigma }$ be a non-orbit stratum,  which is different from the main stratum. The stratum $W_{\sigma}$  does not belong to all charts. We determine here the relation between the coordinate records of  the set of parameters $F_{\sigma}$ for $W_{\sigma }$ in  different charts  which contain $W_{\sigma}$.
We demonstrate  this  for  one stratum,  which belongs to the charts $M_{12}$ and $M_{13}$. Due to an action of $S_5$,  the  similar relation   will hold for an arbitrary stratum and arbitrary   charts.  The intersections of the charts $M_{12}$ and $M_{13}$ is, in the coordinates of these charts,  given by the condition  $z_{4}^{12}\neq 0$ and $z_{4}^{13}\neq 0$.
Consider the  stratum whose admissible polytope is $P_{\sigma}=K_{34}(9)$,  which  belongs to the both of these charts.  The stratum  $W_{\sigma}$ is, in the chart $M_{12}$, given by 
\[
z_1^{12}z_{5}^{12} = z_{2}^{12}z_{4}^{12},\;\; c_{2, 12}^{'} z_{1}^{12}z_{6}^{12} = c_{2,12}z_{3}^{12}z_{4}^{12},\;\; c_{3,12}^{'}z_{2}^{12}z_{6}^{12} = c_{3,12}z_{3}^{12}z_{5}^{12},
\]
\[
c_{2, 12},c_{2,12}^{'}\neq 0,\;\; c_{2,12}\neq c_{2,12}^{'},\;\;  c_{3, 12},c_{3,12}^{'}\neq 0,\;\; c_{3,12}\neq c_{3,12}^{'}.\;\; c_{2,12}^{'}c_{3,12} = c_{2,12}c_{3,12}^{'}.
\]
It follows that 
\[
c_{3,12}^{'} = \frac{c_{2,12}^{'}c_{3,12}}{c_{2,12}},
\]
which  implies 
\[
(c_{3,12}:c_{3,12}^{'}) = (c_{2,12}:c_{2,12}^{'}) = (c_{12} : c_{12}^{'}).
\]
Thus,  the set of parameters for the stratum $W_{\sigma}$ is, in the coordinates for $M_{12}$,  given by
\[
F_{\sigma, 12}= \{(c_{12}: c_{12}^{'}), \;\; c_{12},c_{12}^{'}\neq 0,\;\; c_{12}\neq c_{12}^{'}\}.
\]

The stratum $W_{\sigma}$  is,   in the chart $M_{13}$,   given by the  equations:
\begin{equation}\label{chart13}
z_{2}^{13}=0, \;\; c_{2,13}^{'}z_{1}^{13}z_{6}^{13} = c_{2,13}z_{3}^{13}z_{4}^{13},\;\; c_{2,13}, c_{2,13}^{'}\neq 0,\;\; c_{2,13}\neq c_{2,13}^{'}.
\end{equation}
Thus,  the  set of parameters for $W_{\sigma}$ is,  in the chart $M_{13}$,  given by
\[
F_{\sigma, 13}= \{(c_{13}: c_{13}^{'}), \;\; c_{13},c_{13}^{'}\neq 0,\;\; c_{13}\neq c_{13}^{'}\}.
\]
 Substituting the formulas~\eqref{coordinatechange} into~\eqref{chart13} we obtain that a homeomorphism between $F_{\sigma, 12}$ and $F_{\sigma, 13}$  is given by:
\begin{equation}\label{35}
(c_{12}:c_{12}^{'}) \to (c_{12}: c_{12}-c_{12}^{'}).
\end{equation}
Due to an  action of the  group $S_5$,  a homeomorphism $f_{\sigma, ij, kl} : F_{\sigma, ij} \to F_{\sigma, kl}$ between the sets of parameters $F_{\sigma, ij}$ and $F_{\sigma, kl}$  of a stratum $W_{\sigma}$    can  be  explicitly constructed  for an arbitrary two charts $M_{ij}$ and $M_{kl}$ such that $W_{\sigma}\subset M_{ij}, M_{kl}$.

\begin{lem}
For any three charts $M_{ij}, M_{kl}$ and $M_{mn}$ such that $W_{\sigma}\subset M_{ij}, M_{kl}, M_{mn}$,  it holds
\[
f_{\sigma, ij, kl} = f_{\sigma, ij, mn}\circ f_{\sigma, mn, kl}.
\]
\end{lem}

\begin{cor}
The automorphisms $f_{\sigma, ij,kl}$ of the space of parameters $F_{\sigma}$  of a stratum $W_{\sigma}$, which are   induced by the coordinate transition  maps  between the charts which contain the stratum $W_{\sigma}$, 
form a group. 
\end{cor}

\section{The proof of Theorem~\ref{allcharts}}
Taking into account Remark~\ref{strata137},  we first prove that the set $\tilde{\mathcal{F}}$  given by  Theorem~\ref{universal} is the universal set of parameters in the following sense. Denote by $\tilde{\mathcal{F}}_{ij}$ the coordinate record of $\tilde{\mathcal{F}}$ in a chart $M_{ij}$. If $W_{\sigma}\subset M_{ij}$, denote by $\tilde{F}_{\sigma, ij}\subset \tilde{\mathcal{F}}_{ij}$ 
a subset defined by   $c\in \tilde{F}_{\sigma, ij}$ if and only if there exists a sequence $(x_n, c_{n})\in \stackrel{\circ}{\Delta}_{5,2}\times F_{ij}$ such that $c_n\to c$ in $\tilde{\mathcal{F}}_{ij}$ and $h^{-1}(x_n, f_{ij}(c_n))$ converges to a point from $W_{\sigma}/T^5$, 
 for the earlier  defined  homeomoprhisms $f_{ij} : F_{ij}\to F$ and $h: W/T^n\to \stackrel{\circ}{\Delta}_{5,2}$.
\begin{thm}\label{mainb}
For an arbitrary charts $M_{ij}$ and $M_{kl}$, there is a  homeomorphism $\tilde{f}_{ij,kl} : \tilde{\mathcal{F}}_{ij} \to \tilde{\mathcal{F}}_{kl}$ such that
\begin{enumerate}
\item  $\tilde{f}_{ij,kl} = f_{ij,kl}$ on $F_{ij}$,
\item   for any stratum $W_{\sigma}\subset M_{ij}\cap M_{kl}$ it holds $\tilde{f}_{ij,kl}(\tilde{F}_{\sigma, ij}) =  \tilde{F}_{\sigma, kl}$.
\end{enumerate}
\end{thm}

Since $\tilde{\mathcal{F}}_{ij}\cong \tilde{\mathcal{F}}$ for any $\{i,j\}\subset \{1,\ldots , 5\}$,  Theorem~\ref{mainb} implies:
\begin{cor}
The automorphisms $\tilde{f}_{ij,kl}$ of the universal space of parameters $\tilde{\mathcal{F}}$,  which are induced by the coordinate transition  maps  between the charts, form a group.
\end{cor}

It is enough to prove Theorem~\ref{mainb} for the charts $M_{12}$ and $M_{13}$, since all our arguments are   compatible with an action of the group $S_5$. We need to prove that the  homeomorphism $f_{12,13} : F_{12}\to F_{13}$ given by
\[
((c_{1, 12}:c^{'}_{1,12}), (c_{2,12}:c_{2, 12}^{'}), (c_{3,12}:c_{3,12}^{'})) \to
\]
\[
\to ((c_{1, 12}:c_{1,12}-c^{'}_{1,12}), (c_{2,12}:c_{2,12}-c_{2, 12}^{'}), ( (c_{1,12}-c_{1,12}^{'})c_{2,12}^{'}c_{3,12}:c_{1,12}^{'}(c_{2,12}-c_{2,12}^{'})c_{3,12}^{'})  ),
\] 
can be extended to a   homeomorphism $\tilde{f}_{12,13} : \tilde{\mathcal{F}}_{12} \to \tilde{\mathcal{F}}_{13}$.
We proceed with the proof through the following lemmas. 
\begin{lem}
The homeomorphism $f_{12,13}$ is  defined on $\bar{F}_{12,13}$ except on the sets $\bar{F}_{12}^{1^{'}2^{'}},\;\;  \bar{F}_{12}^{2^{'}3^{'}},\;\;  \bar{F}_{12}^{1^{'}3},\;\; ((1:1), (1:1), (1:1))$.
\end{lem}
\begin{proof}.
It follows directly from  the  definition of  $f_{12,13}$ that $f_{12,13}$  is not defined in the following cases: $c_1=c_1^{'}, c_{2}=c_{2}^{'}$; $c_{1}=c_{1}^{'}, c_{3}^{'} =0$; $c_1^{'}=c_{2}^{'} =0$; $c_{2}^{'} = c_{3}^{'}=0$; $c_{1}^{'} = c_{3}=0$; $c_{2}=c_{2}^{'}, c_{3}=0$. Note that the relation $c_1c_{2}^{'}c_3 = c_1^{'}c_2c_{3}^{'}$ gives that the sixth and the fifth case  are the same,  as well as  the second and the forth case. Therefore, $f_{12,13}$ is not defined on the sets
$\bar{F}_{kl}^{1^{'}2^{'}},\;\;  \bar{F}_{kl}^{2^{'}3^{'}},\;\;  \bar{F}_{kl}^{1^{'}3},\;\; ((1:1), (1:1), (1:1)).$
\end{proof}

Let us consider now the set  $\bar{F}_{12}^{1^{'}3} = (1:0)\times \C P^{1}\times (0:1) $.  
We prove that  $f_{12,13}$ can be continuously extended  to this set.
\begin{lem}\label{prva}
The homeomorphism $f_{12,13} : F_{12}\to F_{13}$ can be continuously extended to a  homeomorphism $\bar{f}_{12,13}^{1^{'}3} : \bar{F}_{12}^{1^{'}3} \to \bar{F}_{13}^{11^{'}}$ by $\bar{f}_{12,13}^{1^{'}3} ((1:0), (c_2 : c_{2}^{'}), (0:1)) = 
((1:1), (c_2: c_2 -  c_{2}^{'}), (c_{2}: c_2-c_{2}^{'}))$.
\end{lem}
\begin{proof}
Let us  consider a  point $c_{0}= ((1:0), (c_2:c_{2}^{'}), (0:1))$, where $c_{2}\neq 0$ and let 
$c(n) = ((c_{1,12}(n):c_{1,12}^{'}(n)), (c_{2,12}(n):c_{2,12}^{'}(n)), (c_{3,12}(n):c_{3,12}^{'}(n)))$ be a sequence of  points from $F_{12}$ which converges to $c_{0} \in \C P^1\times \C P^1\times \C P^1$. It implies that
\[
(c_{1,12}(n):c_{1,12}^{'}(n))\to (1:0) \Rightarrow (1: \frac{c_{1,12}^{'}(n)}{c_{1,12}(n)})\to (1:0) \Rightarrow  \frac{c_{1,12}^{'}(n)}{c_{1,12}(n)} \to 0,
\]
\[
(c_{3,12}(n):c_{3,12}^{'}(n))\to (0:1) \Rightarrow (\frac{c_{3,12}(n)}{c_{3,12}^{'}(n)}:1)\to (0:1) \Rightarrow  \frac{c_{3,12}(n)}{c_{3,12}^{'}(n)} \to 0, 
\]
\[
(c_{2,12}(n):c_{2,12}^{'}(n))\to (c_2:c_2^{'}) \Rightarrow (1: \frac{c_{2,12}^{'}(n)}{c_{2,12}(n)})\to (1:\frac{c_{2}^{'}}{c_2}) \Rightarrow  \frac{c_{2,12}^{'}(n)}{c_{2,12}(n)} \to \frac{c_{2}^{'}}{c_2}.
\]
Since $c(n)\in F_{12}$ we see that  $f_{12,13}(c(n))$ is well defined  and 

\[
(c_{1, 12}(n):c_{1,12}(n)-c^{'}_{1,12}(n)) = (1: 1-\frac{c^{'}_{1,12}(n)}{c_{1,12}(n)})\to (1:1),
\]
\[
 (c_{2,12}(n):c_{2,12}(n)-c_{2, 12}(n)^{'})= (1: 1-\frac{c^{'}_{2,12}(n)}{c_{2,12}(n)})\to (1: 1- \frac{c_{2}^{'}}{c_{2}}).
\]
Since $c(n)\in F_{12}$, it follows that $c(n)$   satisfies the cubic equation, so we obtain 
\[
c_{1,12}^{'}(n) c_{3,12}^{'}(n) = \frac{c_{1,12}(n)c_{2,12}^{'}(n)c_{3,12}(n)}{c_{2,12}(n)}.
\]
It implies that
 \[
( (c_{1,12}^{'}(n)-c_{1,12}(n))c_{2,12}^{'}(n)c_{3,12}(n):c_{1,12}^{'}(n)(c_{2,12}^{'}(n)-c_{2,12}(n))c_{3,12}^{'}(n)) =
\] 
\[( (c_{1,12}^{'}(n)-c_{1,12}(n))c_{2,12}^{'}(n)c_{3,12}(n): \frac{c_{1,12}(n)c_{2,12}^{'}(n)c_{3,12}(n)}{c_{2,12}(n)})(c_{2,12}^{'}(n)-c_{2,12}(n))=
\] 
\[
((c_{1,12}^{'}(n)-c_{1,12}(n): \frac{c_{1,12}(n)}{c_{2,12}(n)}(c_{2,12}^{'}(n)-c_{2,12}(n))) = (1: \frac{1-\frac{c_{2,12}^{'}(n)}{c_{2,12}(n)}}{1-\frac{c_{1,12}^{'}(n)}{c_{1,12}(n)}})\to (1: 1-\frac{c_{2}^{'}}{c_{2}}).
\]
In this way we conclude that 
\[
h_{12,13}(c(n)) \to ((1:1),(1: 1- \frac{c_{2}^{'}}{c_{2}}), (1: 1- \frac{c_{2}^{'}}{c_{2}})).
 \]
Therefore,  we set 
\[
\bar{f}_{12,13}^{1^{'}3}((1:0), (c_2:c_{2}^{'}), (0:1)) = ((1:1),(1: 1- \frac{c_{2}^{'}}{c_{2}}), (1: 1- \frac{c_{2}^{'}}{c_{2}})),
\]
where $c_{2}\neq 0$. 

For $c_{2}=0$ we obtain the point $((1:0), (0:1), (0:1))$ and,  for a sequence $c(n)$ which converges to this point,  it holds
\[
  \frac{c_{1,12}^{'}(n)}{c_{1,12}(n)} \to 0,\;\; \frac{c_{2,12}(n)}{c_{2,12}^{'}(n)} \to 0,\;\;  \frac{c_{3,12}(n)}{c_{3,12}^{'}(n)} \to 0.
\]
If we consider the sequence $f_{12,13}(c(n))$  we obtain 
\[
(c_{1, 12}(n):c_{1,12}(n)-c^{'}_{1,12}(n)) = (1: 1-\frac{c^{'}_{1,12}(n)}{c_{1,12}(n)})\to (1:1),
\]
\[
 (c_{2,12}(n):c_{2,12}(n)-c_{2, 12}^{'}(n))= (\frac{c_{2,12}(n)}{c_{2,12}^{'}(n)}: \frac{c_{2,12}(n)}{c_{2,12}^{'}(n)}-1)\to (0:1),
\]
Since $c(n)\in F_{12}$,  it follows that $c(n)$ satisfies  the cubic equation  which  implies
\[
c_{2,12}^{'}(n) c_{3,12}(n) = \frac{c_{1,12}^{'}(n)c_{2,12}^{'}(n)c_{3,12}^{'}(n)}{c_{1,12}(n)}.
\]
Substituting this into the third coordinate for  $f_{12,13}(c(n))$ we obtain
\[
( (c_{1,12}^{'}(n)-c_{1,12}(n))c_{2,12}^{'}(n)c_{3,12}(n):c_{1,12}^{'}(n)(c_{2,12}^{'}(n)-c_{2,12}(n))c_{3,12}^{'}(n)) =
\] 
\[
( (c_{1,12}^{'}(n)-c_{1,12}(n)) \frac{c_{1,12}^{'}(n)c_{2,12}(n)c_{3,12}^{'}(n)}{c_{1,12}(n)}: c_{1,12}(n)c_{3,12}^{'}(n)(c_{2,12}^{'}(n)-c_{2,12}(n)))=
\]
\[
( (c_{1,12}^{'}(n)-c_{1,12}(n)) \frac{c_{1,12}^{'}(n)c_{2,12}(n)}{c_{1,12}(n)}: c_{1,12}(n)(c_{2,12}^{'}(n)-c_{2,12}(n)))=
\]
\[
( (\frac{c_{1,12}^{'}(n)}{c_{1,12}(n)}-1) \frac{c_{1,12}^{'}(n)}{c_{1,12}(n)}c_{2,12}(n): c_{2,12}^{'}(n)-c_{2,12}(n))=
\]
\[
( (\frac{c_{1,12}^{'}(n)}{c_{1,12}(n)}-1) \frac{c_{1,12}^{'}(n)}{c_{1,12}(n)}\frac{c_{2,12}(n)}{c_{2,12}^{'}(n)}: 1-\frac{c_{2,12}(n)}{c_{2,12}^{'}(n)})=
\]
\[
 (\frac{(\frac{c_{1,12}(n)^{'}}{c_{1,12}(n)}-1) \frac{c_{1,12}^{'}(n)}{c_{1,12}(n)}\frac{c_{2,12}(n)}{c_{2,12}^{'}(n)}}{ 1-\frac{c_{2,12}(n)}{c_{2,12}^{'}(n)}}: 1) \to (0:1).
\]
Therefore,  we define
\[
\bar{f}_{12,13}^{1^{'}3}((1:0), (1:0), (0:1)) = ((1:1),(0: 1), (0: 1)).
\]
\end{proof}

In a  similar way it can be proved:
\begin{lem}\label{druga}
The homeomorphism $f_{12,13} : F_{12}\to F_{13}$ can be continuously extended to a homeomorphism $\bar{f}_{12,13}^{2^{'}3^{'}} : \bar{F}_{12}^{2^{'}3^{'}} \to \bar{F}_{13}^{22^{'}}$ by $\bar{f}_{12,13}^{2^{'}3^{'}} ((c_1:c_{1}^{'}), (1 : 0), (1:0)) = ((c_1:c_1 - c_{1}^{'}), (1:1), (c_1-c_{1}^{'}:c_{1}))$.
\end{lem}
\begin{proof}
Let $c_{0} = ((c_1:c_1^{'}), (1:0), (1:0))$ and let  a sequence $c(n)\in F_{12}$ converges to $c_{0}$.
It implies that
\[
\frac{c_{1}^{'}(n)}{c_{1}(n)}\to \frac{c_{1}^{'}}{c_{1}},\;\; \frac{c_{2}^{'}(n)}{c_{2}(n)}\to 0,\;\; \frac{c_{3}^{'}(n)}{c_{3}(n)}\to 0.
\]
Therefore, 
\[
(c_{1}(n): c_{1}(n) - c_{1}^{'}(n))\to (1: 1-\frac{c_1^{'}}{c_1}),\;\; (c_{2}(n): c_{2}(n)-c_{2}^{'}(n)) \to (1:1),
\]
\[
((c_{1,12}(n)-c_{1,12}^{'}(n))c_{2,12}^{'}(n)c_{3,12}(n) : c_{1,12}^{'}(n)(c_{2,12}(n)-c_{2,12}^{'}(n))c_{3,12}^{'}(n)) =
\]
\[
((c_{1,12}(n)-c_{1,12}^{'}(n))c_{2,12}^{'}(n)c_{3,12}(n) : \frac{c_{1,12}(n)c_{2,12}^{'}(n)c_{3,12}(n)}{c_{2,12}(n)}(c_{2,12}(n)-c_{2,12}^{'}(n))) =\]
\[
(c_{1,12}(n)-c_{1,12}^{'}(n) : \frac{c_{1,12}(n)}{c_{2,12}(n)}(c_{2,12}(n) - c_{2,12}^{'}(n))) = 
\]
\[(1: \frac{1-\frac{c_{2,12}^{'}(n)}{c_{2,12}(n)}}{1-\frac{c_{1,12}^{'}(n)}{c_{1,12}(n)}})\to (c_{1}-c_{1}^{'} : c_{1}).
\]
\end{proof}

\begin{lem}\label{prvi}
The homeomorphism $f_{12,13} : F_{12}\to F_{13}$ continuously extends to a constant map $\bar{f}_{12,13}^{1^{'}2^{'}} : \bar{F}_{12}^{1^{'}2^{'}} \to \bar{F}_{13}^{11^{'}}\cap \bar{F}_{13}^{22^{'}}\cap F_{13}^{33^{'}}$ by $\bar{f}_{12,13}^{1^{'}2^{'}} ((1:0), (1 : 0), (c_3:c_{3}^{'})) = ((1:1), (1:1), (1:1))$.
\end{lem}
\begin{proof}
Let $c_{0} = ((1:0), (1:0), (c_3 :c_{3}^{'}))$ and let  $c(n)$  be  a  sequence of points from $F_{12}$ which converges to $c_{0}$. It holds:
\[
(c_{1,12}(n) : c_{1,12}^{'}(n))\to (1:0) \Rightarrow \frac{c_{1.12}^{'}(n)}{c_{1,12}(n)}\to 0,
\]
\[
(c_{2,12}(n): c_{2,12}^{'}(n))\to (1:0)\Rightarrow \frac{c_{2,12}^{'}(n)}{c_{2,12}(n)}\to 0,
\]
\[
(c_{3,12}(n):c_{3,12}^{'}(n)) \to (c_3:c_3^{'}) \Rightarrow \frac{c_{3,12}^{'}(n)}{c_{3,12}(n)}\to \frac{c_{3}^{'}}{c_{3}}.
\]
It implies that
\[
(c_{1,12}(n):c_{1,12}(n)-c_{1,12}^{'}(n)) = (1:1-\frac{c_{1,12}^{'}(n)}{c_{1,12}(n)})\to (1:1),
\]
\[
(c_{2,12}(n):c_{2,12}(n)-c_{2,12}^{'}(n)) = (1:1-\frac{c_{2,12}^{'}(n)}{c_{2,12}(n)})\to (1:1),
\]
\[
((c_{1,12}(n)-c_{1,12}^{'}(n))c_{2,12}^{'}(n)c_{3,12}(n) : c_{1,12}^{'}(n)(c_{2,12}(n)-c_{2,12}^{'}(n))c_{3,12}^{'}(n)) =
\]
\[
((c_{1,12}(n)-c_{1,12}^{'}(n))c_{2,12}^{'}(n)c_{3,12}(n) : \frac{c_{1,12}(n)c_{2,12}^{'}(n)c_{3,12}(n)}{c_{2,12}(n)}(c_{2,12}(n)-c_{2,12}^{'}(n))) =\]
\[
(c_{1,12}(n)-c_{1,12}^{'}(n) : \frac{c_{1,12}(n)}{c_{2,12}(n)}(c_{2,12}(n)-c_{2,12}^{'}(n))) = (1: \frac{1-\frac{c_{2,12}^{'}(n)}{c_{2,12}(n)}}{1-\frac{c_{1,12}^{'}(n)}{c_{1,12}(n)}}) \to (1:1).\]
Therefore,  the homeomorphism $f_{12,13}$  continuously extends to $\bar{f}_{12,13}^{1^{'}2^{'}} : \bar{F}_{12}^{1^{'}2^{'}} \to ((1:1),(1:1),(1:1))=\bar{F}_{13}^{11^{'}}\cap \bar{F}_{13}^{22^{'}}\cap \bar{F}_{13}^{33^{'}}$.

\end{proof}

\begin{lem}
The homeomorphism $f_{12,13} : F_{12} \to F_{13}$ can not be continuously extended to the point  $((1:1), (1:1), (1:1))\in \bar{F}_{12}$.
\end{lem}
\begin{proof}
Let us consider the point $((1:1), (1:1), (1:1))$  and a sequence $c(n)\in F_{12}$ which converges to this point. Then

\[
(c_{1,12}(n):c_{1,12}(n)-c_{1,12}^{'}(n)) = (1:1-\frac{c_{1,12}^{'}(n)}{c_{1,12}(n)})\to (1:0),
\]
\[
(c_{2,12}(n):c_{2,12}(n)-c_{2,12}^{'}(n)) = (1:1-\frac{c_{2,12}^{'}(n)}{c_{2,12}(n)})\to (1:0),
\]
\[
((c_{1,12}(n)-c_{1,12}^{'}(n))c_{2,12}^{'}(n)c_{3,12}(n) : c_{1,12}^{'}(n)(c_{2,12}(n)-c_{2,12}^{'}(n))c_{3,12}^{'}(n)) =
\]
\[
(c_{1,12}(n)-c_{1,12}^{'}(n) : \frac{c_{1,12}(n)}{c_{2,12}(n)}(c_{2,12}(n)-c_{2,12}^{'}(n))) = (1: \frac{1-\frac{c_{2,12}^{'}(n)}{c_{2,12}(n)}}{1-\frac{c_{1,12}^{'}(n)}{c_{1,12}(n)}}). \]
Since $\frac{c_{1,12}^{'}(n)}{c_{1,12}(n)}\to 1$ and $\frac{c_{2,12}^{'}(n)}{c_{2,12}(n)}\to 1$ we deduce  that the above limit is not defined, so  $f_{12,13}$ can not be continuously extended to  the point $((1:1), (1;1), (1:1))$ .
\end{proof}

\begin{rem}\label{problem}
As Lemma~\ref{prva} and Lemma~\ref{druga} show the problem with an extension of    the homeomorphism $f_{12,13} : F_{12}\to F_{13}$ to the boundary of $F_{12}$ in $\C P ^1\times \C P^1\times \C P^1$,   arises at the points from $\bar{F}_{12}^{1^{'}2^{'}}$ and at the point $((1:1), (1:1), (1:1))$. That is the reason for considering the blowup of $\mathcal{F}$ at the point $((1:1), (1:1), (1:1))$ as an universal set of parameters.
\end{rem}

Denote by $\bar{f}_{12,13}$ an extension of the  homeomorphism $f_{12,13} : F_{12}\to F_{13}$ to $\mathcal{F} \setminus \{\bar{F}_{12}^{1^{'}2^{'}}\cup ((1:1),(1:1),(1:1))\}$.

We now  prove the first statement of Theorem~\ref{mainb}.
\begin{prop}\label{tildef}
The map  $\tilde{f}_{12,13} : \tilde{\mathcal{F}}_{12} \to \tilde {\mathcal{F}}_{13}$ defined by
\begin{equation}\label{int}
\tilde{f}_{12,13} = \bar{f}_{12,13}\circ \pi \;\; \text{on}\;\;  \tilde{\mathcal{F}}_{12}\setminus G_{12},
\end{equation}
\begin{equation}\label{gran}
\tilde{f}_{12,13}(((1:0), (1:0), (c_3 : c_{3}^{'})), (1:1) )= (((1:1),(1:1),(1:1)),  (c_3:c_3^{'})),
\end{equation}
\begin{equation}\label{divisor}
\tilde{f}_{12,13}(((1:1), (1:1), (1:1)), (c:c^{'})) = (((1:0), (1:0), (c:c^{'})), (1:1)),
\end{equation}
where $G_{12} =  (((1:1), (1:1), (1:1)), (c_3:c_3^{'})) \cup ( ((1:0),  (1:0), (c_3:c_3^{'})),  (1:1)),\; (c_3: c_3^{'})\in \C P^1$ and $\pi :\tilde{\mathcal{F}}\to \mathcal{F}$ is a projection, 
is a  homeomorphism.
\end{prop}
\begin{proof}
 It follows from Lemma~\ref{prva} and Lemma~\ref{druga} that the map $f_{12,13}$  continuously extends to a
homeomorphism $f_{12,13}^{1^{'}3} : \bar{F}_{12}^{1^{'}3} \to \bar{F}_{13}^{11^{'}}$ defined by $f_{12,13}^{1^{'}3} ((1:0), (c_2:c_{2}^{'}), (0:1))  = ((1:1), (c_2:c_2-c_2^{'}), (c_2:c_2-c_2^{'}))$ and a  homeomorphism $f_{12,13}^{2^{'}3^{'}} : \bar{F}_{12}^{2^{'}3^{'}} \to \bar{F}_{13}^{22^{'}}$ defined by
$f_{12,13}^{2^{'}3^{'}}((c_1 : c_1^{'}), (1:0), (1:0)) = ((c_1: c_1-c_1^{'}), (1:1), (c_1-c_1^{'}:c_1))$.   Together with Remark~\ref{blowgen} this  provides the proof for~\eqref{int}.  In order to define an  extension of $f_{12, 13}$ on $\tilde{F}_{12}$ we need to define this extension  on the set $\bar{F}_{12}^{1^{'}2^{'}}$ and 
the divisor $((1:1), (1:1), (1:1)), (c:c^{'}))$. The set $\bar{F}_{12}^{1^{'}2^{'}}$ ca be embedded  into $\tilde{F}_{12}$ as $(((1:0), (1:0), (c_3 : c_{3}^{'})), (1:1))$.   Let us consider a sequence $y(n) = (((1:c_{1}^{'}(n)),(1:c_{2}^{'}(n)), (c_{3}(n):c_{3}^{'}(n))), (1-c_{1}^{'}(n): 1-c_{2}^{'}(n)))$   from $\pi ^{-1}(F_{12})\subset \tilde{\mathcal{F}}_{12}$  which converges to a point $(((1:0), (1:0), (c_3:c_3^{'})), (1:1))$. It holds
\[
\tilde{f}_{12,13}(y_{n})  = f_{12,13}(\pi (y_{n}))   = ((1:1-c_{1}^{'}(n)),(1:1-c_{1}^{'}(n)),((1-c_1^{'}(n))c_{2}^{'}(n)c_3(n): c_1^{'}(n)(1-c_{2}^{'}(n))c_3^{'}(n))).
\]
The fourth coordinate of the point $\tilde{f}_{12,13}(y_{n})$ in  $\tilde{F}_{12,13}$   is  $(c_3(n):c_3^{'}(n))$. Namely, the fourth coordinate of this point is $(x_1(n):x_2(n))$ such that  that $c_1^{'}(n)x_2(n) = c_{2}^{'}(n)x_1(n)$, which implies that $(x_{1}(n):x_{2}(n)) = (1: \frac{c_{2}^{'}(n)}{c_{1}^{'}(n)})$. The cubic equation  implies that   $(x_{1}(n) : x_{2}(n))=(c_{3}(n): c_{3}^{'}(n))$. Since by  Lemma~\ref{prvi}  the continuous extension of $f_{12,13}$ on $\bar{F}_{12}^{2^{'}3^{'}}$ maps  $\bar{F}_{12}^{2^{'}3^{'}}$  to $((1:1), (1:1), (1:1))$ we obtain that the sequence $\tilde{f}_{12,13}(y_{n})\subset \tilde{\mathcal{F}}_{13}$ converges to the point $(((1:1), (1:1), (1:1)), (c_3:c_{3}^{'}))$. This  proves the formula~\eqref{gran}. The formula~\eqref{divisor}  follows in a similar way from~\eqref{blow} in Lemma~\ref{blowlem}.
\end{proof}

We prove the second statement of Theorem~\ref{mainb}, which states that  a homeomorphic type of $\tilde{F}_{\sigma,ij}$ does not depend on a chart $M_{ij}$:
\begin{prop}\label{onechart}
Let $W_{\sigma}\subset M_{ij}, M_{kl}$. Then $\tilde{f}_{ij, kl}(\tilde{F}_{\sigma, ij}) = \tilde{F}_{\sigma, kl}$ for the homeomorphism  $\tilde{f}_{ij, kl}: \tilde{\mathcal{F}}_{ij}\to \tilde{\mathcal{F}}_{kl}$.
\end{prop}
\begin{proof}
The proof immediately follows: a sequence $z_n$ from the main stratum,  which converges to a  point from  $W_{\sigma}$ defines, when considering  it in the charts $M_{ij}$ and $M_{kl}$,  the sequences of parameters $c_{n}^{ij}$ and $c_{n}^{kl}$. We have that
$c_{n}^{kl}= f_{ij, kl}(c_{n}^{ij})$, which  implies that $\tilde{f} _{ij, kl}(\lim\limits_{n\to \infty} c_{n}^{ij}) = \lim\limits_{n\to 
\infty} c_{n}^{kl}$, which  proves the statement.
\end{proof}   

In this way the proof of Theorem~\ref{mainb} is completed.

Now we prove Theorem~\ref{allcharts} which states that  any stratum in $G_{5,2}$ can be parametrized by considering   the universal set of parameters just in one chart. Fix a chart $M_{ij}$ and consider an arbitrary stratum $W_{\sigma}$. Let $M_{kl}$ be an arbitrary chart such that $W_{\sigma}\subset M_{kl}$ and $\tilde{F}_{\sigma, kl}$ be the virtual space  of parameters for $W_{\sigma}$ in  the chart $M_{kl}$.
 Let 
\[
\tilde{F}_{\sigma, ij, kl} = \tilde{f}_{kl, ij}(\tilde{F}_{\sigma, kl})\subset \tilde{\mathcal{F}}_{ij},
\]
for the homeomorphism $\tilde {f}_{kl, ij} : \tilde{\mathcal{F}}_{kl}\to \tilde{\mathcal{F}}_{ij}$.
\begin{prop}\label{any}
It holds
\[
\tilde{F}_{\sigma, ij, kl} = \tilde{F}_{\sigma, ij, mn},
\]
for any charts $M_{kl}, M_{mn}$ such that $W_{\sigma}\subset M_{kl},M_{mn}$.
\end{prop} 
\begin{proof}
Consider the homeomorphisms $\tilde{f}_{ij, kl}: \tilde{\mathcal{F}}_{ij}\to \tilde{\mathcal{F}}_{kl}$, $\tilde{f}_{ij, mn} ; \tilde{\mathcal{F}}_{ij}\to \tilde{\mathcal{F}}_{mn}$ and $\tilde{f}_{kl, mn} : \tilde{\mathcal{F}} _{kl}\to \tilde{\mathcal{F}} _{mn}$. Clearly  Lemma~\ref{tricharts} holds  after extension to blowup, that is $ \tilde{f}_ {kl, ij}  = \tilde{f}_{mn, ij}\circ \tilde{f}_{kl, mn}$. By Proposition~\ref{onechart} we have that $\tilde{f}_{kl, mn}(\tilde{F}_{\sigma, kl})= \tilde{F}_{\sigma, mn}$, which  implies that  $\tilde{f}_{kl, ij}(\tilde{F}_{\sigma, kl})= \tilde{f}_{mn, ij}(\tilde{F}_{\sigma, mn})$,  that is  $\tilde{F}_{\sigma, ij, kl} = \tilde{F}_{\sigma, ij, mn}$.
\end{proof}

\begin{rem}
The set $\tilde{F}_{\sigma,ij, kl}$ we denote by $\tilde{F}_{\sigma, ij}$. Note that in this way we defined the sets $\tilde{F}_{\sigma, ij}$ for an arbitrary chart $M_{ij}$ and an arbitrary admissible set  $\sigma$.
\end{rem}

Proposition~\ref{any} can be reformulated as follows:
\begin{prop}
Any stratum $W_{\sigma}$ in $G_{5,2}$ can be parametrized by $\tilde{F}_{\sigma, ij}\subset \tilde{\mathcal{F}}_{ij}$  such that for any two chart $M_{kl}$, $W_{\sigma}\subset M_{kl}$ it holds  $\tilde{f}_{ij,kl}(\tilde{F}_{\sigma, ij}) = \tilde{F}_{\sigma, kl}$. 
\end{prop}

\subsection{The virtual spaces of parameters in the chart $M_{12}$ }
Using Theorem~\ref{mainb} we describe here the virtual spaces of parameters $\tilde{F}_{\sigma, 12}\subset \tilde{\mathcal{F}}_{12}$ for the  strata $W_{\sigma}$ whose admissible polytopes are $K_{ij}(9), K_{ij, kl}, K_{ij}(7), P_{ij}, O_i$.  

\begin{thm}\label{bound-univ}
The virtual space of parameters $\tilde{F}_{\sigma, 12}$  for a stratum  $W_{\sigma}$ whose admissible polytope is  $ K_{ij}(9), K_{ij, kl}, K_{ij}(7), P_{ij}, O_i$ is as follows:
\begin{enumerate}
\item $K_{12}(9) \to (((1:1), (1:1), (1:1)), (c:c^{'}))$,\;\; $K_{ij}(9) \to \tilde{F}_{ij, 12}$,\; $ij\neq 12$;
\item $K_{12, 34}\to  (((1:1), (1:1), (1:1)), (0: 1))$,\;\;   $K_{12, 35}\to  (((1:1), (1:1), (1:1)), (1: 0))$;
\item $K_{12, 45}\to  (((1:1), (1:1), (1:1)), (1: 1))$,\;\;  $K_{ij, kl} \to \tilde{F}_{(ij, kl),12}$, \; $ij\neq 12$; 
\item $K_{34}(7)\to ((1:1), (c:c^{'}), (c:c^{'}))$, \;\;  $K_{35}(7)\to ((c:c^{'}), (1:1), (c^{'}:c))$;
\item $K_{45}(7)\to ((c:c^{'}), (c:c^{'}), (1:1))$,\;\;  $K_{12}(7) \to  (((1:1), (1:1), (1:1)), (c: c^{'}))$;
\item $O_1\to F_{12}\cup \tilde{F}_{12,12}\cup \tilde{F}_{13,12}\cup \tilde{F}_{14, 12}\cup \tilde{F}_{15,12}$;
\item $O_2\to F_{12}\cup \tilde{F}_{12, 12}\cup \tilde{F}_{23, 12}\cup \tilde{F}_{24, 12}\cup \tilde{F}_{25, 12}$;
\item $P_{34}\to ((1:1), (c:c^{'}), (c:c^{'}))$, \;\;  $P_{35}(7)\to ((c:c^{'}), (1:1), (c^{'}:c))$;
\item  $P_{45}(7)\to ((c:c^{'}), (c:c^{'}), (1:1))$,\;\;  $P_{12}(7) \to  (((1:1), (1:1), (1:1)), (c: c^{'}))$;
\item  $K_{ij}(7)\to \tilde{F}_{ij, 12}(7)$,\;\;  $P_{ij}\to \tilde{F}_{ij, 12}(6)$, \; $ij\neq 34, 34, 45, 12$,\;\;  $O_{i}\to \tilde{F}_{i,12}$, $i\neq 1,2$.
\end{enumerate}
where $(c:c^{'})\in \C P^{1}$.
\end{thm}
\begin{proof}
We demonstrate  the proof for the stratum whose admissible polytope is $K_{12}(9)$. This stratum  belongs to the chart $M_{13}$ and it can be directly checked  that $\tilde{F}_{12, 13} = ((1:0), (1:0), (c_3:c_3^{'}))$, $(c_3: c_3^{'})\in \C P^{1}_{A}$. It follows from Proposition~\ref{tildef}  that $\tilde{F}_{12, 12} = \tilde{f}^{-1}_{12,13}(\tilde{F}_{12,13}) = (((1:1), (1:1), (1:1)), (c_3:c_3^{'}))$, which  proves the first statement. Note that for the strata with the admissible polytopes $O_1$ and $O_2$, the statement follows from Remark~\ref{oct}.
\end{proof}

\subsection{The projection from $\tilde{F}_{\sigma, ij}$ to $F_{\sigma}$.} We prove that, using this construction,  we can define   a  projection from $\tilde{F}_{\sigma, ij}$ to $F_{\sigma}$ for any non-orbit stratum $W_{\sigma}$ and any chart $M_{ij}$,   .

Assume first that $F_{\sigma}$ is not a point. If $W_{\sigma}\subset M_{ij}$,   from the construction  of $\tilde{F}_{\sigma, ij}$ it follows that there is  a canonical projection $\tilde{g}_{\sigma ,ij}:  \tilde{F}_{\sigma, ij} \to F_{\sigma, ij}$.   If $W_{\sigma}\not \subset M_{ij}$, let $M_{kl}$ be a chart such that $W_{\sigma}\subset M_{kl}$  and let $\tilde{g}_{\sigma ,kl}:  \tilde{F}_{\sigma, kl} \to F_{\sigma, kl}$  be a canonical projection. Let $\tilde{f}_{ ij, kl}:  \tilde{F}_{\sigma, ij}\to \tilde{F}_{\sigma, kl}$  be the homeomorphism given by Theorem~\ref{mainb}.  We obtain  a  projection  $ \tilde{g} _{\sigma, ij, kl}:\tilde{F}_{\sigma, ij} \to F_{\sigma, kl}$ defined by the composition
\begin{equation}
 \tilde{F}_{\sigma, ij} \stackrel{\tilde{f}_{ij, kl}}{\to} \tilde{F}_{\sigma, kl}\stackrel{\tilde{g} _{\sigma, kl}}{\to}F_{\sigma, kl}.
\end{equation}

Let $f_{\sigma, kl, mn} : F_{\sigma, kl}\to F_{\sigma, mn}$ be the  homeomorphism   defined in~\ref{mainstratum} and~\ref{non-orbit}.  From the construction of $\tilde{F}_{\sigma, ij}$ it directly follows:
\begin{lem}\label{com}
If  $W_{\sigma}\subset M_{kl}, M_{mn}$ then
\[
f_{\sigma, kl, mn} \circ \tilde{g}_{\sigma, kl} = \tilde{g}_{\sigma, mn}\circ \tilde{f}_{kl, mn}.
\]
\end{lem}

Thus, we   obtain a  projection $\tilde{g}_{\sigma, ij} : \tilde{F}_{\sigma, ij} \to F_{\sigma}$  defined by the map
\begin{equation}\label{pr-ch}
g_{\sigma, ij, kl} = f_{\sigma, kl}\circ \tilde{g}_{\sigma, ij, kl}.
\end{equation}

The projection $\tilde{g}_{\sigma, ij}$ does not depend on the choice  of a  chart $M_{kl}$,  such that $W_{\sigma}\subset M_{kl}$.
\begin{prop}
It holds  $g_{\sigma, ij, kl} = g_{\sigma, ij, mn}$ for any $\sigma \in \mathcal{A}$ and any charts $M_{kl}, M_{mn}$,  such that $W_{\sigma}\subset M_{kl}, M_{mn}$.
\end{prop}
\begin{proof}
Combining Lemma~\ref{tricharts} and Lemma~\ref{com} we obtain the following commutative diagram:

\begin{equation*}\begin{CD}
 \tilde{F}_{\sigma, ij} @>{\tilde{f}_{ij,kl}}>> {\tilde{F}}_{\sigma, kl} @>{\tilde{g} _{\sigma, kl}}>> F_{\sigma, kl}@>{f_{\sigma, kl}}>>F_{\sigma}\\
@VV {\equiv}V       @VV{\tilde{f}_{kl, mn}}V                   @VV{f_{kl, mn}}V @VV{\equiv}V \\ 
 \tilde{F}_{\sigma, ij}@>{\tilde{f}_{ij,mn}}>> \tilde{F}_{\sigma, mn} @>{\tilde{g} _{\sigma, mn}}>> F_{\sigma, mn}@>f_{\sigma, kl}>>F_{\sigma},
\end{CD}\end{equation*}
which proves the statement.
\end{proof}

If $F_{\sigma}$ is a point then $g_{\sigma, ij} : \tilde{F}_{\sigma, ij} \to F_{\sigma}$ is obviously uniquely defined.

Altogether this leads:

\begin{cor}\label{canonic-proj}
There exists a  canonical projection $g_{\sigma, ij} : \tilde{F}_{\sigma, ij}\to F_{\sigma}$ for any $\sigma\in \mathcal{A}$ and any chart $M_{ij}$.
\end{cor}

\begin{rem}
We pointed in Remark~\ref{param-chart} that if  $W_{\sigma}\subset M_{ij}$ and   $F_{\sigma}$ is a point, one can not  write down this point in the parameter coordinates for  the chart $M_{ij}$. Thus, in this case $F_{\sigma, ij}$ is not defined, so it is not  defined a  projection $\tilde{F}_{\sigma, ij}\to F_{ij, \sigma}$. Nevertheless the  projection $\tilde{F}_{\sigma, ij}\to F_{\sigma}$ is uniquely defined.
\end{rem}
The  space $\tilde{F}_{\sigma, ij}$ is in general    larger then the space  of parameters $F_{\sigma}$  as Proposition~\ref{parametrization-7} shows.

\section{An algebraic manifold in $(\C P^1)^5$ which realizes  $\tilde{\mathcal{F}}$}
We prove that  the space of parameters $F$ of the main stratum,  as well as its compactification $\tilde{\mathcal{F}}$ which  we have described in  previous sections  can be realized by  algebraic submanifolds in $(\C P^1)^{5}$.

\begin{prop}\label{embedd}
There is an  embedding  $I : F \to (\C P^1)^{5}$.
\end{prop}
\begin{proof}
Let $L$ be a point from the main stratum $W$  of $G_{5,2}$  and $A_{L}$ an arbitrary $(5\times 2)$- matrix  representing $L$.  The rows of the matrix $L$  define the  five non-collinear  points in $\C ^2$. We enumerate them by  the rows  which define them,  that is $1,2,3,4, 5$. We choose the following $5$ lexicographically ordered variations of four of these five points: $1234$, $1235$, $1245$, $1345$, $2345$ and map them to $\C ^{*}$ by the cross-ratio from   projective geometry. More precisely, we consider the map $\{1234, 1235, 1245, 1345, 2345\}\to \C ^{*}$ defined by 
\begin{equation}~\label{jedan}
1234\to \frac{P^{13}P^{24}}{P^{14}P^{23}},\; 1235\to \frac{P^{13}P^{25}}{P^{15}P^{23}}, \; 1245\to \frac{P^{14}P^{25}}{P^{15}P^{24}},
\end{equation}
\begin{equation}\label{dva}
1345\to \frac{P^{14}P^{35}}{P^{15}P^{34}},\; 2345\to \frac{P^{24}P^{35}}{P^{25}P^{34}}.
\end{equation}
where $P^{ik}=P^{ik}(A_L)$  are  the Pl\"ucker coordinates for $L$.   

 Using this  map we obtain  a map from the main stratum $W$ to $(\C P^1)^{5}$ defined by $ijkl\to (P^{ik}P^{jl}:P^{il}P^{jk})\in \C P^1$. Note that this new map is well defined,   it  does not depend on the choice of a basis for $L$ neither on the choice of a  basis in $\C ^n$, since  Pl\"ucker coordinates are defined uniquely  up to common constant. This map is clearly equivariant for the action of the torus $(\C ^{*})^{5}$,  since  Pl\"ucker coordinates are equivariant for this action. In this we obtain an embedding of $F = W/(\C ^{*})^{5}$ in $(\C P^1)^{5}$. 
\end{proof}

Let $(c_i:c_i^{'})$, $1\leq i\leq 5$,  be  coordinates in $(\C P^{1})^{5}$.
We  describe an  algebraic manifold in $(\C P^1)^5$ which contains the image of $F$ by the above map $ijkl\to \frac{P^{ik}P^{jl}}{P^{il}P^{jk}}$.

\begin{prop}\label{image}
The  image of an  embedding $I : F\to (\C P^1)^{5}$  belongs to the intersection of the  following hypers
urfaces
\begin{equation}\label{jedanh}
c_1c_2^{'}c_3 =  c_1^{'}c_2c_3^{'},\;\;  c_{2}^{'}c_4(c_1-c_1^{'}) = c_1^{'}c_{4}^{'}(c_2-c_2^{'}),
\end{equation}
\begin{equation}\label{dvah}
 \;\; (c_1-c_1^{'})c_2c_5 = c_1(c_2-c_2^{'})c_5^{'}, \;\;  c_3c_4^{'}c_5 =  c_3^{'}c_4c_5^{'}.
\end{equation}
\end{prop}
\begin{proof}
It immediately follows  from~\eqref{jedan}  that the first equation in~\eqref{jedanh} is satisfied.  
We now  use this conclusion and  properties of the cross-ratio. Since
\[
1324 \to 1- \frac{c_1}{c_1^{'}}, \;\; 1325 \to 1-\frac{c_2}{c_{2}^{'}}, \;\; 1345\to \frac{c_4}{c_4^{'}},
\]
 we analogously obtain that  
\[ (c_1^{'}-c_{1})c_{2}^{'}c_4=c_1^{'}(c_2^{'}-c_2)c_{4}^{'}.
\]
Since
\[
2314 \to 1-\frac{c_1^{'}}{c_{1}}, \;\;  2315 \to 1-\frac{c_2^{'}}{c_{2}}, \;\; 2345 \to \frac{c_5}{c_5^{'}},
\]
we conclude that   
\[
(c_1-c_1^{'})c_2c_5=c_1(c_2-c_2^{'})c_5^{'}.
\]
Since 
\[
4512\to \frac{c_3^{'}}{c_3}, \;\;  4513\to \frac{c_4}{c_4^{'}}, \;\; 4523\to \frac{c_5}{c_5^{'}},
\]
it follows that 
\[
c_3c_4^{'}c_5 = c_3^{'}c_4c_{5}^{'}.
\]

\end{proof}
 
Denote  by  $\mathcal{G}$ an  algebraic manifold in $(\C P^1)^{5}$ obtained as the intersection of the surfaces given by Proposition~\ref{image}. 

\begin{lem}\label{devet}
The virtual spaces of parameters  $\tilde{F}_{ ij, 12} \cong \C P^1_{A}$ of the strata whose admissible polytopes are $K_{ij}(9)$  for $T^5$-action on $G_{5,2}$,   can be embedded into $\mathcal{G}$ as follows:
\begin{enumerate}
\item $\tilde{F}_{ 13, 12}\to ((0:1), (0:1), (c:c^{'}), (1:1), (c^{'}:c))$, 
\item $\tilde{F}_{14, 12}\to ((1:0), (c:c^{'}), (0:1), (0:1), (c-c^{'}:c))$,
\item $\tilde{F}_{ 15, 12}\to ((c:c^{'}), (1:0), (1:0), (1:0), (c:c-c^{'}))$,
\item $\tilde{F}_{23, 12}\to ((1:0), (1:0), (c:c^{'}), (c:c^{'}), (1:1))$,
\item $\tilde{F}_{24, 12}\to ((0:1),(c:c^{'}), (1:0), (c^{'}-c:c^{'}), (0:1))$, 
\item 
$\tilde{F}_{25, 12}\to ((c:c^{'}), (0:1), (0:1), (c^{'}:c^{'}-c), (1:0))$,
\item $\tilde{F}_{34, 12}\to ((1:1), (c:c^{'}), (c:c^{'}), (1:0), (1:0))$, 
 \item $\tilde{F}_{ 35, 12}\to ((c:c^{'}), (1:1), (c^{'}:c), (0:1), (0:1))$,
\item $\tilde{F}_{12, 12}\to  ((1:1), (1:1), (1:1), (c:c^{'}), (c:c^{'}))$, 
\item  $\tilde{F}_{45, 12}\to ((c:c^{'}), (c:c^{'}), (1:1), (1:1), (1:1))$.
\end{enumerate}
\end{lem}
\begin{proof}
The stratum $W_{ij}$ whose admissible polytope  is $K_{ij}(9)$ is defined by the condition that $P^{ij}=0$. Combining this and~\eqref{jedan},~\eqref{dva} we immediately obtain  the embeddings $(1) - (8)$. In order to prove that the embeddings $(9)$ and $(10)$ hold  we use in addition  the Pl\"ucker relations presented  in Example~\ref{Pl2}.
\end{proof}
In the same way we prove the following lemmas: 
\begin{lem}
The virtual spaces of parameters  $\tilde{F}_{ (ij,kl), 12}\cong\{\ast\}$ of the strata whose admissible polytopes are $K_{ij,kl}$ can be embedded into $\mathcal{G}$ as follows
\begin{itemize}
\item $\tilde{F}_{(14, 23), 12} \to   ((1:0), (1:0),(0:1),(0:1), (1:1))$,
\item $ \tilde{F}_{(13,24), 12} \to  ((0:1), (0:1), (1:0), (1:1), (0:1))$, 
\item $\tilde{F}_{(15,24), 12} \to   ((0:1), (1:0), (1:0), (1:0), (0:1))$,
\item $\tilde{F}_{(23,45). 12}\to  ((1:0), (1:0), (1:1), (1:1), (1:1))$,
\item $\tilde{F}_{(24,35), 12}\to  ((0:1), (1:1), (1:0), (0:1), (0:1))$,
\item  $\tilde{F}_{(25,34), 12}\to ((1:1), (0:1), (0:1), (1:0), (1:0))$, 
\item $\tilde{F}_{(15, 23), 12}\to ((1:0), (1:0), (1:0), (1:0), (1:1))$,
\item  $\tilde{F}_{(13, 25), 12}\to ((0:1), (0:1), (0:1), (1:1),(1:0))$,
\item $\tilde{F}_{(14,25), 12}\to ((1:0), (0:1), (0:1), (0:1), (1:0))$,
\item   $\tilde{F}_{(13,45), 12}\to ((0:1), (0:1), (1:1), (1:1), (1:1))$,
\item  $\tilde{F}_{(14,35), 12}\to ((1:0), (1:1), (0:1), (0:1), (0:1))$,
\item $\tilde{F}_{(15,34), 12}\to ((1:1), (1:0), (1:0), (1:0), (1:0))$.
\item $\tilde{F}_{(12,34), 12}\to ((1:1), (1:1), (1:1), (1:0), (1:0))$,
\item  $\tilde{F}_{(12,35), 12 } \to  ((1:1), (1:1), (1:1), (0:1), (0:1))$,
\item $\tilde{F}_{(12,45), 12} \to  ((1:1), (1:1), (1:1), (1:1), (1:1))$. 
\end{itemize}
\end{lem}

\begin{lem}
The virtual spaces of parameters  $\tilde{F}_{ij, 12}(7)\cong \C P^1$ of the strata whose admissible polytopes are $K_{ij}(7)$ can be embedded into $\mathcal{G}$ as follows
\begin{itemize}
\item $\tilde{F}_{23, 12}(7)\to  ((1:0), (1:0), (c:c^{'}), (c:c^{'}), (1:1))$,
\item  $\tilde{F}_{24, 12}(7)\to  ((0:1),(c:c^{'}), (1:0), (c^{'}-c: c^{'}), (0:1))$,
\item $\tilde{F}_{25, 12}(7)\to ((c:c^{'}), (0:1), (0:1), (c:c^{'}-c), (1:0))$,
\item $\tilde{F}_{13, 12}(7)\to  ((0:1), (0:1), (c:c^{'}), (1:1), (c^{'}:c))$,
\item $\tilde{F}_{14, 12}(7)\to  ((1:0), (c:c^{'}), (0:1), (0:1), (c-c^{'}:c))$,
\item  $\tilde{F}_{15, 12}(7)\to ((c:c^{'}), (1:0), (1:0), (1:0), (c:c-c^{'}))$,
\item $\tilde{F}_{12, 12}(7)\to ((1:1), (1:1), (1:1), (c:c^{'}), (c:c^{'}))$,
\item $\tilde{F}_{34, 12}(7) \to ((1:1), (c:c^{'}), (c:c^{'}),(1:0), (1:0))$,
\item $\tilde{F}_{35, 12}(7)\to ((c:c^{'}), (1:1), (c^{'}:c), (0:1), (0:1))$,
\item $\tilde{F}_{45, 12}(7)\to ((c:c^{'}), (c:c^{'}), (1:1), (1:1), (1:1)) $.  
\end{itemize}
\end{lem}

\begin{lem}
The  virtual spaces of parameters  $\tilde{F}_{ i, 12}\cong \{(c_1:c_1^{'}), (c_2:c_2^{'}), (c_3:c_3^{'})) | c_1c_2^{'}c_3=c_1^{'}c_2c_3^{'}, c_3, c_3^{'}\neq 0, c_3\neq c_3^{'}\}$ of the strata whose admissible polytopes are $O_{i}$ can be embedded into $\mathcal{G}$.
\end{lem}
\begin{proof}
The virtual space of parameters for $O_i$ is given by $\tilde{F}_{i, 12} =  F\cup (\cup _{j\neq i}\tilde{F}_{ij, 12})$. An  embedding of $\tilde{F}_{i}$ in $\mathcal{G}$ is given  by~\eqref{jedan},~\eqref{dva} and Lemma~\ref{devet}.
\end{proof}
In an  analogous way we verify an  existence of the embedding of  virtual spaces of parameters for all other strata. 
Since the virtual spaces of parameters $\tilde{F}_{ \sigma, 12}$ produce  a compactification of $F$ that coincides with  $\tilde {\mathcal{F}}$,     previous lemmas together with Proposition~\ref{embedd} and Proposition~\ref{image} imply:
\begin{thm}
The embedding $I : F\to (\C P^1)^{5}$ extends to an  embedding $\hat{I} : \tilde{\mathcal{F}} \to (\C P^{1})^{5}$, such that $\hat{I}(\tilde{\mathcal{F}}) \subset \mathcal{G}$.
\end{thm}

Since $\tilde{\mathcal{F}}$ and $\mathcal{G}$ are compact manifolds of the same dimension we deduce:

\begin{cor}
The universal space of parameters $\tilde{\mathcal{F}}$   is homeomorphic to the algebraic manifold $\mathcal{G}\subset (\C P^{1})^{5}$.
\end{cor}

\section{The orbit space $G_{5,2}/T^{5}$}

 \subsection{Summary}  We summarize results from previous sections about $F_{\sigma ,ij}$ and $\tilde{F}_{\sigma, ij}$ for all 
admissible polytopes $P_{\sigma}$. 
\begin{enumerate}
\item $P_{\sigma} =\Delta _{5,2}$ then $F_{\sigma } \cong \tilde{F}_{\sigma, ij} \cong  F$,
\item $P_{\sigma} = K_{pq}(9)$ then  $F_{\sigma }\cong \tilde{F}_{\sigma, ij} \cong \C P^{1}_{A}$,
\item $P_{\sigma } = O_{l}$ then   $F_{\sigma} \cong  \C P^{1}_{A}$ and $\tilde{F}_{\sigma, ij} \cong (\C P^1_{A}\times \C P^1\times \C P^{1})\cap\tilde{\mathcal{F}} $,
\item $P_{\sigma} = K_{kl, mn}$ then $ F_{\sigma} = \tilde{F}_{\sigma, ij}$ is a point,
\item $P_{\sigma}\neq \Delta _{5,2}, K_{pq}(9), O_{l}, K_{kl, mn}$ then $ F_{\sigma}$ is a point and $ \tilde{F}_{\sigma, ij}$ is not a point.
\end{enumerate}

\subsection{Description of the  orbit space $G_{5,2}/T^5$}

Let $\mathcal{P}$ be the formal union of all admissible polytopes:
\begin{equation}\label{P}
\mathcal{P} = \cup _{\sigma \in \mathcal{A}}P_{\sigma}.
\end{equation}
Denote by   $\tilde{p} : \mathcal{P} \to \Delta_{5,2}$ the canonical projection.
There is a canonical map $\tilde{\mu} : G_{5,2}/T^5 \to \mathcal{P}$ defined by
\begin{equation}\label{tildemu}
\tilde{\mu}(X) = x\in P_{\sigma}\;\; \text{if and only if}\;\; X\in W_{\sigma}/T^5\; \text{and}\; \tilde{p}(x) = \hat{\mu}(X).
\end{equation}
In other words $\tilde{\mu}$ is defined by
\[
\tilde{p}\circ \tilde{\mu} = \hat{\mu} .
\]
We assume  $\mathcal{P}$ to be equipped with the topology  induced by the map $\tilde{\mu}$: $U\subset \mathcal{P}$ is an open set if and only if $\tilde{\mu}^{-1}(U)$ is an open in $G_{5,2}/T^5$. Note that $\mathcal{P}$ is a compact space for this topology.

Let us consider the set  
\begin{equation}\label{spaceE}
\mathcal{E} = \cup _{\sigma\in \mathcal{A}}{\stackrel{\circ}P}_\sigma\times \widetilde{F}_{\sigma, ij}.
\end{equation}

The embeddings ${\stackrel{\circ}P}_{\sigma}\hookrightarrow \mathcal{P}$\, and\, $\widetilde{F}_{\sigma, ij}\hookrightarrow
\tilde{\mathcal{F}}$ define  a canonical   map $f: \mathcal{E}\to  \mathcal{P}\times \widehat{\mathcal{F}}$. It follows from  the definition of $\tilde{\mathcal{F}}$ and the description of  admissible polytopes that the   map $f$ is a surjection. We assume that $\mathcal{E}$ is equipped with the topology induced by  the map $f$,  that is $U\subset \mathcal{E}$ is an open set if and only if $f(U)$  is an open set in $\mathcal{P}\times \tilde{\mathcal{F}}$. The space $\mathcal{E}$ is a compact space for this topology. By Theorem~\ref{mainb}, considering all charts for $G_{5,2}$, we obtain  in this way     ten homeomorphic topological spaces .

Recall that  $G_{5,2}/T^5 = \cup _{\sigma \in \mathcal{A}}W_{\sigma}/T^{\sigma}$ and 
define a mapping  
\[ 
H : \mathcal{E} \to G_{5,2}/T^5, \;  \; H( {\stackrel{\circ}P}_{\sigma}\times
\widetilde{F}_{\sigma, ij})\subset  W_{\sigma}/T^{\sigma}
\]
by
\[
H(x,c) = h_{\sigma} ^{-1}(x, g_{\sigma, ij}(c)).
\]

Here  $h_{\sigma} : W_{\sigma}/T^{\sigma}\to {\stackrel{\circ}P}_{\sigma}\times F_{\sigma}$ is a canonical homeomorphism stated by Proposition~\ref{canon-trivial} and  $g_{\sigma , ij} : \widetilde{F}_{\sigma, ij}\to F_{\sigma}$ is a canonical projection stated by Corollary~\ref{canonic-proj}.

\begin{thm}\label{orbit-main}
The orbit space $G_{5,2}/T^5$ is a quotient space of the space $\mathcal{E}$ 
by an  equivalence relation defined by the map $H$.
\end{thm}
\begin{proof}

We prove that the map $H$ is continuous and that it is a surjection. Since the space $G_{5,2}$ is Hausdorff, it  will imply that the quotient space of $\mathcal{E}$ by the map $H$ and the orbit space $G_{5,2}/T^5$ are homeomorphic.  The fact that the map $H$  is continuous follows from the construction of  the space $\mathcal{E}$. 
We verify  it in a  couple of illustrative cases. 

Assume that we fixed the  $M_{12}$ and it is given a sequence $(x_n,c_n)\in \stackrel{\circ}{\Delta}_{5,2}\times \tilde{F}_{12}$, which converges to a point $(x_0, c_0)\in \stackrel{\circ}{K}_{13}\times \tilde{F}_{13, 12}$. Since
$c_n= ((c_{1n}:c_{1n}^{'}), (c_{2n}:c_{2n}^{'}), (c_{3n}:c_{3n}^{'}))$ and $c_0 = ((1:0), (1:0), (c:c^{'}))$, we obtain  that $c_{1n}^{'}, c_{2n}^{'}\to 0$.  Let $p_{n} = h^{-1}(x_n,c_n)\in W/T^5$, then $p_n$ writes in the chart $M_{12}$
as  $p_n= [(z_{1n},\ldots ,z_{6n})]$ and it holds $c_{1n}^{'}z_{1n}z_{5n} = c_{1n}z_{2n}z_{4n}$, $c_{2n}^{'}z_{1n}z_{6n} = c_{2n}z_{3n}z_{4n}$,  $c_{3n}^{'}z_{2n}z_{6n} = c_{3n}z_{3n}z_{4n}$. The conditions that $c_{1n}^{'}, c_{2n}^{'}\to 0$ and that $\hat{\mu}(p_n) = x_n\to x_0\in \stackrel{\circ}{K}_{13}$ implies that $z_{4n}\to 0$. 
Let us consider a sequence $q_n$ in $G_{5,2}/T^5$ which is given in the chart $M_{12}$ by $q_n = [(z_{1n}, z_{2n}, z_{3n}, 0, z_{5n}, z_{6n})]$. Note that apart from the Pl\"ucker coordinate $P^{13}$ all other Pl\"ucker coordinates for $p_n$ and $q_n$ coincide and $P^{13}(p_n) \to 0 = P^{13}(q_n)$. It implies that $\lim \hat{\mu}(q_n) = \lim \hat{\mu}(p_n) = \lim x_n = x_0$. The   sequence $q_n$ belongs to the stratum whose admissible polytope is  $K_{13}$ and $h_{13,12}(q_n) = (y_n, (c_{3n}:c_{3n}^{'}))$, where $y_n = \hat{\mu}(q_n)$. Since $h_{13,12} : W_{13}/T^5\to \stackrel{\circ}{K}_{13}\times \C P^{1}_{A}$ is a homeomorphism and $ (y_n, (c_{3n}:c_{3n}^{'}) )\to (x_0, (c:c^{'})) \in \stackrel{\circ}{K}_{13}\times \C P^{1}_{A}$, it implies that  $q_n\to p_0 = h_{13,12}^{-1}(x_0, (c:c^{'}))$. Then the sequence $p_n$ is convergent as well and
$p_{n} \to h_{13,12}^{-1}(x_0, f_{13, 12}(c_0)) = H(x_0, c_0)$. 

Let us consider now a  sequence $(x_n,c_n)\in \stackrel{\circ}{\Delta}_{5,2}\times \tilde{F}_{12}$  converging to a point $(x_0, c_0)\in \stackrel{\circ}{K}_{13}(7)\times \tilde{F}_{13, 12}(7)$. Then $c_0 = ((1:0), (1:0), (c:c^{'}))$ and, as previously, it holds $c_{1n}^{'}, c_{2n}^{'}\to 0$. Now, if we consider the sequence $p_n = [(z_{1n},\ldots ,z_{6n})] = h(x_n,c_n)$  in the chart $M_{12}$,  the condition that  $\hat{\mu}(p_n) = x_n\to x_0\in \stackrel{\circ}{K}_{13}(7)$ implies that $z_{2n}, z_{3n}\to 0$. The sequence $q_n = [(z_{1n}, 0, 0, z_{4n}, z_{5n}, z_{6n})]$ belongs to the stratum whose admissible polytope is $K_{13}(7)$ and $\lim \hat{\mu}(q_n) = \lim \hat{\mu}(p_n) = \lim x_n = x_0$. Since the space of parameters $F_{13}(7)$ is a point,  it follows  that $\tilde{\mu} : W_{13}(7)/T^5 \to \stackrel{\circ}{K}_{13}(7)$ is a homeomorphism. Therefore, as  $\tilde{\mu}(q_n) = y_n \to x_0 \in \stackrel{\circ}{K}_{13}(7)$ we conclude  that $\tilde{\mu}^{-1}(y_n)\to \tilde{\mu}^{-1}(x_0) = p_0$. It implies that  the sequence $p_n$ is convergent  and $p_n\to p_0 = H(x_0, c_0)$. 

Let us consider now a  sequence $(x_n,c_n)\in \stackrel{\circ}{\Delta}_{5,2}\times \tilde{F}_{12}$ converging to a point $(x_0, c_0)\in \stackrel{\circ}{O}_{3}\times \tilde{F}_{3, 12}$. Then $c_0 = ((c_1:c_1^{'}), (c_2:c_{2}^{'}), (c_{3}:c_{3}^{'}))$, where $(c_i:c_i^{'})\in \C P^1$, $i=1,2$, and $(c_3:c_3^{'})\in \C P^1_{A}$.  Taking $p_n = [(z_{1n},\ldots ,z_{6n})] = h(x_n,c_n)$  in the chart $M_{12}$,  the condition that  $\hat{\mu}(p_n) = x_n\to x_0\in \stackrel{\circ}{O}_{3}$ and that $O_3$ is a facet of $\Delta _{5,2}$  implies $z_{1n}, z_{4n}\to 0$. The sequence $q_n = [(0, z_{2n}, z_{3n}, 0, z_{5n}, z_{6n})]$ belongs to the stratum whose admissible polytope is $O_3$ and $\lim \hat{\mu}(q_n) = \lim \hat{\mu}(p_n) = \lim x_n = x_0$. Further,  $h_{3,12}(q_n) = (y_n, (c_{3n}:c_{3n}^{'}))$, where $y_n = \hat{\mu}(q_n)$. Since $(y_n, (c_{3n}:c_{3n}^{'})) \to (x_0, (c_3:c_3^{'}))$,  we conclude that $q_{n}\to h_{3,12}^{-1}(x_0, (c_3:c_3^{'})) = p_{0}$. Therefore, $p_{n}\to p_{0} = h_{3,12}^{-1}(x_0, f_{3,12}(c_0)) = H(x_0, c_0)$.

Thus,  $H$ is continuous. Since $G_{5,2}/T^5 =\cup_{\sigma\in \mathcal{A}} W_{\sigma}/T^5$ and $H(\stackrel{\circ}{P}_{\sigma}\times \tilde{F}_{\sigma, ij}) = W_{\sigma}/T^5$, the map $H$ is onto, which  implies that it induces a homeomorphism between the quotient space of $\mathcal{E}$ defined by $H$ and the orbit space $G_{5,2}/T^5$. 
\end{proof}

 \begin{cor}
The equivalence relation on $\mathcal{E}$ defined by the map $H$  is given by  the equivalence relation on each $\stackrel{\circ}{P}_{\sigma}\times \tilde{F}_{\sigma , ij}$:
\[
(x_1,c_1)\backsimeq (x_2,c_2) \;\; \text{if and only if}\;\; x_1=x_2, \; g_{\sigma, ij}(c_1) = g_{\sigma, ij}(c_2).
\]
\end{cor}
Note that  the equivalence  relation $\backsimeq$ is trivial for $P_{\sigma} = \Delta _{5,2}, K_{ij}(9), K_{ij,kl}$ since $\tilde{F}_{\sigma, ij} \cong F_{\sigma }$. This relation is non-trivial for the  octahedra $O_{i}$,  pyramids $K_{ij}(7)$ and prisms $P_{ij}$.  

\begin{cor}\label{embedG}
The orbit space $G_{5.2}/T^5$ is homeomorphic to  
$\mathcal{E}/\approx$ , where 
\begin{equation}\label{final}
(x_{\sigma _1}, c_{\sigma _1, ij})\backsimeq (x_{\sigma _2}, c_{\sigma _2, ij})\;\; \text{ if and only if}\;\; \sigma _1=\sigma _2, \;\; x_{\sigma _1} = x_{\sigma _2}\;\;  \text{and}\;\; g_{\sigma _1, ij}(c_{\sigma _1, ij}) = g_{\sigma _2, ij}(c_{\sigma _2, ij}).
\end{equation} 
\end{cor}

\begin{rem}
The formula~\eqref{final} is explicit. Namely, while $F_{\sigma}$ are  abstract spaces of   parameters of the strata, the spaces  $\tilde{F}_{\sigma, ij}$ are explicit  subspaces of $\tilde{\mathcal{F}}$ which makes this formula applicable. Theorem~\ref{bound-univ},  Proposition~\ref{parametrization-9}, Proposition~\ref{parametrization-8} and Proposition~\ref{parametrization-7}
provide the description of  $\tilde{F}_{\sigma , 12}$ when  $P_{\sigma}$ is an  admissible polytope which belongs to the interior of $\Delta _{n,k}$.
\end{rem}

Let us consider now the projection $\mathcal{P}\times \tilde{\mathcal{F}}\to \Delta _{5,2}\times \tilde{\mathcal{F}}$.  Define an equivalence relation on  $\Delta _{5,2}\times \tilde{\mathcal{F}}$ by
\begin{equation}\label{relP}
(x, c_{1,ij})\backsimeq (x, c_{2,ij})\; \text{ if and only if }\; c_{1,ij},c_{2,ij}\in \tilde{F}_{\sigma , ij}, \; x\in \stackrel{\circ}{P_{\sigma}} \; \text{and}\;  g_{\sigma , ij}(c_{1,ij}) = g_{\sigma ,  ij}(c_{2, ij}).
\end{equation}
It follows from~\eqref{final} and~\eqref{relP}  that there is a  canonical projection from  the space  $(\mathcal{P}\times \tilde{\mathcal{F}})/\backsimeq$ to the space $(\Delta_{5,2}\times \tilde{\mathcal{F}})/\backsimeq$. Corollary~\ref{embedG} implies
\begin{cor}
There is a canonical  continuous map  $G_{5,2}/T^5\to (\Delta _{5,2}\times \tilde{\mathcal{F}})/\backsimeq$, where $\backsimeq$ is an equivalence relation defined by~\eqref{relP}.
\end{cor}

\begin{ex}\label{42}
We demonstrate the application of the formula~\eqref{final} in the case of  Grassmann manifold $G_{4,2}$ whose orbit space is described in~\cite{MMJ}. In this case, as it follows from~\cite{MMJ},  the admissible polytopes are the octahedron $\Delta _{4,2}$, the six pyramids $P_{i}$, $1\leq i\leq 6$, the three rectangles $R_{12}$, $R_{34}$, $R_{56}$ and the faces of $\Delta _{4,2}$. We assume the  numeration to be such that $P_1\cap P_{2} = R_{12}$, $P_{3}\cap P_{4} = R_{34}$ and $P_{5}\cap P_{6} = R_{56}$. It is proved in~\cite{MMJ} that $\tilde{F}_{12}=  F_{12}= \C \setminus \{0,1\}$. Denote by $\tilde{F}_{12}(P_{\sigma})$ the virtual space of parameters of an admissible polytope $P_{\sigma}$ in the chart $M_{12}$. It follows from~\cite{MMJ} that   $\tilde{F}_{12}(P_{1}) = \tilde{F}_{12}(P_{2})=0$, $\tilde{F}_{12}(P_{3}) = \tilde{F}_{12}(P_{4}) = 1$, $\tilde{F}_{12}(P_{5}) = \tilde{F}_{12}(P_{6}) =\infty$, then  $\tilde{F}_{12}(R_{12}) =0$, $\tilde{F}_{12}(R_{34}) =1$ and $\tilde{F}_{12}(R_{56})=\infty$, while $\tilde{F}_{12}(P_{I}) =\C P^1$ for the faces of $\Delta _{4,2}$. Then formula~\eqref{final} gives
\[
G_{4,2}/T^4 \cong (\cup_{I} \stackrel{\circ}{P_{I}} \times \tilde{F}_{ij}(P_{I}))/\backsimeq,
\]
where $(x_{I}, c_{I}) \backsimeq(x_{J}, c_{J})$ if and only if $I=J$ and $x_{I} =x_{J}\in \partial\Delta _{4,2}$. This can be further written as
\[
G_{4,2}.T^4\cong (\Delta _{4,2}\times \C P^1)/\backsimeq,
\]
where $(x,c)\backsimeq (x^{'}, c^{'})$ if and only if  $x=x^{'}\in \partial \Delta _{4,2}$, which is exactly the formula obtained in~\cite{MMJ}.
\end{ex}

\begin{rem}
Let us consider in $\C ^{5}$  coordinates subspaces of the dimensions $k=2,3,4$ and the corresponding embeddings of $G_{k, 2}$ in $G_{5,2}$. Since an embedding of a coordinate subspace is equivariant  for the coordinate vise action of the torus $T^5$ we obtain an equivariant embedding of $G_{k,2}$ in $G_{5,2}$. In this way the embeddings of the corresponding orbit spaces are defined as well. For $k=4$ we obtain the  embeddings of the five $5$-dimensional spheres in $G_{5,2}/T^5$. Moreover, we know how these spheres are glued together in $G_{5,2}/T^2$.  Any pair of these $4$-dimensional coordinate subspaces intersect each other  in $\C ^5$ along a coordinate subspace $\C ^{3}$. A  coordinate subspace $\C ^{3}$  produce in $G_{5,2}$ the complex projective space $\C P^{2} = G_{3,2}$ whose orbit space, by the torus action, is the $2$-dimensional simplex. In this way, using the cell decomposition of $G_{4,2}/T^4$ described in~\cite{MMJ}, we obtain that the orbit spaces $G_{4,2}/T^4$ are glued in $G_{5,2}/T^2$ along the  $2$-dimensional simplices.
\end{rem}

\section{The homology groups of $G_{5,2}/T^5$}
We compute the homology groups of the orbit space $G_{5,2}/T^5$ appealing  to Theorem~\ref{orbit-main} and  results on the description of the orbit spaces of the strata, their admissible polytopes, spaces of parameters and virtual spaces of parameters. 

We first compute the top degree homology group for $G_{n,2}/T^n$ for $n\geq 4$.
Let $V_2 = \tilde{\mu}^{-1}(\cup P_{\sigma})$, where $P_{\sigma}$ runs through all  admissible polytopes different from  $\Delta _{5,2}$.  Then $V_2$ is  a closed subset in $G_{n,2}/T^n$ and $(G_{5,2}/T^n)\setminus V_2 = W/T^n$ is  the orbit space of the main stratum. Since $W/T^n$ is a dense set in $G_{5,2}/T^5$, it follows that $(G_{n,2}/T^n)/V_2$ is the Alexandrov  one-point compactification of $W/T^n$.  Recall that $W/T^n\cong \stackrel{\circ}{\Delta}_{n,2}\times F$   and denote its  one-point compactification by  $(\stackrel{\circ}{\Delta}_{n,2}\times  F)^{\star}$. Note that $F$ is homeomorphic to $\C P^{1}_{A}\times \C P^{1}_{A}$, which  implies that  $\Delta _{5,2}\times (\C P^1\times \C P^1)$ is a  compactification of  $\stackrel{\circ}{\Delta}_{n,2}\times  F$. 
 Proposition~\ref{em} implies that    
\[(\stackrel{\circ}{\Delta}_{n,2}\times F_{n})^{\star}  \cong  (\Delta _{5,2}\times (\C P^{1})^{n-3})/((\partial \Delta _{n,2}\times (\C P^{1})^{n-3})\cup (\Delta _{5,2}\times G_{n})).
\]
 Using this we prove:
\begin{thm}
The top degree homology group $H_{3n-7}(G_{n,2}/T^n)$ is isomorphic to $\Z$.
\end{thm}
\begin{proof}
The dimension of the space of  parameters $F_{\sigma}$ for any $P_{\sigma}\neq \Delta _{5,2}$ is by  Proposition~\ref{enm}  less then or equal $2n-8$. It implies that the dimension of $V_2$ is   $n-1+2n-8 = 3n-9$, which  further gives that $H_{ 3n-7}(V_2) = H_{3n-8}(V_2)=0$. Then the exact homology sequence of the pair $(G_{n,2}/T^n, V_2)$  implies that $H_{3n-7}(G_{n,2}/T^n) = H_{3n-7}((G_{n,2}/T^n)/V_2) \cong H_{3n-7}((\stackrel{\circ}{\Delta} _{n,2}\times F)^{\star})$. Let us  consider  a  homology sequence of the pair $(X, Y)$ where  $X=\Delta _{5,2}\times (\C P^{1})^{n-3}$ and $Y=(\partial \Delta _{n,2}\times (\C P^{1})^{n-3})\cup (\Delta _{5,2}\times G_{n})$. It holds  that $H_{i}(X) = 0$ for $i> 2n-6$, which  implies that $H_{3n-7}(X) =H_{3n-8}(X) =0$.  It follows from Proposition~\ref{em}  that the dimension of $Y$ is equal to $3n-8$ and  $H_{3n-8}(Y)=\Z $. Then a  homology sequence gives that $H_{3n-7}(X/Y)\cong \Z$,
which  proves the statement.
\end{proof}

 We proceed with the computation of the other homology groups for the case $n=5$, that is for the space $G_{5,2}/T^5$.
For that purpose we will consider a  filtration $V_1\subset V_2\subset V_{3}=G_{5,2}/T^5$
and compute the corresponding relative and absolute homology group. In the course of doing this  we describe the cell decomposition of $G_{5,2}/T^5$.
\subsection{$V_1$ and its homology groups}
Let $V_{1} =  \hat{\mu}^{-1}(\partial \Delta _{5,2}) =  \cup \tilde{\mu}^{-1}(P_{\sigma}) $, where the union goes over all $P_{\sigma}\in \mathcal{P}$ such that $P_{\sigma}\subset \partial \Delta _{5,2}$.  In other words,  $V_1$  is the  union of the orbit spaces of the strata whose admissible polytopes belong to $\partial \Delta _{5,2}$.   The boundary $\partial \Delta _{5,2}\cong S^3 = \cup _{i=1}^{5}(O_i\cup T_i)$,  so $T_{i}$ and $O_{i}$, $1\leq i\leq 5$ provide   the combinatorial decomposition of $S^3$.  Thus,  $V_1 \cong \cup _{i=1}^{5}\hat{\mu}^{-1}(T_{i}) \cup \cup _{i=1}^{5}\hat{\mu} ^{-1}(O_{i})$ and we recall that $\hat{\mu} ^{-1}(T_{i})\cong T_i$ and $\hat{\mu} ^{-1}(O_i)\cong S^5 \cong (O_{i}\times \C P^{1})/\backsimeq$, where $\partial O_{i}\times \C P^1 \backsimeq \partial O_{i}$.  Therefore,  there is  the  continuous map $\hat{\mu} : V_1\to S^3$, but also the section $S^3\to V_1$. On $T_i$ this section is given by $\hat{\mu}^{-1}$, while on  $O_i$ it  is given by  the composition of $\hat{\mu}^{-1}$ with  the projection  on a fixed  parameter  $\ast\in \C P^1$. In this way we obtain
\begin{equation}\label{V1p}
V_1 = S^3 \cup  \cup _{i=1}^{5}(O_i\times (\C P^{1}\setminus \{\ast\}))/\backsimeq, \;\; \partial O_i\times (\C P^{1}\setminus \{\ast\})\backsimeq \partial O_{i}.
\end{equation}
Note that $(O_i\times (\C P^{1}\setminus \{\ast\}))/\backsimeq$ is homeomorphic to  $ S^{5}\setminus {D}^{3}$, where $\backsimeq$ is a relation defined as in~\eqref{V1p} and  $\partial \stackrel{\circ}{D^{3}} = \partial O_i\subset S^3$.
It implies that  
\[V_1/S^3\cong \vee  _{5}S^5.
\]

\begin{lem}
The nontrivial homology groups for the space $V_1$ are $H_{5}(V_1) =\Z ^{5}$ and  $H_{3}(V_1) = H_{0}(V_1) = \Z$.
\end{lem}
The universal coefficient theorem implies:
\begin{cor}
The nontrivial homology groups for $V_1$ with $\Z _{2}$-coefficients are  $H_{5}(V_1;\Z _{2}) =\Z _{2}^{5}$ and  $H_{3}(V_1;\Z _{2}) = H_{0}(V_1;\Z _{2}) = \Z_{2}$.
\end{cor}
\begin{lem}
There is an  induced action of the symmetric group $S_5$ on the homology groups for  $V_1$ with one orbit in each homology group.
\end{lem}
\begin{proof}
It follows from Lemma~\ref{fund-5} that the symmetric group $S_5$  permutes the tetrahedra and permutes the octahedra which means  that the sphere $S^3$ is invariant for this  action.   Therefore, there is an induced action of $S_5$ on the   homology groups for $V_1$ given as follows: it acts trivially on $H_{3}(V_1)$ and $H_{0}(V_1)$, while  its action on  $H_{5}(V_1)$  is induced by   the permutation of the octahedra.
\end{proof}
\subsection{$V_2$ and its homology groups} 
Let $V_{2} = \cup \tilde{\mu} ^{-1}(P_{\sigma})$, where the union goes over all admissible polytopes  $P_{\sigma}\in \mathcal{P}$ , which are different from $ \Delta _{5,2}$. Then  $V_1\subset V_{2}$ and set 
$V_{21} = V_{2}/V_1$. In order to compute  the homology groups for $V_{21}$ we consider its   filtration   $L_1\subset L_{2}\subset V_{21}$, where  subspaces $L_1$ and $L_2$ are given as follows.   The space $L_2$ is a projection on $V_{21}$ of the  union of the orbit spaces of the strata whose admissible polytopes are  different from $K_{ij}(9)$ and $\Delta _{5,2}$, that is 
 $L_2 = \cup _{\sigma} \tilde{\mu} ^{-1}(P_{\sigma})/(\cup  _{\sigma} \tilde{\mu} ^{-1}(P_{\sigma})\cap V_1)$, where  $P_{\sigma}$ goes through   admissible polytopes  such that  $P_{\sigma}\neq K_{ij}(9), \Delta _{5,2}$.  The space  $L_1$ is a projection on $V_{21}$ of  the union of the orbit spaces of the strata over the admissible prisms and their faces,  that is $L_1= =\cup _{i=1}^{10}\tilde{\mu}^{-1}(P_{i})/ (\cup _{i=1}^{10}\tilde{\mu} ^{-1}(P_{i})\cap V_1)$ .
Recall that there are $10$ prisms, the  boundary of any prism belongs to $\partial \Delta _{5,2}$  and the space of parameters for the  stratum over any prism is a  point. It implies that 
\[
L_1 = \vee _{10}S^3.
\]
It follows that  the nontrivial homology groups for $L_1$ are $H_{3}(L_1) = \Z ^{10}$ and $H_{0}(L_1)=\Z$.  The universal coefficient theorem implies that the nontrivial homology groups with $\Z _2$-coefficients are $H_{3}(L_1;\Z _2)=\Z _{2}^{10}$ and
$H_{0}(L_1;\Z _{2}) = \Z _{2}$. 
\begin{lem}
There is an  action of $S_5$ on the  homology groups for $L_1$ induced by the  action of $S_5$ on  the prisms which is given by the composition of the representation of $S_5$ in  $S_{10}$ and the action of $S_{10}$ which permutes the prisms.
\end{lem}
\begin{proof}
The symmetric group $S_5$ permutes the prisms and the corresponding strata, which  implies that $S_5$  acts by  permutations on $H_{3}(L_1)$ and it acts  trivially on $H_{0}(L_1)$.
\end{proof}

It holds  $L_1\subset L_2$. Moreover,  admissible polytopes for the strata from $L_2$ are the prisms $P_l$, the pyramids $K_{pq}(7)$ and  the polytopes with eight vertices $K_{ij,kl}$. By Proposition~\ref{bound-8} and Proposition~\ref{bound-7},  a  facet of any of polytopes $K_{pq}(7)$ or $K_{ij,kl}$ is either $P_l$ either it belongs to $\partial \Delta _{5,2}$. There are altogether $25$  polytopes of the types  $K_{pq}(7)$ or  $K_{ij,kl}$ and    the space of parameters  for   any of these polytopes is a point. It implies that
\[
L_2/L_1 = \vee _{25}S^4.  
\]
It follows that the nontrivial homology groups for  $L_2/L_1$ are $H_{4}(L_2/L_1) = \Z ^{25}$ and $H_{0}(L_2/L_1) = \Z$. Again by the universal coefficient theorem the nontrivial homology groups for $L_2/L_1$ with $\Z _{2}$-coefficients are  $H_{4}(L_2/L_1;\Z _{2}) = \Z _{2}^{25}$ and $H_{0}(L_2/L_1;\Z _{2}) = \Z _{2}$.

\begin{lem}
There is an  induced action of the symmetric group $S_5$ on  the homology groups for $L_2/L_1$ which has two orbits .
\end{lem}
\begin{proof}
The symmetric groups  $S_5$ permutes $K_{pq}(7)$,  permutes $K_{ij, kl}$ and the corresponding strata. Therefore, there is an induced action of $S_5$ on $H_{4}(L_2/L_1)$ which is given by the permutation  of  elements  which are determined  by $K_{pq}(7)$ and by the  permutation of  elements 
which are determined by  $K_{ij, kl}$. Thus,  there are two orbits for $S_5$-action, one consisting of $10$ elements  and the other one consisting of $15$  elements.
\end{proof}

\begin{lem}\label{L2}
The nontrivial homology groups for $L_2$ are $H_{4}(L_2) = \Z ^{15}$ and $H_{0}(L_2) = \Z$.
\end{lem}
\begin{proof}
 The exact homology sequence of the pair $(L_2, L_1)$ gives directly  that  $H_{1}(L_2) = H_{2}(L_2) = H_{i}(L_2)=0$ for $i\geq 5$.  In order to determine $H_{3}(L_2)$ and $H_{4}(L_2)$ we consider a cell decomposition for $L_2$. The three dimensional cells in this cell decomposition are given by  the admissible  prisms $P_l$  over six vertices, while the four-dimensional cells are given by  the admissible polytopes $K_{ij,kl}$  and $K_{pq}(7)$. A  facet of $K_{pq}(7)$  which belongs to $\stackrel{\circ}{\Delta}_{5,2}$  is some  prism $P_{l}$.  It implies  that the boundary in $L_2$  of the cell  $K_{pq}(7)$ consists of  the 
three-dimensional cell  $P_l$ and a  point. The number of the pyramids $K_{pq}(7)$ is $10$  and it is  equal to the number of the prisms $P_{l}$ and, obviously, different pyramids  $K_{pq}(7)$ have different prisms in their boundaries. In this way we deduce that $H_{3}(L_2) = 0$.  Then the exact homology sequence implies  that $H_{4}(L_2) = \Z ^{15}$. 
\end{proof}

By the universal coefficient theorem it follows:
\begin{cor}
The  nontrivial homology groups for $L_2$ with $\Z_{2}$-coefficients are $H_{4}(L_2;\Z _{2}) = \Z _{2}^{15}$ and $H_{0}(L_2;\Z _{2}) = \Z _{2}$.
\end{cor}
\begin{lem}
The symmetric group $S_5$ acts on $H_{4}(L_2)$ by  permutations and this action  is induced by the  $S_5$-action on the set $\{K_{ij,kl}\}$.
\end{lem}
The  union $K$  of the orbit spaces of  the $10$  strata over $\stackrel{\circ}{K_{ij}}(9)$, $1\leq i<j\leq 5$ is a dense set in $V_{21}$ and by Proposition~\ref{bound-9},  $K= V_{21}\setminus L_2$ .  It implies that $K^{\star} = V_{21}/L_2$  is an  one-point  compactification of $K$. 

 Recall that the orbit space of  each of these strata  is homeomorphic to $\stackrel{\circ}{K_{ij}(9)}\times \C P^{1}_{A}$,
where $A=\{(1:0), (0:1), (1:1)\}$. 

\begin{lem}
The one-point compactification $K_{ij}^{\star}(9)$  of the orbit space   $\stackrel{\circ}{K_{ij}(9)}\times \C P^{1}_{A}$ in $V_{21}/L_2$ is given by 
\begin{equation}\label{first}
(K_{ij}(9)\times \C P^1)/\backsimeq , \;\; \text{where}\;  \partial K_{ij}(9)\times \C P^1 \backsimeq K_{ij}(9)\times A\backsimeq \ast.
\end{equation} 
Moreover, 
\begin{equation}\label{second}
K_{ij}^{\star}(9)\cong S^6 \vee S^5\vee S^5.
\end{equation}
\end{lem}
\begin{proof}
The statement follows from the observation  that $K_{ij}(9)\times \C P^1$ is a compactification of     $\stackrel{\circ}{K_{ij}(9)}\times \C P^{1}_{A}$ and  $(K_{ij}(9)\times \C P^1)\setminus  (\stackrel{\circ}{K_{ij}(9)}\times \C P^{1}_{A}) =   (\partial K_{ij}(9)\times \C P^1)\cup  (K_{ij}(9)\times A)$.  
We want to emphasize that the one-point compactification $K_{ij}^{\star}(9)$ of $\stackrel{\circ}{K}_{ij}(9)\times \C P^1_{A}$ in $V_{21}/L_2$   has an explicit geometric realization. We demonstrate   it for the admissible polytope $K_{12}(9)$ and   due  to the action of the symmetric group $S_5$, it will then hold for any other polytope $K_{ij}(9)$.  It follows from Proposition~\ref{bound-9} that the boundary of the stratum  over $K_{12}(9)$ contains  the four-dimensional strata over the polytopes  $K_{12,34}, K_{12,35}, K_{12, 45}$, $K_{34}(7), K_{35}(7)$ and $K_{45}(7)$. It holds  $K_{12, 34}\cup K_{34}(7) = K_{12,35}\cup K_{35}(7)=K_{12,45}\cup K_{45}(7) = K_{12}$. Moreover, since  the stratum over $K_{12}(9)$ is parametrized by $\C P^{1}_{A}$,  it directly follows that the pairs $(K_{12,34}, K_{34}(7))$, $(K_{12,35}, K_{35}(7))$ and $(K_{12}, K_{45}(7))$ are parametrized by the points $(1:0), (0:1)$ and  $(1:1)$, respectively.  In addition any stratum over a  face of $K_{12}(9)$ is in the boundary of  the stratum over $K_{12}(9)$ which  completes the proof.  The second statement~\eqref{second} follows from the observation that $K^{\star}_{ij}(9)\cong S^3\ast  (K_{ij}(9)\times \C P^1)/(K_{ij}(9)\ast A) \cong  S^3\ast (S^2\vee S^1\vee S^1)\cong  S^6\vee S^5\vee S^5$.
\end{proof}
\begin{rem}\label{cycles}
We will further use  a  cell decomposition for $K_{ij}^{\star}(9)$ which is for our purposes more geometrical, that is $K_{ij}^{\star}(9) \cong \stackrel{\circ}{K_{ij}}(9)\times  ((S^2\vee S^1\vee S^1)\setminus \{*\})\cup \{\ast\}$, which  implies that the six-cell is $e_{ij} = \stackrel{\circ}{K_{ij}}(9)\times  (S^2\setminus \{\ast\})$, while the five-cells are $f_{ij}^{l}= \stackrel{\circ}{K_{ij}}(9)\times  p_{l}$, where   $p_{1}$ and $p_{2}$ are   non-intersecting paths  in $\C P^1$ such that  $p_{1}$ connects the points $(0:1)$ and $(1:1)$ and $p_{2}$ connects the points $(1:1)$ and $(1:0)$. Note that the cells $e_{ij}$ and $f_{ij}^{1}$, $f_{ij}^{2}$ represent  six-dimensional  and five-dimensional cycles in $K_{ij}^{\star}$.
\end{rem}
It follows from~\eqref{first} that 
\[
V_{21}/L_{2} = \cup _{1\leq i<j\leq 5} (K_{ij}(9)\times \C P^1)/\backsimeq , \;\; \text{where}\;  \partial K_{ij}(9)\times \C P^1 \backsimeq K_{ij}(9)\times A\backsimeq \ast . 
\]
Then ~\eqref{second} implies that

\begin{equation}\label{V21L2second}
V_{21}/L_{2} \cong (\vee _{10}S^6)\vee (\vee _{20}S^5).
\end{equation}
The  description of the nontrivial homology groups for $V_{21}/L_2$ directly follows from~\eqref{V21L2second}:
 \[
H_{6}(V_{21}/L_2) = \Z ^{10},\;\;  H_{5}(V_{21}/L_2) = \Z^{20}\; \text{and} \;  H_{0}(V_{21}/L_2) =\Z.
\]
\begin{rem}\label{inv}
The symmetric group $S_5$ acts on the set of  the orbit spaces $\stackrel{\circ}{K}_{ij}(9)\times \C P^{1}_{A}$ by permutations.  We assume the expression~\eqref{V21L2second} to be invariant under this action. Then there is an induced action of  $S_5$ on  the homology groups for $V_{21}/L_2$. Note that this action on $H_{6}(V_{21}/L_2)$ has one orbit, while on  $ H_{5}(V_{21}/L_2)$ it has two orbits each of them containing $10$ elements.
\end{rem}

In order to avoid an  orientation issue we compute further the  homology groups with coefficients in $\Z _{2}$.

\begin{prop}\label{V12}
The nontrivial homology groups for $V_{21}$ with $\Z _{2}$-coefficients  are:
 \[ H_{6}(V_{21};\Z _{2}) = \Z _{2}^{10},\;\; H_{5}(V_{21};\Z _{2})=\Z _{2}^{5},\;\; H_{4}(V_{21};\Z _{2}) = \Z _{2} \; \text{and}\; 
H_{0}(V_{21};\Z _{2}) = \Z_{2}.
\] 
\end{prop}
\begin{proof}
Since $H_{i}(L_2)=0$ for $i\geq 5$, the exact homology sequence of the pair $(V_{21}, L_2)$ directly implies that $H_{i}(V_{21}) = 0$ for $i\geq 7$ and $H_{6}(V_{21})= \Z ^{10}$. In order to determine $H_{5}(V_{21})$  we note that  all five-dimensional  cycles  in $V_{21}$ comes from $V_{21}/L_{2}$, since $L_2$ has no five-dimensional cells. It follows from Remark~\ref{cycles} that the five dimensional cycles  in $V_{21}/L_2$   are of the form  $\sum _{1\leq i<j\leq 5}(\alpha _{ij}^{1}f_{ij}^{1}+\alpha _{ij}^{2}f_{ij}^{2})$. 
 The boundaries    $\partial _{4}^{V_{21}}f_{ij}^{l}$, $l=1,2$,  in $V_{21}$ are  four-dimensional cycles  consisting of the orbit spaces of  four-dimensional strata in  the boundary of the strata over $\stackrel{\circ}{K}_{ij}(9)$. Note that we do not need to take care about an  orientation issue since we work with $\Z _{2}$-coefficients.
It  follows from Proposition~\ref{parametrization-9}, Proposition~\ref{parametrization-8} and Proposition~\ref{parametrization-7} and Theorem~\ref{bound-univ}  that  these boundaries are as follows:
\[ 
 f_{12}^{1} \to \stackrel{\circ}{K}_{12,34}+\stackrel{\circ}{K}_{12,45}+\stackrel{\circ}{K}_{34}+\stackrel{\circ}{K}_{45}, \; f_{12}^{2} \to \stackrel{\circ}{K}_{12,45}+\stackrel{\circ}{K}_{12,35}+\stackrel{\circ}{K}_{45}+\stackrel{\circ}{K}_{35},
\]  
\[ 
f_{13}^{1} \to \stackrel{\circ}{K}_{13,24}+\stackrel{\circ}{K}_{13,45}+\stackrel{\circ}{K}_{24}+\stackrel{\circ}{K}_{45},\; f_{13}^{2} \to \stackrel{\circ}{K}_{13,45}+\stackrel{\circ}{K}_{13,25}+\stackrel{\circ}{K}_{45}+\stackrel{\circ}{K}_{25},
\]
\[ 
f_{14}^{1} \to\stackrel{\circ}{K}_{14,23}+\stackrel{\circ}{K}_{14,35}+\stackrel{\circ}{K}_{23}+\stackrel{\circ}{K}_{35},\; f_{14}^{2} \to\stackrel{\circ}{K}_{14,35}+\stackrel{\circ}{K}_{14,25}+\stackrel{\circ}{K}_{35}+\stackrel{\circ}{K}_{25},
\]
 \[f_{15}^{1} \to \stackrel{\circ}{K}_{15,23}+\stackrel{\circ}{K}_{15,34}+\stackrel{\circ}{K}_{23}+\stackrel{\circ}{K}_{34},\; f_{15}^{2} \to \stackrel{\circ}{K}_{15,34}+\stackrel{\circ}{K}_{15,24}+\stackrel{\circ}{K}_{34}+\stackrel{\circ}{K}_{24},
\] 
\[
f_{23}^{1} \to \stackrel{\circ}{K}_{15,23}+\stackrel{\circ}{K}_{23,45}+\stackrel{\circ}{K}_{15}+\stackrel{\circ}{K}_{45},\;  f_{23}^{2} \to \stackrel{\circ}{K}_{23,45}+\stackrel{\circ}{K}_{14, 23}+\stackrel{\circ}{K}_{45}+\stackrel{\circ}{K}_{14},
\]
 \[f_{24}^{1} \to \stackrel{\circ}{K}_{15, 24}+\stackrel{\circ}{K}_{24,35}+\stackrel{\circ}{K}_{15}+\stackrel{\circ}{K}_{35},\; f_{24}^{2} \to \stackrel{\circ}{K}_{24,35}+\stackrel{\circ}{K}_{13,24}+\stackrel{\circ}{K}_{35}+\stackrel{\circ}{K}_{13},
\] 
\[f_{25}^{1} \to \stackrel{\circ}{K}_{14, 25}+\stackrel{\circ}{K}_{25,34}+\stackrel{\circ}{K}_{14}+\stackrel{\circ}{K}_{34},\; f_{25}^{2} \to\stackrel{\circ}{K}_{25,34}+\stackrel{\circ}{K}_{25,13}+\stackrel{\circ}{K}_{34}+\stackrel{\circ}{K}_{13},
\]
\[ f_{34}^{1} \to \stackrel{\circ}{K}_{15, 34}+\stackrel{\circ}{K}_{12,34}+\stackrel{\circ}{K}_{15}+\stackrel{\circ}{K}_{12},\; f_{34}^{2} \to \stackrel{\circ}{K}_{12,34}+\stackrel{\circ}{K}_{25,34}+\stackrel{\circ}{K}_{12}+\stackrel{\circ}{K}_{25},
\]
\[ 
f_{35}^{1} \to\stackrel{\circ}{K}_{14,35}+\stackrel{\circ}{K}_{12,35}+\stackrel{\circ}{K}_{14}+\stackrel{\circ}{K}_{12},\; f_{35}^{2} \to \stackrel{\circ}{K}_{12,35}+\stackrel{\circ}{K}_{24,35}+\stackrel{\circ}{K}_{12}+\stackrel{\circ}{K}_{24},
\] 
\[f_{45}^{1} \to \stackrel{\circ}{K}_{23,45}+\stackrel{\circ}{K}_{12,45}+\stackrel{\circ}{K}_{23}+\stackrel{\circ}{K}_{12},\; f_{45}^{2} \to \stackrel{\circ}{K}_{12,45}+\stackrel{\circ}{K}_{13,45}+\stackrel{\circ}{K}_{12}+\stackrel{\circ}{K}_{13}.
\] 
The generators for $H_{4}(L_2)$ are by the proof of Lemma~\ref{L2} given by $g_{ij, kl} =\stackrel{\circ}{K}_{ij,kl}+\stackrel{\circ}{K}_{ij}(7)+\stackrel{\circ}{K}_{kl}(7)$. In terms of these generators the above boundaries can be written  as follows:
\[
 f_{12}^{1}\to g_{12,34}+g_{12,45},\; f_{12}^{2}\to g_{12,45}+g_{12,35}, \;
f_{13}^{1} \to g_{13,24}+g_{13,45},\; f_{13}^{2} \to g_{13,45}+g_{13,25},
\]
\[ 
f_{14}^{1} \to g_{14,23}+g_{14,35},\; f_{14}^{2} \to g_{14,35}+g_{14,25},\;
f_{15}^{1} \to g_{15,23}+g_{15,34},\; f_{15}^{2} \to g_{15,34}+g_{15,24},
\] 
\[
f_{23}^{1} \to g_{15,23}+g_{23,45},\;  f_{23}^{2} \to g_{23,45}+g_{23,14},\;
f_{24}^{1} \to g_{15,24}+g_{24,35},\; f_{24}^{2} \to g_{24,35}+g_{13,24},
\] 
\[f_{25}^{1} \to g_{14,25}+g_{25,34},\; f_{25}^{2} \to g_{25,34}+g_{25,13},\;
 f_{34}^{1} \to g_{15,34}+g_{12,34},\; f_{34}^{2} \to g_{12,34}+g_{25,34},
\]
\[ 
f_{35}^{1} \to g_{14,35}+g_{12,35},\; f_{35}^{2} \to g_{12,35}+g_{24,35},\;
f_{45}^{1} \to g_{23,45}+g_{12,45},\; f_{45}^{2} \to g_{12,45}+g_{13,45}.
\] 
Now, searching  for  five-dimensional cycles in $V_{21}$ we  consider the system: 
\begin{equation}\label{sys}
\partial _{4}^{V_{21}}(\sum \limits _{1\leq i<j\leq 5}(\alpha _{ij}^{1}f_{ij}^{1}+ \alpha _{ij}^{2}f_{ij}^{2})) = \sum \limits _{1\leq i<\leq 5}\alpha _{ij}^{1}\partial _{4}^{V_{21}}f_{ij}^{1}+ \sum\limits _{1\leq i<j\leq 5}\alpha _{ij}^{2}\partial _{4}^{V_{21}}f_{ij}^{2} = 0.
\end{equation}
Substituting the  above expressions for the boundaries in the system~\eqref{sys}  we obtain:
\[
\alpha _{12}^{1}+\alpha _{34}^{1}+\alpha_{34}^{2}=0, \; \alpha _{12}^{1}+\alpha _{12}^{2}+\alpha _{45}^{1}+\alpha _{45}^{2}=0, \; \alpha _{12}^{2}+\alpha _{35}^{1}+\alpha _{35}^{2}=0, \; \alpha _{13}^{1}+\alpha _{24}^{2}=0,
\]
\[  \alpha _{13}^{1}+\alpha _{13}^{2}+\alpha _{45}^{2}=0,\;
\alpha _{13}^{2}+\alpha _{25}^{2}=0, \; \alpha _{14}^{1}+\alpha _{23}^{2}=0, \; \alpha_{14}^{1}+\alpha _{14}^{2}+\alpha _{35}^{1}=0, 
\]
\[
 \alpha _{14}^{2}+\alpha _{25}^{1}=0, \; \alpha _{15}^{1}+\alpha _{23}^{1}=0,\; \alpha _{15}^{1}+\alpha _{15}^{2}+\alpha _{34}^{1}=0, \; \alpha _{15}^{2}+\alpha _{24}^{1}=0, 
\]
\[ \alpha _{23}^{1}+\alpha _{23}^{2}+\alpha _{45}^{1}=0, \;  \alpha _{24}^{1}+\alpha _{24}^{2}+\alpha _{35}^{2}=0, \; \alpha _{25}^{1}+\alpha _{25}^{2}+\alpha _{34}^{2}=0.
\]
 
The solution of this system is given by 
\[
\alpha _{45}^{1}=\alpha _{13}^{1}+\alpha _{13}^{2}, \; \alpha _{45}^{2}= \alpha_{12}^{1}+\alpha _{12}^{2}+\alpha _{13}^{1}+\alpha _{13}^{2}, \;    \; \alpha _{35}^{2}=\alpha _{12}^{2}+\alpha _{14}^{1}+\alpha _{14}^{2}, \;\alpha _{35}^{1} = \alpha _{14}^{1}+\alpha _{14}^{2}, 
\]
\[ \alpha _{34}^{2} = \alpha _{13}^{2}+\alpha _{14}^{2}, \; \alpha _{34}^{1}=\alpha _{12}^{1}+\alpha _{13}^{2}+\alpha _{14}^{2},\;  
\alpha _{25}^{2}= \alpha _{13}^{2}, \; \alpha _{25}^{1}=  \alpha _{14}^{2}, \; \alpha _{24}^{2}=\alpha _{13}^{1},
\]
\[ \alpha _{24}^{1}=\alpha _{12}^{2}+\alpha _{13}^{1}+\alpha _{14}^{1}+\alpha _{14}^{2},\;  
 \alpha _{23}^{2}= \alpha _{14}^{1}, \; \alpha _{23}^{1} = \alpha _{12}^{1}+\alpha _{12}^{2}+ \alpha _{13}^{1}+\alpha _{13}^{2}+\alpha _{14}^{1}, 
\]
\[
 \alpha _{15}^{2} = \alpha _{12}^{2}+\alpha _{13}^{1}+\alpha_{14}^{1}+\alpha _{14}^{2},\;
\alpha _{15}^{1} = \alpha _{12}^{1}+\alpha _{12}^{2}+\alpha _{13}^{1}+\alpha _{13}^{2}+\alpha _{14}^{1},
\]
where  $\alpha _{12}^{1}, \alpha _{12}^{2},  \alpha _{13}^{1},\alpha _{13}^{2},  \alpha _{14}^{1}, \alpha _{14}^{2}$  
are  free variables. 
Thus, the  dimension of the solution space  of this system is $6$, which implies that   $H_{5}(V_{21};\Z _{2})=\Z _{2}^6$.  It further implies that the image of the map $ H_{5}(V_{21}/L_2;\Z_{2})\to H_{4}(L_2,\;Z_{2})$ is $\Z _{2} ^{14}$.   Thus, if we now consider the exact homology sequence $ H_{5}(V_{21}/L_2; \Z_{2}) \to H_{4}(L_2;\Z_{2}) = \Z _{2}^{15}\to H_{4}(V_{21},\Z_{2})\to 0=H_{4}(V_{21}/L_2, \Z_{2})$,   we deduce that $H_{4}(V_{21}, \Z_{2})=\Z _{2}$.  Moreover, as a  generator for $H_{4}(V_{21}, \Z _{2})$   
can be taken  any of  generators $g_{ij, kl}$ for $H_{4}(L_2)$.

\end{proof}

\begin{prop}\label{V2}
The  non-trivial homology groups for $V_{2}$  with coefficients in $\Z _{2}$ are $H_{6}(V_2;\Z_{2}) = \Z _{2}^{6}$, $H_{5}(V_2;\Z_{2}) =\Z _{2}^{7}$  and $H_{0}(V_2;\Z_{2})=\Z_{2}$.
\end{prop}

\begin{proof}
Combining the exact homology sequence of the pair $(V_2,V_1)$ and  Lemma~\ref{V12}  we obtain  that $H_{i}(V_2;\Z_{2})=0$ for $i\geq 7$. Moreover, $V_2$ has no cells of the dimension greater or equal then $7$. The six-dimensional cycles  in $V_2$ comes from $V_{21}$ since $V_1$ has no six-dimensional cells. These cycles are, by Remark~\ref{cycles}, given by 
$e_{ij} = \stackrel{\circ}{K_{ij}}(9)\times (S^{2}\setminus \{\ast\})$ ,  where $1\leq i <j\leq 5$. Then $\partial _{5}^{V_2}(e_{ij}) =  \partial _{5}^{V_{21}}(e_{ij}) +  \partial _{5}^{V_1}(e_{ij}) = 0+  \partial _{5}^{V_1}(e_{ij})$. The five-dimensional cycles  in $V_1$ are cycles  which correspond to the five-spheres over the octahedra on $\partial \Delta _{5,2}$.  They represent a  basis for $H_{5}(V_{1};\Z_{2})$  and we denote them by $S_{i}$, $1\leq i\leq 5$.  It follows from Proposition~\ref{bound-9} that the boundary of $K_{ij}(9)$ contains exactly the octahedra $O_i$ and $O_j$. 
Note that, since we work with $\Z _{2}$-coefficients,  we do not need to take care about an  orientation. We  obtain that $\partial _{5}^{V_1}(e_{ij})$ are as follows:
\[
e_{12} \to  S_{1}+S_{2},\; e_{13} \to S_{1}+S_{3}, \;  e_{14}\to  S_{1}+S_{4}, \;  e_{15}\to S_{1}+S_{5},
\]
\[
e_{23}\to  S_{2}+S_{3},\; e_{24}\to  S_{2}+S_{4},\; e_{25}\to  S_{2}+S_{5},\;  
\]
\[
e_{34}\to  S_{3}+S_{4},\; e_{35}\to  S_{3}+S_{5},\; 
e_{45}\to  S_{4}+S_{5}.
\]  
We consider  homogeneous   system of five linear equations given by

\begin{equation}\label{cycle5}
\sum _{1\leq i<j\leq 5}\alpha _{ij}\partial _{6}e_{ij} = \sum _{1\leq i<j\leq 5}\alpha _{ij}(S_{i}+S_{j}) = 0,   
\end{equation}

and it can be directly  checked  that its solution space has the dimension $4$.  More precisely, the space of solutions of  this system is given by
\[
\alpha _{15} = \alpha _{12}+\alpha _{13}+\alpha _{14}, \;\; \alpha _{25} = \alpha _{12}+\alpha _{23}+\alpha _{24}, 
\]
\[\alpha _{35}= \alpha _{13}+\alpha_{23}+\alpha _{34}, \;\; \alpha _{45}=\alpha _{14}+\alpha _{24}+\alpha _{34}.
\]

where $\alpha _{12}, \alpha _{13}, \alpha _{14}, \alpha _{23}, \alpha _{24}$ and $\alpha _{34}$ are  free $\Z _{2}$-variables.

 It implies that $H_{6}(V_2;\Z_{2})\cong \Z_{2}^{6}$ and that the generators for $H_{6}(V_2;\Z_{2})$ are given by the cycles
\begin{equation}\label{gen}
c_{12}=e_{12}+e_{15}+e_{25}, \;\;  c_{13}=e_{13}+e_{15}+e_{35},\;\;  c_{14}=e_{14}+e_{15}+e_{45}, 
\end{equation}
\[
c_{23}=e_{23}+e_{25}+e_{35}, \;\;c_{24}=e_{24}+e_{25}+e_{45},\;\; c_{34}=e_{34}+e_{35}+e_{45}.
\]

Therefore, the quotient $H_{5}(V_1;\Z _{2})/\Im \partial _{5}(H_{6}(V_{21};\Z _{2})$ is isomorphic to  $\Z _{2}$.

Now the exact homology sequence $H_{6}(V_{21};\Z_{2})\to H_{5}(V_1;\Z_{2}) \to H_{5}(V_{2};\Z_{2})\to H_{5}(V_{21};\Z_{2}) \cong \Z _{2}^6\to 0$ implies that $H_{5}(V_{2};\Z_{2}) = \Z _{2}^7$. 
Further, the map $\Z_{2}\cong  H_{4}(V_{21};\Z_2)\to H_{3}(V_{1};\Z_{2})\cong \Z _{2}$ is an isomorphism, since the generator for  $H_{3}(V_{1};\Z_{2})$ is given by $\partial \Delta _{5,2}$, while the generator for  $ H_{4}(V_{21};\Z_2)$ is given by, for example,   $g_{12, 34}$ whose boundary is $\partial \Delta _{5,2}$. Then the exact homology sequence implies that $H_{4}(V_2, \Z _{2})=0$. The fact that  $H_{i}(V_2;\Z_{2})=0$ for $i\neq 0, 4, 5,6$  follows also directly from the exact homology sequence.
\end{proof}

\begin{cor}
The space $V_{21}$ has no torsion in homology.
\end{cor}
\begin{proof}
By the  universal coefficient formula we have that  $H_{i}(V_{21})=0$, $1\leq i\leq 3$, $H_{i}(V_2)=0$, $1\leq i\leq 4$, $H_{5}(V_{21})\cong \Z ^{l}\oplus \Z _{2}^q$, where $l+q=6$ and $H_{6}(V_{21})=\Z ^{l}\oplus \Z _{2}^q$, where $l+q=10$.   Then an exact homotopy sequence of the pair $(V_2, V_1)$ implies $H_{4}(V_{21})\cong \Z$, while an exact homotopy sequence of the pair $(V_{21}, L_2)$ implies that $H_{5}(V_{21})\cong \Z ^{6}$ and $H_{6}(V_{21})\cong \Z ^{10}$.
\end{proof}

\begin{thm}
The non-trivial homology groups of the space $V_{2}$ are $H_{6}(V_2)\cong \Z ^{5}$, $H_{5}(V _{2})\cong \Z ^{6}\oplus \Z _{2}$ and $H_{0}(V_2)\cong \Z$.
\end{thm}
\begin{proof}
We  follow the proof of Proposition~\ref{V2}. Note that the boundary  $\partial _{5}^{V_1}(e_{ij})$ with integral coefficients is  the same as  with $\Z _{2}$-coefficients. It follows from the observation that this boundary is given by the  octahedra $O_i$ and $O_j$ of $K_{ij}$, while it is 
identity on the parameter space $S^{2}\setminus \{\ast\}$. Moreover, being a facet of $\Delta _{5,2}$, an orientation of the octahedra $O_{i}$ is given by an exterior normal for $\Delta _{5,2}$,  so it coincides with the induced orientation from $K_{ij}$  for any $K_{ij}$, $i\neq j$, since  $K_{ij} \subset \Delta _{5,2}$. Therefore,  the $6$-cycles in $V_2$ are given by  the solution of~\eqref{cycle5} with integer coefficients. The solution space of this system has dimension $5$ and it is given by 
\[
\alpha _{15} =- \alpha _{12}-\alpha _{13}-\alpha _{14}, \;\; \alpha _{25} = -\alpha _{12}-\alpha _{23}-\alpha _{24}, 
\]
\[\alpha _{35}= \alpha _{12}+\alpha_{14}+\alpha _{24}, \;\; \alpha _{45}=\alpha _{12}+\alpha _{13}+\alpha _{23}.
\]
\[
\alpha _{34}=-\alpha _{12}-\alpha_{13}-\alpha_{14}-\alpha_{23}-\alpha_{24},
\]
where $\alpha _{12}, \alpha_{13}, \alpha_{14}, \alpha_{23}, \alpha _{24}$ are free integer variables. Therefore, $H_{6}(V_2)\cong \Z ^{5}$ and  $\Im \partial _{5}(H_{6}(V_{21})$ is a span of $S_1+S_2, S_1+S_3, S_1+S_4, S_1+S_5, S_2+S_3$, 
which coincides with a  span of $2S_1, S_1+S_2, S_1+S_3, S_1+S_4, S_1+S_5$.  It implies that  $H_{5}(V_1)/\Im \partial _{5}(H_{6}(V_{21})\cong \Z _{2}$.  Now an exact homology sequence $H_{6}(V_{21})\to H_{5}(V_1) \to H_{5}(V_{2})\to H_{5}(V_{21}) \cong \Z^6\to 0$ implies that $H_{5}(V_{2}) = \Z^{6}\oplus \Z _{2}$, which completes the proof. 
\end{proof}

\subsection{The homology groups for $V_3 = G_{5,2}/T^5$.}
We denoted by $V_2$ the union of the orbit spaces of  the  strata  which are different from the main stratum.  As we already pointed,  $V_3/V_2$  is an one-point compactification of  $W/T^5$, which  implies 
\[
V_3/V_2 \cong  (\stackrel{\circ}{\Delta}_{5,2}\times F)^{\star} \cong (\Delta _{5,2} \times  U)/(U_1\cup U_2),
\]
\[
U = \C P^1\times \C P^1, \;\;  U_{1}= \partial \Delta _{5,2}\times \C P^1\times \C P^1,
\]
\[ U_2=(\Delta _{5,2}\times A\times \C P^1) \cup  (\Delta _{5,2}\times \C P^1\times A ) \cup (\Delta _{5,2}\times \Delta (\C P^1)).
\]
Note that $X_1= (\Delta _{5,2}\times U)/U_1\cong S^3\ast U$.  We denote by $X_2$ the projection of $U_2$ on $X_1$, that is  $X_2=U_2/(U_1\cap U_2) \cong S^{3}\ast U_2$ . We obtain that
\[
 X_1/X_2 \cong (S^3\ast U)/ (S^3\ast U_2) \cong (S^{3}\ast (U/U_2))/(S^{3}\ast \{\text{pt}\}).
\]
It follows that
\[
V_3/V_2\cong S^3\ast (U/U_2).
\]
\begin{prop}\label{V32}
The non-trivial homology groups for $V_3/V_2$ are as follows:
\[
H_{8}(V_3/V_2) = H_{0}(V_3/V_2)=\Z, \;\; H_{7}(V_3/V_2) =\Z ^{5}, \;\; H_{6}(V_3/V_2)=\Z ^{6}.
\]
\end{prop}
\begin{proof}
Since $\bar{H}_{i}(S^{3}\ast (U/U_2) )= H_{i-1}((S^3\times (U/U_2))/(S^{3}\vee (U/U_2)))$,  it follows  that  
\begin{equation}\label{ahom}
\bar{H}_{i}(S^3\ast (U/U_2)) = H_{i-1}(S^3\times (U/U_2))/(H_{i-1}(S^3)\oplus H_{i-1}(U/U_2)).
\end{equation}

 In order to compute the homology groups for $U/U_2$, we first compute the homology groups for $U_2$.  Let us set $U_2= U_3\cup U_4$ for  $U_3=(\Delta _{5,2}\times A\times \C P^1) \cup  (\Delta _{5,2}\times \C P^1\times A)$ and $U_4=\Delta _{5,2}\times \Delta (\C P^1))$.  The homology groups for $U_3$ we compute by applying the  Mayer-Vietoris sequence to the pair $(\C P^1\times A, A\times \C P^1)$. Since $(\C P^1\times A)\cap (A\times \C P^1)=A\times A$, it follows that  $H_{i}(U_3)=H_{i}(\C P^1\times A)\oplus H_{i}(A\times \C P^1)$ for $i\geq 2$. The exact sequence $0\to H_1(U_3)\to H_{0}(A\times A)\to H_{0}(A\times \C P^1)\oplus H_{0}(\C P^1\times A)\to H_{0}(U_3)$ gives the exact sequence  $0\to H_{1}(U_3)\to \Z ^{9}\to \Z ^3\oplus \Z ^{3}\to Z$,  which  implies that $H_{1}(U_3)\cong \Z ^{4}$. Thus, the nontrivial homology groups for $U_3$ are given by
\begin{equation}
H_{i}(U_3) = \left\{
\begin{array}{cc}
\Z^3\oplus \Z^{3}, &  i=2\\
\Z ^{4}, &  i=1\\
\Z ,  &  i=0.
\end{array}\right .
\end{equation}
Now we apply the  Mayer-Vietoris sequence to the pair $(U_3,U_4)$. Since $U_3\cap U_4 = \Delta (A)$,  it follows $H_{i}(U_2)=H_{i}(U_3)\oplus H_{i}(U_4)$ for $i\geq 2$. We further obtain  the exact sequence $0\to H_{1}(U_3)\oplus H_{1}(U_4)\to  H_{1}(U_2)\to H_{0}(\Delta (A))\to H_{0}(U_3)\oplus H_{0}(U_4) \to H_{0}(U_2)$ which gives the exact sequence  $0\to \Z ^4\to H_{1}(U_2)\to \Z ^{3}\to \Z\oplus \Z\to \Z$. It follows that $H_{1}(U_2)=\Z^6$.  We obtain the nontrivial homology groups for $U_2$:
\begin{equation}
H_{i}(U_2) = \left\{
\begin{array}{cc}
\Z^3\oplus \Z^{3}\oplus \Z, &  i=2,\\
\Z ^{6}, &  i=1,\\
\Z ,  &  i=0.
\end{array}\right .
\end{equation} 
We consider now the exact homology sequence of the pair $(U, U_2)$. It immediately gives that $H_{4}(U/U_2)=\Z$. We further obtain the exact sequence $0\to H_{3}(U/U_2)\to H_{2}(U_2)\to H_{2}(U)\to H_{2}(U/U_2)\to H_{1}(U_2)\to 0$, which gives the exact sequence  $0\to H_{3}(U/U_2)\to \Z ^{7}\to \Z ^{2}\to H_{2}(U/U_2)\to \Z ^6\to 0$. Note that the map $H_{2}(U_2)\to H_{2}(U)$ is induced by the inclusion $U_2\to U$. Let us  consider a  cell decomposition for $\C P^1=S^2$ which consists of the $2$ two-dimensional cells $D_1,D_2$,  then the $3$ one-dimensional cells $I_1, I_2, I_3$  and the $3$ zero-dimensional cells given by the points from $A$. Then the seven generators for $H_{2}(U_2)$ are given by $(D_1\cup D_2)\times A$, $A\times (D_1\cup D_{2})$ and $\Delta (D_1\cup D_2)$, while the generators for $H_{2}(U)$ are given by $(D_1\cup D_2)\times \{\ast\}$ and  $\{\ast\} \times (D_1\cup D_{2})$, where $\{\ast \}$ is a fixed point from $A$. It implies that the map $H_{2}(U_2)\to H_{2}(U)$ is an  epimorphism which   implies   that $H_{3}(U/U_2)=\Z ^{5}$. The exact sequence $0\to H_{2}(U/U_2)\to \Z ^6\to 0$ implies that $H_{2}(U/U_2)=\Z ^6$. Since all three spaces $U$, $U_2$, $U/U_2$ are connected we obtain that $H_{1}(U/U_2)=0$ and $H_{0}(U/U_2)=\Z$. Altogether the nontrivial homology groups for $U/U_2$ are:
\begin{equation}
H_{i}(U/U_2) = \left\{
\begin{array}{cc}
\Z, &  i=4,\\
\Z ^{5}, &  i=3,\\
\Z ^{6},  &  i=2,\\
\Z,     & i=0.
\end{array}\right .
\end{equation}
Together with~\eqref{ahom} this proves the statement. 
\end{proof}
\begin{proof}
We provide also the proof of the previous statement which is  more geometrical.
Let 
\[
X= \Delta _{5,2}  \times \C P^1\times \C P^1
\]
and
\[
Y= (\partial \Delta _{5,2}\times \C P^1 \times \C P^1) \cup (\Delta _{5,2}\times A \times \C P^1) \cup  (\Delta_{5,2}\times \C P^1 \times A).
\]
Then
\[
X/Y \cong ( \stackrel{\circ}{\Delta}_{5,2}\times ((S^2\vee S^1\vee S^1)\setminus \{\ast\})\times ((S^2\vee S^1\vee S^1)\setminus \{\ast\}))\cup \{\ast\}.
\]
It follows that 
\[
X/Y \cong S^8\vee (\vee _{4}S^7)\vee (\vee_{4}S^6).
\]
Therefore,
\begin{equation}\label{XY}
H_{8}(X/Y) = \Z, \; H_{7}(X/Y) =\Z ^4, \; H_{6}(X/Y) = \Z ^{4},\;  H_{0}(X/Y)= \Z.
\end{equation}
Let $Z = \Delta_{5,2}\times \Delta (\C P^1\times \C P^1)$. Then  $Z\cap Y =  \Delta_{5,2}\times \Delta (A)$.  It follows that
\[
Z/(Z\cap Y) =  (\stackrel{\circ}{\Delta}_{5,2}\times \Delta (\C P^1_{A}))\cup \{\ast\},
\]
that is
\[
Z/(Z\cap Y) \cong  (\stackrel{\circ}{\Delta}_{5,2}\times ((S^2\vee S^1\vee S^1)\setminus \{\ast\}))\cup \{\ast\}.
\]
Therefore, 
\[
H_{6}(Z/(Z\cap Y)) =\Z, \;  H_{5}(Z/(Z\cap Y)) =\Z^{2}, \; H_{0}(Z/(Z\cap Y)) =\Z.
\]
Set  $C= Z/(Z\cap Y)$. 
Then $V_3/V_2 \cong (X/Y)/C$.  Using this we compute the homology groups for $V_3/V_2$.
The exact homology sequence of the pair $(X/Y, C)$ immediately gives that 
\[
H_{8}(X/U)=\Z,\; H_{0}(X/U) =0, \;  H_{i}(X/U)=0, \; i=1,2,3,4,5.
\]
We need to compute $ H_{7}((X/Y)/C)$ and  $ H_{6}((X/Y)/C)$. Let us consider the exact sequence:
\begin{equation}\label{eseq}
0=H_{7}(C)\to \Z ^{4}=H_{7}(X/Y) \to H_{7}((X/Y)/C)\to \Z = H_{6}(C)\to \Z ^{4}=H_{6}(X/Y)
\end{equation}
\[
\to H_{6}((X/Y)/C)\to \Z ^{2}=H_{5}(C)\to 0=H_{5}(X/Y).
\]

The map $\Z ^{4}=H_{7}(X/Y)\to H_{7}((X/Y)/C)$ is a  monomorphism.  The map $\Z = H_{6}(C)\to \Z ^{4}=H_{6}(X/Y)$ is induced by the   diagonal embedding  $S^2\to ((S^2\setminus \{\ast\})\times (S^2\setminus \{\ast\}))\cup \{\ast\} \cong S^4$, so it is trivial.  It follows from~\eqref{eseq} that 
\[
H_{7}((X/Y)/C) = \Z ^5.
\]
It further implies that the map $\Z ^{4}=H_{6}(X/Y)\to H_{6}((X/Y)/C)$  is a monomorphism. Since the map  $H_{6}((X/Y)/C)\to \Z ^{2}=H_{5}(C)$ is an epimorphism it follows that
\begin{equation}\label{C2}
H_{6}((X/Y)/C) = \Z ^{6}.
\end{equation}
\end{proof}

The universal coefficient theorem implies:
\begin{cor}
The nontrivial homology groups for $V_3/V_2$ with $\Z_{2}$-coefficients  are as follows:
\begin{equation}
H_{8}(V_3/V_2;\Z_{2})\cong \Z_{2}, \; H_{7}(V_3/V_2;\Z _{2})\cong \Z _{2} ^5, \; H_{6}(V_3/V_2;\Z_{2})\cong  \Z _{2} ^{6}, \; H_{0}(V_3/V_2;\Z_{2})\cong \Z _{2}.
\end{equation}
\end{cor}
Combining  Proposition~\ref{V2} and Lemma~\ref{V32} we obtain the homology groups for $G_{5,2}/T^5$ with $\Z_{2}$-coefficients.
\begin{thm}\label{V3Z2}
The non-trivial homology groups for $V_3= G_{5,2}/T^5$  with $\Z _{2}$-coefficients are
\begin{equation}
H_{0}(V_3;\Z_{2}) =H_{5}(V_3, \Z _{2})=H_{6}(V_3; \Z _2) = H_{8}(V_3;\Z _{2})\cong  \Z _{2}.
\end{equation}
\end{thm}
\begin{proof}
The exact homology sequence of the pair $(V_3,V_2)$ directly implies  that $H_{8}(V_3;\Z_{2})\cong H_{0}(V_3;\Z_{2})\cong \Z_{2}$ and $H_{i}(V_3;\Z _{2})\cong 0$ for $i\neq 5,6,7$. 

We analyze the  map $\Z _{2}^{5}\cong H_{7}(V_{32};\Z _{2})\to H_{6}(V_2;\Z_2)\cong \Z _{2}^{6}$. The generator from $ H_{7}(V_{32})$ which maps in~\eqref{eseq} to the generator in $H_{6}(C)$, maps to an element $c_{\Delta}$  in $H_{6}(V_2)$ that  contains $e_{45}$.    Recall that the  other generators in $H_{7}(V_{32})$ are  by the proof of Proposition~\ref{V32}   given by   $m_{l}^{1} = \stackrel{\circ}{\Delta} _{5,2}\times (S^2\backslash \{\ast\})\times (S_{l}^{1}\backslash \{\ast\})$,   $m_{l}^{2} = \stackrel{\circ}{\Delta} _{5,2}\times  (S_{l}^{1}\backslash \{\ast \})\times (S^{2}\setminus \{\ast\})$, $l=1,2$. These are cells of the orbit space of the  main stratum and their six-dimensional boundaries  consist of  the $6$-cells of the other strata  whose set of parameters is in the boundary of the set of parameters  of the $3$-cells $(S^2\backslash \{\ast\})\times (S_{l}^{1}\backslash \{\ast\})$  and $(S_{l}^{1}\backslash \{\ast \})\times (S^{2}\setminus \{\ast\})$ in 
$\tilde{\mathcal{F}}$, $ l=1,2$. More precisely, since we want to consider the six-boundaries of these cells in $V_3$ we write  them in the form
\[
m_{l}^{1} = \stackrel{\circ}{\Delta}_{5,2} \times F_{l}^{1}, \;\; 
m_{l}^{2} = \stackrel{\circ}{\Delta}_{5,2}\times F_{l}^{2}, \; l=1,2,
\]
where
\[ F_{l}^{1} = \{((c_1: c_1^{'}),(\phi _{2l}:\phi_{2l}^{'}), (c_3:c_3^{'})), \; c_1\phi _{2}^{'}c_3=c_1^{'}\phi _{2}c_{3}^{'}\},
\]
\[
F_{l}^{2} =  \{((\phi _{1l}:\phi_{1l}^{'}), (c_2:c_2^{'}),  (c_3:c_3^{'})), \; \phi _{1}c_{2}^{'}c_3=\phi _{1}^{'}c_2c_{3}^{'}\},
\]
and  $(\phi _{i1}:\phi_{i1}^{'})$ is a path in $\C P^{1}_{A}$  which connects $(0:1)$ and $(1:1)$, while   $(\phi _{i2}:\phi_{i2}^{'})$ is a path in $\C P^{1}_{A}$ which connects $(1:1)$ and $(1:0)$ for $i=1,2$.

We compute   the six-dimensional boundaries  of these cells in  $V_3$. Since we work with  $\Z _{2}$-coefficients we do not need to take care  about an orientation issue.
Recall that the cells in $V_{2}$ of the form $\stackrel{\circ}{K}_{ij}(9) \times (S^2\setminus \{\ast\})$ we denoted  by $e_{ij}$.  Thus, as the six-boundaries of the cells $m_{l}^{k}$, we obtain:  
\[
m_{1}^{1} \to e_{15}+e_{35}+e_{23}+e_{12}, \;\; m_{2}^{1}\to e_{35}+e_{12}+e_{25}+e_{13},
\]
\[
m_{1}^{2}\to e_{14}+e_{34}+e_{23}+e_{12}, \;\; m_{2}^{2}\to e_{12}+e_{13}+e_{24}+e_{34}.
\]

Following   notation from  the proof of Proposition~\ref{V2} these boundaries can be written as
\begin{equation}\label{boundV361}
m_{1}^{1}\to c_{12}+c_{23}, \;\;  m_{2}^{1}\to c_{12}+c_{13}, 
\end{equation}
\begin{equation}\label{boundV362}
m_{1}^{2}\to c_{12}+c_{14}+c_{23}+c_{24}, \;\; m_{2}^{2}\to c_{12}+c_{13}+ c_{24}+c_{34}.
\end{equation}
Comparing this with~\eqref{gen} we see that these elements together with $c_{\Delta}$  give  five  of the  six  generators in $H_{6}(V_2)$. The exact homology sequence implies that $H_{7}(V_3;\Z_2) = 0$.

The six-dimensional cycles in $V_3$ comes either from $V_2$ either  from $V_3/V_2$. We have just proved that  five  of the six non-cohomologous $6$-cycles in $V_2$ are eliminated by the $7$-cells in $V_3$, which  implies that  one   cycle for $V_2$   survives after embedding in  $H_{6}(V_3)$. The four generators in $H_{6}(V_3/V_2)$ correspond to the six-cells in $V_{3}/V_2$   of the form $p_{kl} = \stackrel{\circ}{\Delta}_{5,2}\times S^{1}_{l}\times S^{1}_{k}$, $k, l=1,2$. Their boundaries in $V_3$ belong to $V_2$ and they map trivially to  $H_{5}(C)$,  after restriction to  $C$. In order to describe  the boundaries of $p_{kl}$ explicitly we consider them in the form $p_{kl}=\stackrel{\circ}{\Delta}_{5,2}\times p_{l}^{k}$, where $p_{l}^{k}=\{((\phi _{1l}:\phi _{1l}^{'}), (\phi _{2k}:\phi _{2k}^{'}), (\phi :\phi ^{'}))\in (\C P^{1}_{A})^3,  \;  \phi _{1l}\phi_{2k}^{'}\phi = \phi _{1l}^{'}\phi_{2k}\phi^{'}\}$ and $(\phi_{ij}:\phi _{ij}^{'})$ are as above. Therefore, following  notation from the proof of Proposition~\ref{V12},  we obtain that these generators map to the following five-cycles in $V_2$:
\[
\stackrel{\circ}{\Delta}_{5,2}\times S_{1}^1\times S_{1}^1 \to f_{14}^{1}+f_{23}^{1}+f_{23}^{2}+f_{15}^{1}+f_{34}^{1}+f_{35}^{1}+f_{12}^{1}+f_{12}^{2},
\]
\[
\stackrel{\circ}{\Delta}_{5,2}\times S_{1}^1\times S_{2}^{1} \to f_{14}^{2}+f_{25}^{1}+f_{34}^{2}+f_{35}^{1}+f_{12}^{1}+f_{12}^{2},
\]
\[
\stackrel{\circ}{\Delta}_{5,2}\times S_{2}^1\times S_{2}^1 \to f_{34}^{2}+f_{35}^{2}+f_{12}^{1}+f_{12}^{2}+f_{13}^{1}+f_{13}^{2}+f_{24}^{2}+f_{25}^{2},
\]
\[
\stackrel{\circ}{\Delta}_{5,2}\times S_{2}^1\times S_{1}^{1} \to f_{15}^{2}+f_{35}^{2}+f_{34}^{1}+f_{12}^{1}+f_{12}^{2}+f_{24}^{1}.
\]

These cycles represent  four generators in $H_{5}(V_2;\Z _2) = \Z _{2}^{7}$.  The other two generators in $H_{6}(V_3/V_2)$ are by~\eqref{eseq}  the generators whose boundaries give the generators for $H_{5}(C)$.  Let   $c_1$ and $c_2$  be  cycles in $V_2$  obtained  as the  boundaries of these  generators for  $H_{6}(V_3/V_2)= H_{6}((X/Y)/C)$, which   after restriction  to $C$ give the generators for $H_{5}(C)$, according to~\eqref{XY} and ~\eqref{C2}.   Since the generators for $H_{5}(C)$ are given by the cells $\stackrel{\circ}{\Delta}_{5,2}\times \Delta (S^{1}_{l}\times S^{1}_{l})$, $l=1,2$    the cycle $c_1$ must $f_{45}^{1}$ as a summand and the cycle $c_2$  must contain  $f_{45}^{2}$ as a summand, while $f_{45}^{1}+f_{45}^{2}$  can not be contained as a summand in both cycles.   Therefore, $c_1, c_2$ are linearly independent and they are  independent with the cycles obtained from $p_{kl}$, which  implies that   they eliminate two more generators in  $H_{5}(V_2;\Z _{2})$. We obtain  the exact homology sequence  
\[
0\to \Z _{2}\to H_{6}(V_3;\Z _{2})\to H_{6}(V_3/V_2;\Z _{2})=\Z _{2}^6 \to H_{5}(V_2;\Z _{2})=\Z _{2} ^{7},
\]
where the map $\Z _{2}^6\to \Z _{2}^{7}$ is a monomorphism. It implies that $H_{6}(V_3)\cong \Z _2$. We further obtain the exact sequence
\[
0\to  \Z_{2}\to H_{5}(V_3; \Z _{2})\to H_{5}(V_3/V_2;\Z _{2}) \cong 0,
\]
which implies $H_{5}(V_3; \Z _{2}) \cong \Z _{2}$.  
\end{proof}

\begin{thm}\label{V3Z}
The non-trivial integral homology of the orbit space $V_3= G_{5,2}/T^5$ are given by $H_{0}(V_3)=H_{8}(V_3)\cong \Z$ and $H_{5}(V_{3})\cong \Z _{2}$. 
\end{thm}
\begin{proof}
The universal coefficient formula immediately  implies that $H_{i}(V_3)=0$, $i\neq 0,5,6,8$ and that $H_{0}(V_3)=H_{8}(V_3)\cong \Z$. Moreover, it implies that for $H_{5}(V_3)$ and $H_{6}(V_3)$ two possibilities may occur, that is either $H_{5}(V_3)=H_{6}(V_3)\cong \Z$ either $H_{5}(V_3)\cong \Z _{2}$, $H_{6}(V_3)=0$. Now an exact homology sequence of the pair $(V_3, V_2)$ produces an exact sequence 
\[
\Z ^{6}\cong H_{6}(V_3/V_2)\to H_{5}(V_2)\cong \Z^{6}\oplus \Z _{2}\to H_{5}(V_{3}).
\]
Note that the torsion element in $H_{5}(V_2)$ is represented by a cell of the form $\stackrel{\circ}{O}_i\times (S^{2}_{+}\cup S^{2}_{-})$. On the other hand  it follows from the proof of Theorem~\ref{V3Z2} that the   six-dimensional cells in $V_3/V_2$ are either  of the form $\stackrel{\circ}{\Delta} _{5,2}\times S^{1}_{l}\times S^{1}_{k}$ either the restriction of their $5$-boundaries  to $C$ is of the form   $\stackrel{\circ}{\Delta}_{5,2}\times \Delta(S^1\times S^1)$. Therefore, the cells of the form  $\stackrel{\circ}{O}_i\times (S^{2}_{+}\cup S^{2}_{-})$ do not contribute in $5$-boundaries of the six-dimensional cells from $V_3/V_2$. It implies that an image of the homomorphism 
$H_{6}(V_3/V_2)\to H_{5}(V_2)$ does not contain the group $\Z _{2}$. It further implies that the group  $H_{5}(V_{3})$  has a torsion, that is $H_{5}(V_{3})\cong \Z _{2}$, which completes the proof.
\end{proof}

\begin{rem}
In the meantime motivated by  our work the same result on homology groups for the orbit space $G_{5,2}/T^5$  is obtained in~\cite{HS} using geometric invariant theory and in particular the result that  geometric invariant theory quotients for the standard linearisation of $(\C ^{*})^{5}$-action on $G_{5,2}$  are the  del Pezzo surfaces of degree $5$. We are very greatfull to Hendrik Suess for pointing out  to us his paper which helped us  to correct some mistakes in previous calculations.
\end{rem}

\begin{cor}
The orbit space $G_{5,2}/T^5$ is a  rational homology sphere $S^8$.
\end{cor}

\begin{cor}
The orbit space $G_{5,2}/T^5$ is homotopy equivalent to the space  obtained by attaching the disc $D^{8}$ to the  space $\Sigma ^{4}\R P^{2}$.
\end{cor}

\begin{proof}
It follows from Theorem~\ref{V3Z} that the six-skeleton for  $G_{5,2}/T^5$ is homotopy equivalent to the space obtained by attaching the disc $D^6$ to $S^5$ by a characteristic map $S^5\to S^5$ of degree $2$, so the resulting space is a $4$-suspension of $\R P^{2}$. Since $H_{7}(G_{5,2}/T^5)$ is trivial,  it follows that $G_{5,2}/T^5$ is obtained by attaching the disc $D^8$ to the space $\Sigma ^{4}\R P^2$.
\end{proof}

Let us consider now one of the five equivariant embeddings $G_{4,2}\to G_{5,2}$ and let  $X=(G_{5,2}/T^5)/(G_{4,2}/T^4)$. Since by~\cite{MMJ},  the space $G_{4,2}/T^4$ is homeomorphic to $S^5$, it follows that the space $X$ is homeomorphic to $(G_{5,2}/T^5)/S^5$. From  homology computations for $G_{5,2}/T^5$ and, in particular, the  proof of  Proposition~\ref{V2} we immediately obtain:

\begin{thm}
The nontrivial homology groups for $X$ are 
\begin{equation}
H_{0}(X) = H_{6}(X)=H_{8}(X)\cong \Z.
\end{equation}
\end{thm}

Moreover, it holds
\begin{cor}\label{joinvee}
The space $X$ is homotopy equivalent either to $S^3 \ast \C P^2$ either to $S^6\vee S^8$.
\end{cor}
\begin{proof}
 From our description of the cell structure of the orbit space   $G_{5,2}/T^5$ it follows that the $2$-skeleton  of a   quotient of $G_{5,2}/T^5$ by the sphere $S^3$ consists of one point, so the same holds for the space $X$.   The description of the  homology groups for $X$ and the Hurewitz theorem imply  that the space $X$ is $5$-connected, so it is homotopy equivalent to a $CW$-complex $C$ whose $6$-skeleton is the sphere $S^6$. Moreover, since $H_{7}(C)=0$ and $H_{8}(C)=\Z$,  it follows that $C$ is homotopy equivalent to a  $CW$-space which is obtained by attaching the disc $D^8$ to the sphere $S^6$ by a characteristic map $f: S^7\to S^6$. The map $f$ is either homotopy equivalent to the trivial map or it is homotopy equivalent to the $4$-suspension of the Hopf map $S^3\to S^2$. Therefore, in the first case $X$ is homotopy equivalent to the wedge $S^6\vee S^8$, while in the second case $X$ is homotopy equivalent to the  $4$-suspension of $\C P^2$. The two cell complexes $S^6\vee S^8$ and the  $4$-suspension of $\C P^2$ can be differentiated  by  the fact that in the first case the action of the Steenrod operation $S_{q}^{2}$ is trivial, while in the second case it is non-trivial.  
 \end{proof}

We determine an  exact homotopy type for $X$.

\begin{thm}\label{mainhomot}
The  space  $ X= (G_{5,2}/T^5)/S^5$ is homotopy equivalent to  $S^3\ast \C P^2$. 
\end{thm}
\begin{proof}
The spaces $X$ and $\partial \Delta _{5,2}\ast \C P^2$ are simply-connected CW-spaces and  have the same integral homology groups. We will prove that there exists a continuous map $h_{5} : X \to \partial \Delta _{5,2}\ast \C P^2$ which induces an  isomorphism in their homology groups and then, by the Whitehead theorem, it will follow that the map $h_5$ is homotopy equivalence.

Since the homology groups for $X$ are trivial in the dimensions less or equal then $5$, it follows that the $5$-skeleton $C$ of this space is contractible which  implies that $X/C$ is homotopy equivalent to $X$.  From the description of a cell decomposition of $G_{5,2}/T^5$,  it follows that the orbit spaces of the strata, whose admissible polytopes are different from  $\Delta _{5,2}$ and $K_{ij}(9)$,  contain only the cells of the dimension less or equal then $5$. Thus,   $X/C$ contains only  cells from  the orbit spaces of the strata  whose admissible polytopes are $\Delta _{5,2}$ or $K_{ij}(9)$.  Recall that $W/T^5\cong \stackrel{\circ}{\Delta }_{5,2}\times F\cong \stackrel{\circ}{\Delta}_{5,2}\times \tilde {F}_{12}$ and $W_{ij}(9)/T^5\cong \stackrel{\circ}{K}_{ij}(9) \times \C P^{1}_{A}\cong \stackrel{\circ}{K}_{ij}(9) \times \tilde{F}_{ ij, 12}$, where by $\tilde{F}_{12}$ and $\tilde{F}_{ij,12}$ we  denote, as previously,  the virtual spaces of parameters for these strata in the chart $M_{12}$. Note that  virtual spaces of parameters are the subspaces in $\tilde{\mathcal{F}}$. 

 Let $Z =   X \setminus C$ . It is defined the map
\[
h_{5}^{1} : Z \to  \Delta _{5,2}\times \tilde{\mathcal{F}},
\]
by the projection on $\Delta _{5,2}$ and by an embedding into $\tilde{\mathcal{F}}$. Note that the image of  $Z$ by $h_{5}^{1}$ belongs to   $\stackrel{\circ}{\Delta} _{5,2}\times \tilde{\mathcal{F}}$.  Moreover,  the virtual spaces of parameters $\tilde{F}_{12}$ and $\tilde{F}_{ ij,12}$ are open sets and their union  is not the whole  $\tilde{\mathcal{F}}$, which  implies that  the image of $Z$ is not the whole   $\tilde{\mathcal{F}}$.  Further,  it follows from Corollary~\ref{CP2}  that $\tilde{\mathcal{F}}$ is homeomorphic to a space $Y$ which is the  blowup of $\C P^{2}$ at four points. We consider the map
\[
h_{5}^{2} : \Delta _{5,2}\times \tilde{\mathcal{F}} \to \Delta _{5,2}\times \C P^2,
\]
defined by the blowdown of $\tilde{\mathcal{F}}$ at these four points. Next,  we consider the projection
\[
h_{5}^{3} : \Delta _{5,2}\times \C P^2 \to \partial \Delta _{5,2}\ast \C P^2.
\]
We obtain the map 
\begin{equation}\label{comp}
h_{5}^3\circ h_{5}^{2}\circ h_{5}^1 : Z \to \partial \Delta _{5,2}\ast \C P^2. 
\end{equation}
Denote by $D$ the  $5$-skeleton of the space $\partial \Delta _{5,2}\ast \C P^2\cong (\Delta _{5,2}\times \C P^2)/(\partial \Delta _{5,2}\times \C P^2)$.  Then $D$ is contractible and  $(\partial \Delta _{5,2}\ast \C P^2)/D$ is homotopy equivalent to $\partial \Delta _{5,2}\ast \C P^2$.  The space $D$ contains the cell $\stackrel{\circ}{\Delta} _{5,2}\times \{\ast\}$, but since $h_{5}^{1}(Z)$ is a strict subset of $\tilde{\mathcal{F}}$ we may always choose $\{\ast\} \in \C P^2$ such  that  $h_{5}^3\circ h_{5}^{2}\circ h_{5}^1(Y)\subset (\partial \Delta _{5,2}\ast \C P^2)\setminus D$.

In this way we obtain the map 
\[
h_{5} : X\cong X/C \to (\partial \Delta _{5,2}\ast \C P^2)/D \cong \partial \Delta _{5,2}\ast \C P^2,
\]
which is defined by $\eqref{comp}$ on $X\setminus C$ and $h_{5}(C)=D$. 

We prove that the map $h_{5}$ induces an isomorphism in integral homology groups. Since the spaces $X$ and  $S^3\ast \C P^2$  have the  nontrivial integral homology groups  in degrees $0,6,8$,  it is enough to prove that $h_5$ induces an  isomorphism of the homology groups in   degree $6$. From  calculations of the homology groups  for $G_{5,2}/T^5$ it follows that the  generator for $H_{6}(X)$  comes from $V_2$. More precisely, it follows from~\eqref{boundV361},~\eqref{boundV362} and~\eqref{gen} that for a generator in  $H_{6}(X)$ one can take the cycle $c_{12}= e_{12}+e_{15}+e_{25}$. Recall that  $e_{ij}$ are, by~\eqref{first} and  Remark~\ref{cycles},   the six-dimensional cycles in $K_{ij}^{*}(9)$ and they are of the form $\stackrel{\circ}{K}_{ij}(9)\times  \tilde{F}^{'}_{12,ij}$, where $\tilde{F}^{'}_{ij,12}\subset \tilde{F}_{ ij, 12}$ and   $\tilde{F}^{'}_{ij,12}\cong S^{2}\setminus\{\ast\}$. 
It follows from Theorem~\ref{bound-univ} that the virtual spaces of parameters $\tilde{F}_{24,12}$, $\tilde{F}_{25,12}$, $\tilde{F}_{ 23,12}$ and $\tilde{F}_{12, 12}$ after embedding  them into the space $Y$ map to the blowups of $\C P^2$ at four points. It implies that after blowing down  $Y$ at these four points these spaces map to the four points. The other spaces of parameters $\tilde{F}_{ij,12}$ are not affected by this blowing down. Consequently,  in the  six-dimensional homology the image of  the generator $c_{12}$ of $H_{6}(G_{5,2}/T^5)$ is given by the image of $e_{15}$. It holds  that $\tilde{F}_{15,12} = ((c_1:c_1^{'}), (0:1), (0:1))\subset \tilde{\mathcal{F}}$, which is   equal to $((c_1:c_1^{'}), (1:0))\in \C P^1\times \C P^1$ and this is further equal to $(c_1:0:c_1^{'})\in \C P^2$. It follows that the image of $e_{15}$ by the map  $h_{5}$ is $\stackrel{\circ}{\Delta _{5,2}}\times (c_1:0:c_1^{'})$, which  represents a generator for $H_{6}(\partial \Delta _{5,2}\ast \C P^{2})$. This proves the statement. 
\end{proof}

From  Theorem~\ref{Snk}  and Example~\ref{S52} it follows the result.
\begin{cor}
The space  $\Sigma (S_{5,2}/T^5)$  is homotopy equivalent to  $ S^3\ast \C P^2$.
\end{cor}

These results enable us  also  to describe the module structure of  $H^{*}(G_{5,2}/T^5; \Z _{2})$ over the mod $2$  Steenrod algebra as well.

\begin{cor}
The Steenrod square $Sq^{i}$  acts non-trivially on $H^{*}(G_{5,2}/T^5; \Z _{2})$   if and only if $i=1,2$.
\end{cor}

\begin{proof}
It follows from the universal coefficient formula that the cohomology groups for $G_{5,2}/T^5$ with $\Z _{2}$-coefficients coincide with its $\Z_2$-homology groups. It immediately implies that $Sq^{i}=0$  for $i\neq 1,2$. Moreover, being stable operation we have that $Sq^{1} : H_{5}(V_3;\Z _{2})\to H_{6}(V_3;\Z _{2})$ coincides with $Sq^{1} : H^{1}(\R P^{2};\Z _{2})\to H^{2}(\R P^2; \Z_{2})$. The later one is defined by squaring, so it is non-trivial.  In the same way $Sq^{2} : H^{6}(V_3; \Z _{2})\to H^{8}(V_3;\Z _{2})$ coincides with $Sq^{2} : H^{2}(\C P^2; \Z_{2})\to H^{4}(\C P^2;\Z_2)$ which is given by squaring, so it is non-trivial.
\end{proof}

Altogether this leads to the homotopy description of the orbit space $G_{5,2}/T^5$.
\begin{thm}
The orbit space $G_{5,2}/T^5$ is homotopy equivalent  to the  space $X$  obtained by attaching the disc $D^8$ to the  space $\Sigma ^{4}\R P^2$   by the generator of the group $\pi _{7}(\Sigma ^{4}\R P^2)$.  
\end{thm}

\begin{proof} 
Since by~\cite{JW} it holds $\pi _{7}(\Sigma ^{4}\R P^2)=\Z _{4}$ we need to show that $G_{5,2}/T^5$ can not be obtained by attaching the disc $D^8$ to the space $\Sigma ^{4}\R P^2$ by the map $f : S^7 \to \Sigma ^{4}\R P^2$   of order  $2$. In order to verify this   we first consider the exact homotopy sequence of the pair $(\Sigma ^{4}\R P^2, S^5)$. Note that, as a consequence of the Blakers-Massey theorem, we have that  $\pi _{i}(\Sigma ^{4}\R P^2) = \pi _{i}(\Sigma ^{4}\R P^2/S^5)$ for $i\leq 7$, since $\Sigma ^{4}\R P^2$ is $4$-connected and the inclusion $S^5\to \Sigma ^{4}\R P^2$ induces an isomorphism in homotopy groups up to degree $5$. Thus, since $(\Sigma ^{4}\R P^2)/S^5$ is homotopy equivalent to $S^6$, the corresponding homotopy sequence writes as 
\[
 \Z_ 2\to \pi _{7}(\Sigma ^{4}\R P^2)\cong \Z _{4}\to  \Z_{2}\to \Z_2\to \pi _{6}(\Sigma ^{4}\R P^2)\cong \Z _{2} \to  \Z,
\] 
which gives an exact sequence
\[
\Z _{2}\to \Z _{4}\to \Z_{2}\to 0.
\]
It follows that the element of order $2$ in $\pi _{7}(\Sigma ^{4} \R P^2)$ is coming from $\pi _{7}(S^5)$. 
Now, consider a map $f: S^7 \to \Sigma ^{4}\R P^2$. It induces the map $g : S^7\to S^6$ given by the composition
$S^{7}\stackrel{f}{\to}\Sigma ^{4}\R P^2 \to \Sigma ^{4}\R P^2/S^5 \cong S^6$. It follows from  Theorem~\ref{mainhomot} and Corollary~\ref{joinvee} that the map $g$ is a fourfold suspension of the Hopf map $S^3\to S^2$ and $G_{5,2}/T^5$ is homotopy equivalent to $S^3\ast \C P^2$.   If the attaching  map $f$ decomposes through a map $S^7\to S^5 \stackrel{i}{\to} \Sigma ^{4}\R P^2$, it would  imply that $(G_{5,2}/T^5)/S^5$ is  homotopy equivalent to $S^8\vee S^6$ which is not the case. 
\end{proof}

Let $\xi _{n,k}$ and $\eta _{n,n-k}$  be   the canonical vector bundles  over the Grassmannian $G_{n,k}$ and $X_{n,k} = T(\xi _{n-1,k})$, $Y_{n,k}= T(\eta _{n-1,n-k})$ - the corresponding Thom spaces equipped with the $T^n$-action. The problem of the  description of  the orbit space $G_{n,k}/T^n$ naturally brings  the problem of the  description of  the orbit spaces for these Thom spaces. 
In the case $n=5$ and $k=2$,    Corollary~\ref{orbit-Thom} and Corollary~\ref{orbX} imply:

\begin{thm}
The orbit spaces  $X_{5,2}/T^5$ is   homotopy equivalent to  $ D^8 \cup _{f} \Sigma ^{4}\R P^2$ for the generator  $f\in \pi _{7}(\Sigma ^{4}\R P^2)$ , while the orbit space   $Y_{5,2}/T^5$  is homotopy equivalent to  $S^3\ast \C P^2$. 
\end{thm}

\bibliographystyle{amsplain}

\begin{thebibliography}{10}



\bibitem{AT}
M.~F.~Atiyah, {\em Convexity and commuting Hamiltonians}, Bull.~London~Math.~Soc.~{\bf 14}, no.1, 1982, 1--15. 


\bibitem{AZ}
Anton A.~Ayzenberg, {\em Torus actions of complexity $1$ and their local properties}, Proc.~Steklov Inst.~Math.~{\bf 302}, no.1, (2018),  16--32.

\bibitem{AZ1}
Anton Ayzenberg, {\em Torus action on quaternionic projective plane and related spaces}, arXiv:1903.03460.

\bibitem{OB}
Victor~ M.~Buchstaber (joint with Svjetlana  Terzi\'c),  {\em (2n, k)-manifolds and applications},
Mathematisches Forschung Institut Oberwolfach, Report No. 27/2014, p. 5–8,
DOI: 10.4171/OWR/2014/27

\bibitem{BP1}
Victor M.~Buchstaber and Taras E.~Panov,  {\sl Torus Actions and Their Applications in Topology and Combinatorics},  University Lectures Series, Vol.~24, American Mathematical Society, 2002.


\bibitem{BP}
Victor M.~Buchstaber and Taras E.~Panov, {\sl Toric topology}, Mathematical Surveys and Monographs
Vol.~ 204, American Mathematical Society,  2015.

\bibitem{MMJ}
Victor M.~Buchstaber and Svjetlana Terzi\'c, {\em Topology and geometry of the canonical action of $T^4$ on the complex Grassmannian $G_{4,2}$ and the complex projective space $\C P ^{5}$},   Moscow Math.~Jour. Vol.~16, Issue 2,  (2016), 237--273.


\bibitem{MS}
V.~ M.~Buchstaber and S.~Terzi\'c, {\em The foundations of $(2n;k)$-manifolds}, Mat.~Sbornik, Vol.~210, no.4, (2019), 41--86, (in Russian).



\bibitem{Fulton}
William Fulton, {\sl Introduction to Toric Varieties}, Princeton University Press, 1993.

\bibitem{GUST}
V.~Guillemin and V.~Sternberg, {\em Convexity properties of the moment mapping}, Invent.~Math.~{\bf 67}, 1982, 491--513.

\bibitem{GFMC}
I.~M.~Gelfand and R.~MacPherson, {\em Geometry in Grassmannians and a generalization of the dilogarithm}, Adv.~in Math.~{\bf 44}, no.3, (1982), 279--312.

\bibitem{GS}
I.~M.~Gelfand and V.~V.~Serganova, {\em Combinatorial geometries and torus strata on homogeneous compact manifolds}, Russ.~Math.~Surveys~{\bf 42}, no.2, (1987), 133--168.

\bibitem{GEL4}
I.~M.~Gelfand, R.~M.~Goresky, R.~D.~MacPherson and V.~V.~Serganova, {\em Combinatorial Geometries, Convex Polyhedra, and Schubert Cells}, Adv.~in Math.~{\bf 63}, (1987), 301--316.

\bibitem{Hog}
S.~G.~Hoggar, {\em On KO theory of Grassmannians}, Quart.~J.~Math.~Oxford Ser.~(2) {\bf 20},  (1969), 447--463. 

\bibitem{Kap}
M.~M.~ Kapranov, {\em Chow quotients of Grassmannians I}, I.~M.~Gel'fand Seminar,  Adv.~in Soviet Math., 16, part 2, Amer.~Math.~Soc.~(1993), 29--110.

\bibitem{KT}
Yael Karshon and  Susan Tolman, {\em Classification of Hamiltonian torus actions with two-dimensional quotients}, Geom.~ Topol.~{\bf 18}, no.2,  (2014), 669–-716. 

\bibitem{KT-1}
Yael Karshon and  Susan Tolman, {\em Topology of complexity one quotients}, arXiv:1810.01026.


\bibitem{Keel}
S.~Keel, {\em Intersection theory of moduli space of stable  $n$-pointed curves of genus zero}, Trans.~Amer.~Math.~Soc,~{\bf 330}, no. 2, (1992), 545--574. 

\bibitem{KeelTev}
Sean Keel and Jenia Tevelev, {\em Geometry of Chow quotients of Grassmannians}, Duke Math.~ J.~{\bf 134}, no. 2 (2006), 259--311.

\bibitem{K}
Frances Clare Kirwan, {\sl Cohomology of Quotients in Symplectic and Algebraic Geometry}, Mathematical Notes, 31, 
Princeton University Press, 1984.

\bibitem{KM}
Nikita Klemyatin, {\em Universal spaces of parameters for complex Grassmann manifolds $G_{q+1,2}$}, arXiv:1905.03047

\bibitem{NO}
Masashi Noji, Kazuaki Ogiwara, {\em The smooth torus orbit closures in the Grassmannians}, arXiv:1812.11472.

\bibitem{T1}
D.~A.~Timashev, {\em Classification of G-manifolds of complexity 1}, (Russian) Izv.~Ross.~Akad.~Nauk Ser.~Mat.~{\bf 61}, no.2 (1997), 127--162, translation in Izv.~Math.~{\bf 61}, no.2 (1997),  363--397.

\bibitem{SS}
Hendrik S\"uss, {\em Orbit spaces of maximal torus actions on oriented Grassmannians of planes}, arXiv:1810.00981.

\bibitem{HS}
Hendrik S\"uss, {\em Toric topology of the Grassmannian of planes in $\C ^5$  and the del Pezzo surface of degree  $5$}, 
arXiv:1904.13301 

\bibitem{T2}
Dmitri Timashev, {\em Torus actions of complexity one},  Toric topology, 349-–364, Contemp.~ Math., {\bf 460}, Amer.~Math.~Soc., Providence, RI, 2008. 

\bibitem{Ziegler}
G.~Ziegler, {\sl Lectures on polytopes},  Graduate Texts in Mathematics, 152. Springer-Verlag, New York, 1995.

\bibitem{JW}
Jie Wu, {\sl Homotopy Theory of the Suspensions of the Projective Plane},  Memoirs of the AMS,~{\bf 162}, no. 769 (2003).

\end{thebibliography}

\end{document}